\documentclass[a4paper, 11pt]{article}

\usepackage[margin=2.5cm]{geometry}
\usepackage[utf8]{inputenc}
\usepackage{graphicx,xspace}
\usepackage{amssymb}
\usepackage{amsmath}
\usepackage{mathtools}
\usepackage[usenames,dvipsnames,svgnames,table]{xcolor}
\usepackage{enumerate}
\usepackage{enumitem}
\setlist[enumerate, 1]{label=(\roman*)}
\setlist[itemize]{leftmargin=1.5em}
\setlist[description]{leftmargin=1em}
\usepackage[normalem]{ulem} 

\usepackage{ifthen}

\usepackage{longtable}
\usepackage{multirow}
\usepackage{array}
\usepackage{capt-of}
\usepackage{arydshln}

\usepackage{pgf,tikz}
\usepackage{mathrsfs}
\usetikzlibrary{arrows}
\usepackage[outline]{contour}
\contourlength{0.6pt}

\usepackage{graphicx}
\usepackage{float}
\graphicspath{ {pics/} }

\usepackage[hidelinks]{hyperref}  
\usepackage[format=plain,labelsep=period,justification=justified,font=small,labelfont=bf, margin=1em]{caption}

\usepackage{amsthm}

\theoremstyle{plain}
\newtheorem{theorem}{Theorem}[section]
\newtheorem{corollary}[theorem]{Corollary}

\theoremstyle{definition}
\newtheorem{definition}{Definition}[section]

\theoremstyle{remark}
\newtheorem{remark}{Remark}[section]

\newcommand{\Z}{\mathbb{Z}}
\newcommand{\N}{\mathbb{N}}
\newcommand{\R}{\mathbb{R}}

\newcommand\varpm{\mathbin{\vcenter{\hbox{%
  \oalign{\hfil$\scriptstyle+$\hfil\cr
          \noalign{\kern-.3ex}
          $\scriptscriptstyle({-})$\cr}%
}}}}

\mathtoolsset{showonlyrefs,showmanualtags}

\newcommand{\bela}[1]{\begin{equation}\label{#1}}
\newcommand{\ela}{\end{equation}}
\newcommand{\bear}[1]{\begin{array}{#1}}
\newcommand{\ear}{\end{array}}

\newcommand{\bN}{\mbox{\boldmath $N$}}

\newcommand{\br}{\mbox{\boldmath $r$}}

\newcommand{\as}{\\[.6em]}

\definecolor{grey}{rgb}{0.5,0.5,0.5}

\newcommand{\ignore}[1]{}
\newcommand{\jac}[1]{\operatorname{#1}}

\newcommand{\blank}[1]{}

\newcommand{\ICe}{IC$_\text{e}$-net\xspace}
\newcommand{\ICh}{IC$_\text{h}$-net\xspace}

\newenvironment{acknowledgements}{
\paragraph{Acknowledgement.}
}

\title{On mutually diagonal nets on (confocal) quadrics and
3-dimensional webs}
\author{Arseniy V. Akopyan$^1$, Alexander I. Bobenko$^2$,  Wolfgang K. Schief$^3$ and Jan Techter$^2$\bigskip\\  
$^1$Institute of Science and Technology Austria (IST Austria),\\ Am Campus 1, A--3400
Klosterneuburg, Austria\bigskip\\
$^2$Institut f\"ur Mathematik, TU Berlin, \\ Str.\@ des 17.\@ Juni 136, 10623 Berlin, Germany\bigskip\\
$^3$School of Mathematics and Statistics,\\ The University of New South Wales, Sydney, NSW 2052, Australia}
\date{\today}

\begin{document}

\maketitle

\begin{abstract}
Canonical parametrisations of classical confocal coordinate systems are introduced and exploited to construct non-planar analogues of incircular (IC) nets on individual quadrics and systems of confocal quadrics. Intimate connections with classical deformations of quadrics which are isometric along asymptotic lines and circular cross-sections of quadrics are revealed. The existence of octahedral webs of surfaces of Blaschke type generated by asymptotic and characteristic lines which are diagonally related to lines of curvature is proven theoretically and established constructively. Appropriate samplings (grids) of these webs lead to three-dimensional extensions of non-planar IC nets. Three-dimensional octahedral grids composed of planes and spatially extending (checkerboard) IC-nets are shown to arise in connection with systems of confocal quadrics in Minkowski space. In this context, the Laguerre geometric notion of conical octahedral grids of planes is introduced. The latter generalise the octahedral grids derived from systems of confocal quadrics in Minkowski space. An explicit construction of conical octahedral grids is presented. The results are accompanied by various illustrations which are based on the explicit formulae provided by the theory. 
\end{abstract}




\section{Introduction}

Planar and spatial systems of confocal quadrics \cite{HCV52} have been the subject of extensive studies in both mathematics and physics since the discovery of their optical properties by the Ancient Greek. The fascination with confocal quadrics over the centuries culminated in the famous Ivory and Graves-Chasles theorems which were given a modern perspective by Arnold \cite{Arnold}. Their gravitational properties were studied by Newton and Ivory (see, e.g.,\ \cite{Arnold,FuchsTabachnikov}) and deep connections with dynamical systems and spectral theory have also been established \cite{Moser}. Confocal quadrics go hand in hand with confocal coordinate systems which have been employed, for instance, in the theory of separable linear differential equations such as the Laplace equation, which, in turn, is closely related to the theory of special functions such as (generalised) Lam\'e functions~\cite{Erdelyi}. 

The present paper is based on the following, apparently novel, property of planar confocal quadrics. A confocal system of conics on the plane is represented by
\bela{I1}
  \frac{x^2}{\lambda+a} + \frac{y^2}{\lambda+b} = 1,\quad a>b,
\ela
where the parameter $\lambda$ labels the individual conics. In \cite{BSST17}, the privileged parametrisation 
\bela{I2}
\left(\bear{c}x\\ y\ear\right) = \sqrt{a-c}\left(\bear{c}
     \jac{sn}(s,k)\jac{ns}(t,k)\\ 
     \jac{cn}(s,k)\jac{ds}(t,k)
  \ear\right),\quad k = \sqrt{\frac{a-b}{a-c}}
\ela
of the confocal conics outside an ellipse
\bela{I3}
 \frac{x^2}{a-c} + \frac{y^2}{b-c} = 1,\quad c<b
\ela
has been recorded in connection with a canonical discretisation of confocal quadrics. The two families of ellipses and hyperbolae of the confocal system are given by $t=\mbox{const}$ and $s=\mbox{const}$ respectively. As pointed out in \cite{BSST17}, the curves $s\pm t=\mbox{const}$ constitute straight lines so that the net formed by the two families of confocal conics and the net composed of the two families of straight lines are mutually diagonal in the sense of Blaschke \cite{B28} (cf.\ Definition \ref{mutual}). Now, as an implication of this fact, we conclude that one can construct pairs of grids of $\Z^2$ combinatorics by sampling the straight lines and confocal conics in such a manner that the confocal conics of the second grid pass through opposite vertices of the quadrilaterals formed by the straight lines of the first grid as depicted in Figure \ref{pic_sampling}. In \cite{BSST17}, we have shown that repeated application of Ivory's classical theorem \cite{IT} implies that the quadrilaterals circumscribe circles so that the grid of lines constitutes an incircular (IC) net. Such nets have been introduced in \cite{B} and generalised and discussed in great detail in \cite{AB}. In particular, it is known that the lines of an IC-net are tangent to a confocal conic which, in the current situation, turns out to be the ellipse \eqref{I3}. This connection between incircles and mutually diagonal nets of confocal conics and straight lines which are tangent to a confocal conic is consistent with the classical Graves-Chasles theorem~\cite{Darboux}.
\begin{figure}
  \begin{center}
    \scalebox{1.0}{\input{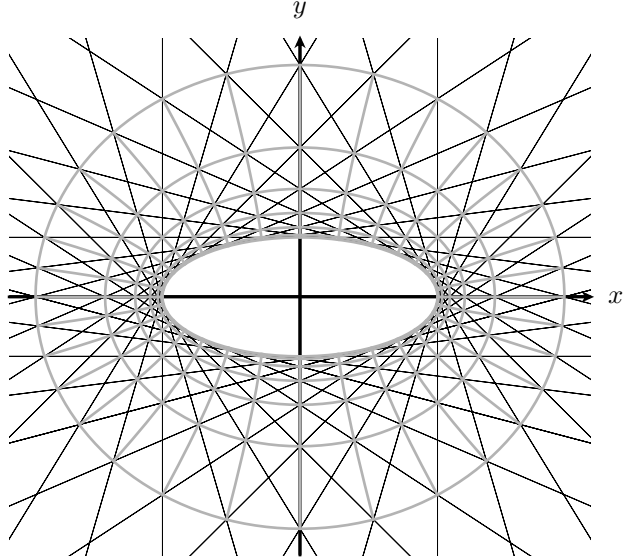}}
  \end{center}
  \caption{A sampling of mutually diagonal nets of straight lines and confocal conics which forms an IC-net.}
  \label{pic_sampling}
\end{figure}
  
Here, we address the natural question as to how one may generalise grids of IC type (and their counterparts in hyperbolic geometry) to two-dimensional grids of lines on surfaces and three-dimensional grids of lines/planes in space. It turns out that confocal systems of quadrics in space are custom-made for the construction of such spatial two- and three-dimensional grids. In fact, the novel representation of three-dimensional confocal coordinate systems presented in \cite{BSST17} which encapsulates, in particular, a three-dimensional analogue of the parametrisation \eqref{I2} is seen to provide not only compact algebraic proofs of the theorems stated in the main body of this paper but also explicit formulae which considerably simplify the visualisation of the results. Thus, in Section 2, we prove that there exists a unique parametrisation of systems of confocal quadrics such that the coordinate systems on all confocal quadrics are simultaneously isothermal-conjugate \cite{Eisenhart1903,Eisenhart1960}, that is, all second fundamental forms are conformally flat. As a corollary of this remarkable fact, we conclude in Section 5 that the two congruences of (straight) asymptotic lines on the family of one-sheeted hyperboloids of a system of confocal quadrics and the four congruences of characteristic lines on the two families of confocal ellipsoids and two-sheeted hyperboloids give rise to an octahedral web (``Achtflachgewebe'') in the sense of Blaschke \cite{B28b}. It turns out that any of the associated three-dimensional spatial grids generated by appropriately sampling the asymptotic and characteristic lines may be regarded as a spatial extension of an underlying IC-net.

The connection between IC-nets and the above-mentioned octahedral grids is a consequence of the theory developed in Section 3. Therein, it is established that IC-nets may be regarded as planar degenerations of spatial AC-grids which constitute appropriate samplings of asymptotic lines on one-sheeted hyperboloids. Such samplings are defined by the property that opposite vertices of the  quadrilaterals formed by the sampled asymptotic lines be connected by lines of curvature. The existence of such samplings is based on the fact that the nets of asymptotic lines and lines of curvature are, once again, mutually diagonal \cite{K76,K77}. In this connection, it is observed that it is a classical fact \cite{HCV52} that the one-parameter family of one-sheeted hyperboloids of a system of confocal quadrics may be regarded as being generated by the deformation of a single one-sheeted hyperboloid which is ``isometric along the asymptotic lines''. In this sense, AC-grids constitute spatial ``isometric'' deformations of IC-nets. A similar statement applies to IC-nets in hyperbolic geometry (HIC-nets) which are represented by planar grids composed of circles which are such that their centres are collinear and the ``quadrilaterals'' admit incircles. This is established in Section 4 in which it is shown that any CC-grid may be deformed into an HIC-net.
By definition, a CC-grid constitutes a sampling of the classical parallel circular cross-sections of an ellipsoid (or a hyperboloid) which is such that there exists a sampling of lines of curvature which is ``diagonal'' to the sampling of circular lines. The deformability of such CC-grids which is ``isometric along the circular lines'' is a direct consequence of the deformability of the two continuous families of circular cross-sections of an ellipsoid as asserted but not proven in \cite{HCV52}. Therein, it is observed that the one-parameter family of deformed ellipsoids obtained in such a manner is not confocal. However, we here  demonstrate that this one-parameter family of deformed ellipsoids may by generated by normalising the intermediate semi-axis of the one-parameter family of ellipsoids of a suitable system of confocal quadrics.

In Section 6, we introduce the Laguerre geometric notion of conical octahedral grids which constitute octahedral grids of oriented planes which are such that any four planes meeting at a point are in oriented contact with a circular cone. Octahedral webs of (non-oriented) planes and their canonical samplings (octahedral grids of planes) were studied in great detail by Sauer \cite{Sauer1925}. We demonstrate that conical octahedral grids are intimately related to systems of confocal quadrics in Minkowski space and (checkerboard) IC-nets. IC-nets are closely related to Poncelet grids \cite{AB} (also known as Poncelet-Darboux grids) introduced by Darboux in \cite{D72} and studied by various authors \cite{LT07,S07} (see, also, the monograph \cite{DR11}). A Poncelet-type theorem for geodesics on ellipsoids in Minkowski space and its connection with confocal quadrics is discussed in \cite{GKT07}. It turns out that the planes of a conical octahedral grid are in oriented contact with a common sphere. In this connection, it is noted that B\"ohm \cite{B} has shown that the requirement that the octahedra of an octahedral grid of (non-oriented) planes or the hexahedra of a grid of (non-oriented) planes of $\Z^3$ combinatorics circumscribe distinct spheres is too restrictive to obtain nontrivial three-dimensional IC-net-type configurations.


\section{Lines of curvature on confocal quadrics}

Any triple of real numbers $a > b> c$ gives rise to an associated system of confocal quadrics
\bela{E1}
  \frac{x^2}{\lambda + a} + \frac{y^2}{\lambda + b} + \frac{z^2}{\lambda + c} = 1
\ela
in a three-dimensional Euclidean space $\R^3$ with coordinates $\br = (x,y,z)$, where the parameter $\lambda$ labels the confocal quadrics $\mathcal{Q}_\lambda$. The values $\lambda = -a,-b,-c$ may be interpreted as limiting cases in which the quadrics become planar. The system of (non-planar) confocal quadrics is naturally decomposed into the three families of quadrics 
\bela{E2}
 \begin{split}
  \frac{x^2}{u_1 + a} + \frac{y^2}{u_1 + b} + \frac{z^2}{u_1 + c} & = 1\\
  \frac{x^2}{u_2 + a} + \frac{y^2}{u_2 + b} + \frac{z^2}{u_2 + c} & = 1\\
  \frac{x^2}{u_3 + a} + \frac{y^2}{u_3 + b} + \frac{z^2}{u_3 + c} & = 1\as
-a < u_1 < -b < u_2 <- c &< u_3 
 \end{split}
\ela
of different signature. Specifically, the parameters $u_1,u_2$ and $u_3$ label two-sheeted hyperboloids, one-sheeted hyperboloids and ellipsoids respectively. Any triple of confocal hyperboloids of different signature meet orthogonally in a unique point $(x,y,z)$ of any octant of $\R^3$ as indicated in Figure \ref{pic_confocal}.
\begin{figure}
  \centering
  \includegraphics[scale=0.18]{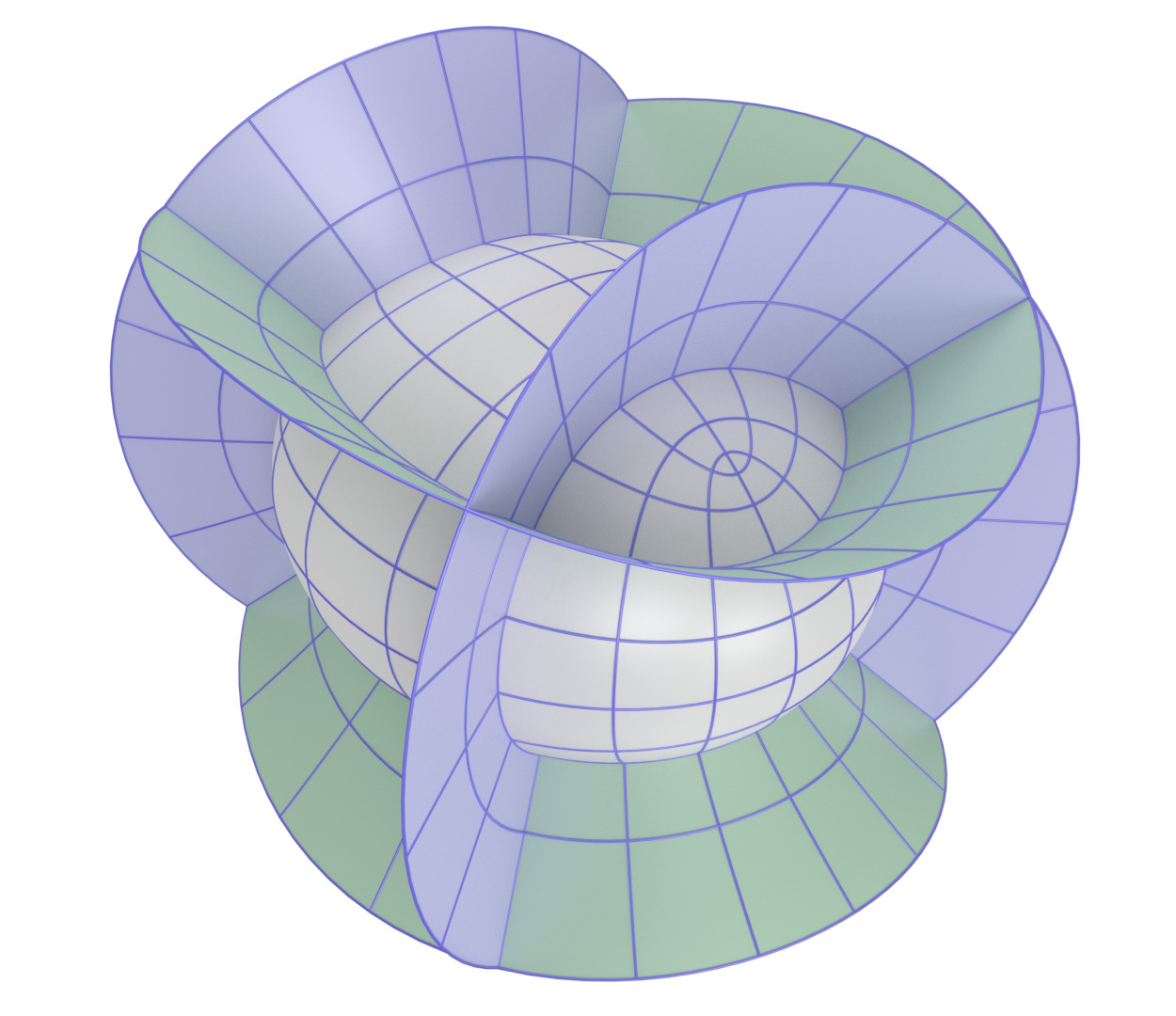}
\caption{The intersection of three different types of confocal quadrics.}
\label{pic_confocal}
\end{figure}
This fact may be expressed by the formulae
\bela{E3}
\begin{split}
  x^2 &= \frac{(u_1+a)(u_2+a)(u_3+a)}{(a-b)(a-c)}\\ 
  y^2 & = \frac{(u_1+b)(u_2+b)(u_3+b)}{(b-a)(b-c)}\\
  z^2 & = \frac{(u_1+c)(u_2+c)(u_3+c)}{(c-a)(c-b)}  
\end{split}
\ela
so that $\br(u_1,u_2,u_3)$ constitutes an orthogonal coordinate system in (the octants of) $\R^3$ to which one commonly refers as confocal coordinate system. Due to the orthogonality of the coordinate lines, which are the curves of intersection of pairs of confocal quadrics of different signature, the two families of coordinate lines on the quadrics $u_i=\mbox{const}$ parametrised by $u_k$ and $u_l$ ($i,k,l$ distinct) represent lines of curvature. 

Confocal quadrics and coordinate systems have a large number of interesting geometric and algebraic properties, some of which will be key to the discussion in the following sections. Here, we mention some important facts which will be exploited in the next section. Thus, for any reparametrisation of the confocal coordinates $u_i = u_i(s_i)$, the second fundamental forms
\bela{E4}
  \mbox{II}_{ik} = -d\br\cdot d\bN_{ik} = e_{ik}ds_i^2 + g_{ik}ds_k^2
\ela
of the quadrics $u_l=\mbox{const}$ with unit normal $\bN_{ik}\sim\br_{u_l}$ admit the factorisation property
\bela{E5}
 \frac{e_{ik}}{g_{ik}} =- \frac{U_i(s_i)}{U_k(s_k)},
\ela
where
\bela{E6}
  U_i = \frac{1}{4}\frac{u_i'^2}{(u_i+a)(u_i+b)(u_i+c)}.
\ela
Hence, we have established the following theorem.

\begin{theorem}\label{parametrisation}
There exists a unique parametrisation $u_i = u_i(s_i)$ of confocal coordinate lines (up to a scaling of the $s_i$ by  the same constant) such that the second fundamental forms of the confocal quadrics are conformally flat, that is,
\bela{E7}
  \mbox{\rm II}_{12} \sim ds_1^2 + ds_2^2,\quad \mbox{\rm II}_{13} \sim ds_1^2 - ds_3^2,\quad \mbox{\rm II}_{23} \sim ds_2^2 + ds_3^2.
\ela
\end{theorem}

\begin{remark}
It is noted that, in the classical literature, coordinates on a surface with respect to which the second fundamental form is conformally flat (modulo an appropriate re-parametrisation of the coordinate lines) are termed {\em isothermal-conjugate} \cite{Eisenhart1960}.
\end{remark}

It is evident that the determination of the unique parametrisation mentioned in the above theorem requires solving the ordinary differential equations
\bela{E8}
  U_1 = -U_2 = U_3 = \mbox{const} > 0,
\ela
where the $U_i$ are given by \eqref{E6}. Each of these differential equations is seen to be essentially the differential equation for the Weierstrass $\wp$-function. However, it turns out convenient to adopt an equivalent parametrisation in terms of Jacobi elliptic functions \cite{NIST}. This is achieved as follows. We first observe that the parametrisation \eqref{E3} reveals the factorisation property
\bela{E9}
  x = \frac{f_1(s_1)f_2(s_2)f_3(s_3)}{\sqrt{(a-b)(a-c)}},\quad
  y = \frac{g_1(s_1)g_2(s_2)g_3(s_3)}{\sqrt{(a-b)(b-c)}},\quad
  z = \frac{h_1(s_1)h_2(s_2)h_3(s_3)}{\sqrt{(a-c)(b-c)}}
\ela
with the functions $f_i$, $g_i$ and $h_i$ being related by the functional equations \cite{BSST17}
\bela{E10}
 \begin{aligned}
  f_1^2(s_1) + g_1^2(s_1) & = a-b,\quad & f_1^2(s_1) +h_1^2(s_1) & = a-c\\
  f_2^2(s_2) - g_2^2(s_2) & = a-b,\quad & f_2^2(s_2) +h_2^2(s_2) & = a-c\\
  f_3^2(s_3) - g_3^2(s_3) & = a-b,\quad & f_3^2(s_3) -h_3^2(s_3) & = a-c.
 \end{aligned}
\ela
For any solution of this system, the corresponding parameters $u_i$ may be reconstructed from the consistent system
\bela{E11}
 \begin{aligned}
   f_1^2(s_1) & = u_1 + a,\quad & g_1^2(s_1) & = -(u_1 + b),\quad & h_1^2(s_1) & = -(u_1 + c)\\
   f_2^2(s_2) & = u_2 + a,\quad & g_2^2(s_2) & = u_2 + b,\quad & h_2^2(s_2) & = -(u_2 + c)\\
   f_3^2(s_3) & = u_3 + a,\quad & g_3^2(s_3) & = u_3 + b,\quad & h_3^2(s_3) & = u_3 + c.
 \end{aligned}
\ela
The above functional equations are naturally parametrised in terms of Jacobi elliptic functions. Indeed, one may directly verify that these are satisfied by setting
\bela{E12}
  \begin{aligned}
  f_1(s_1) & =\sqrt{a-b}\,\jac{sn}(s_1,k_1),\quad & f_2(s_2) &= \sqrt{b-c}\,\frac{\jac{dn}(s_2,k_2)}{k_2},\quad & f_3(s_3) &= \sqrt{a-c}\,\jac{ns}(s_3,k_3)\\
  g_1(s_1) &=\sqrt{a-b}\,\jac{cn}(s_1,k_1),\quad & g_2(s_2) & = \sqrt{b-c}\,\jac{cn}(s_2,k_2),\quad & g_3(s_3) &= \sqrt{a-c}\,\jac{ds}(s_3,k_3)\\
 h_1(s_1) & =\sqrt{a-b}\,\frac{\jac{dn}(s_1,k_1)}{k_1},\quad & h_2(s_2) & = \sqrt{b-c}\,\jac{sn}(s_2,k_2),\quad & h_3(s_3) &= \sqrt{a-c}\,\jac{cs}(s_3,k_3),
 \end{aligned}
\ela
where the moduli of the elliptic functions are defined by
\bela{E13}
  k_1^2 = \frac{a-b}{a-c},\quad k_2^2 = \frac{b-c}{a-c} = 1-k_1^2,\quad k_3= k_1.
\ela
Furthermore, insertion of this parametrisation into \eqref{E9}, \eqref{E11} and \eqref{E6} leads to the following remarkable result.

\begin{theorem}\label{elliptic}
The second fundamental forms of the coordinate surfaces of the confocal coordinate system
\bela{E14}
  \left(\bear{c}x\\ y\\ z\ear\right) = \sqrt{a-c}\left(\bear{c}
     \jac{sn}(s_1,k_1)\jac{dn}(s_2,k_2)\jac{ns}(s_3,k_3)\\ 
     \jac{cn}(s_1,k_1)\jac{cn}(s_2,k_2)\jac{ds}(s_3,k_3)\\
     \jac{dn}(s_1,k_1)\jac{sn}(s_2,k_2)\jac{cs}(s_3,k_3)
\ear\right),
\ela
wherein the moduli $k_i$ are given by \eqref{E13}, are conformally flat with $U_1=-U_2=U_3 = 1/(a-c)$.
\end{theorem}

The geometric implications of the above theorem will be revealed in the following section.


\section{Asymptotic lines on (confocal) one-sheeted hyperboloids}

In this section, we examine the interplay between lines of curvature and asymptotic lines on both a single one-sheeted hyperboloid and its associated family of confocal quadrics. This will lead to, for instance, a connection between privileged spatial deformations of one-sheeted hyperboloids and planar incircular (IC) nets \cite{AB}. To this end, it is convenient to regard a one-sheeted hyperboloid as being embedded in its associated system of confocal quadrics \eqref{E2}. This allows us to exploit the parametrisations of quadrics presented in the previous section. 
In order to proceed, we first need to recall the key notion of ``mutually diagonal nets'' \cite{B28,K76,K77}.

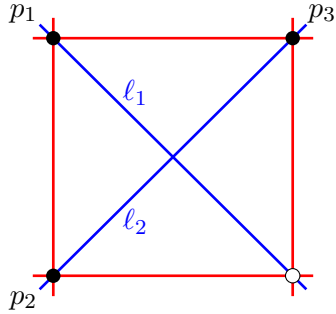
\begin{figure}
  \centering
  \begin{tikzpicture}[line cap=line join=round,>=stealth,x=1.0cm,y=1.0cm, scale=0.9]
      \draw [line width=1pt, color=red,  domain=-0.3:3.8] plot(0.,\x);
      \draw [line width=1pt, color=red,  domain=-0.3:3.8] plot(3.5,\x);
      \draw [line width=1pt, color=red,  domain=-0.3:3.8] plot(\x,0.);
      \draw [line width=1pt, color=red,  domain=-0.3:3.8] plot(\x,3.5);
      \draw [line width=1pt, color=blue,  domain=-0.2:3.7] plot(\x,\x);
      \draw [color=blue] (1.2, 0.8) node {$\ell_2$};
      \draw [line width=1pt, color=blue,  domain=-0.2:3.7] plot(\x,3.5 - \x);
      \draw [color=blue] (1.2, 2.7) node {$\ell_1$};      
      \coordinate[label={[label distance=3]135:$p_1$}] (p1) at (0.0,3.5);
      \fill (p1) circle (3pt);
      \coordinate[label={[label distance=3]225:$p_2$}] (p2) at (0.0,0.0);
      \fill (p2) circle (3pt);
      \coordinate[label={[label distance=3]45:$p_3$}] (p3) at (3.5,3.5);
      \fill (p3) circle (3pt);
      \fill [color=white] (3.5,0.0) circle (3pt);
      \draw (3.5,0.0) circle (3pt);
    \end{tikzpicture}
  \caption{Definition of two mutually diagonal nets. A ``quadrilateral'' of the net $\mathcal{N}_1$
  and its ``diagonals'' from the net $\mathcal{N}_2$.}
    \label{pic_diagonal}
\end{figure}

\begin{definition}\label{mutual}
Given a net $\mathcal{N}_1$ on a surface $\Sigma$, that is, a two-parameter family of curves on $\Sigma$ such that there exist exactly two curves of the family passing through any point on $\Sigma$, a net $\mathcal{N}_2$ is termed {\em diagonal to} $\mathcal{N}_1$ if the existence of a curve of $\mathcal{N}_2$ through a pair of opposite vertices of any ``quadrilateral'' formed by four curves of $\mathcal{N}_1$ implies that the remaining pair of opposite vertices is also connected by a curve of $\mathcal{N}_2$ (cf.\ Figure \ref{pic_diagonal}).
\end{definition}

\begin{remark}
It turns out that the above relation is symmetric \cite{B28}, that is, $\mathcal{N}_2$ being diagonal to $\mathcal{N}_1$ is equivalent to $\mathcal{N}_1$ being diagonal to $\mathcal{N}_2$. Furthermore, an alternative characterisation of this property is the following. Given a pair of intersecting lines $\ell_1$ and $\ell_2$ of $\mathcal{N}_2$ and a point $p_1$ on $\ell_1$, there exist two lines of $\mathcal{N}_1$ which pass through $p_1$ with the points of intersection with $\ell_2$ being denoted by $p_2$ and $p_3$. Then, the two additional lines of $\mathcal{N}_1$ passing through $p_2$ and $p_3$ meet on $\ell_1$ as illustrated in Figure \ref{pic_diagonal}.
\end{remark}

Orthogonal nets, that is, nets consisting of curves which are pairwise orthogonal, admit the following property \cite{K77} which, in particular, applies to nets composed of lines of curvature.

\begin{theorem}\label{bisection}
Any orthogonal net $\mathcal{N}_1$ on a surface $\Sigma$ bisects an infinite number of nets $\mathcal{N}_2$ which are diagonal to $\mathcal{N}_1$.  
\end{theorem}

\begin{proof}
It is known \cite{B28,K76} that for any pair of mutually diagonal nets $\mathcal{N}_1$ and $\mathcal{N}_2$ on a surface $\Sigma$, one may choose a parametrisation $\br(u,v)$ of $\Sigma$ such that the curves $u=\mbox{const}$, $v=\mbox{const}$ and $\alpha = u+v=\mbox{const}$, $\beta = u-v=\mbox{const}$ represent the curves of the nets $\mathcal{N}_1$ and $\mathcal{N}_2$ respectively. Conversely, if $\br(u,v)$ is a parametrisation of a surface $\Sigma$ such that the curves $u=\mbox{const}$ and $v=\mbox{const}$ represent a net $\mathcal{N}_1$ then the net $\mathcal{N}_2$ consisting of the curves $\alpha = u+v=\mbox{const}$ and $\beta = u-v=\mbox{const}$ is diagonal to $\mathcal{N}_1$. If the net $\mathcal{N}_1$ is orthogonal so that $\br_u\cdot\br_v=0$ then it is evident that
\bela{E14a}
  \br_\alpha^2 = \br_\beta^2,\quad \br_u\cdot\br_\alpha = \br_u\cdot\br_\beta,\quad \br_v\cdot\br_\alpha = -\br_v\cdot\br_\beta
\ela
and, hence, any pair of curves of the net $\mathcal{N}_1$ intersecting in a point bisect the two curves of the net $\mathcal{N}_2$ which likewise pass through that point.
\end{proof}

The following theorem represents the basis of the material presented in this section.

\begin{theorem}\label{diagonal}
On any one-sheeted hyperboloid $\mathcal{H}$, the lines of curvature and the (straight) asymptotic lines form mutually diagonal nets.  Moreover, if $P_1,P_3,P_{\bar{1}},P_{\bar{3}}$ are the vertices of a quadrilateral of asymptotic lines and the ``curved diagonals'' constitute lines of curvature as illustrated in Figure \ref{pic_patch} then $\overline{P_1P_3} + \overline{P_{\bar{1}}P_{\bar{3}}} = \overline{P_1P_{\bar{3}}} + \overline{P_{\bar{1}}P_3}$. 
\end{theorem}

\begin{figure}
  \centering
  \begin{tikzpicture}[line cap=line join=round,>=stealth,x=1.0cm,y=1.0cm, scale=1]
    \node[anchor=south west,inner sep=0] (image) at (0,0) {\includegraphics[scale=0.15]{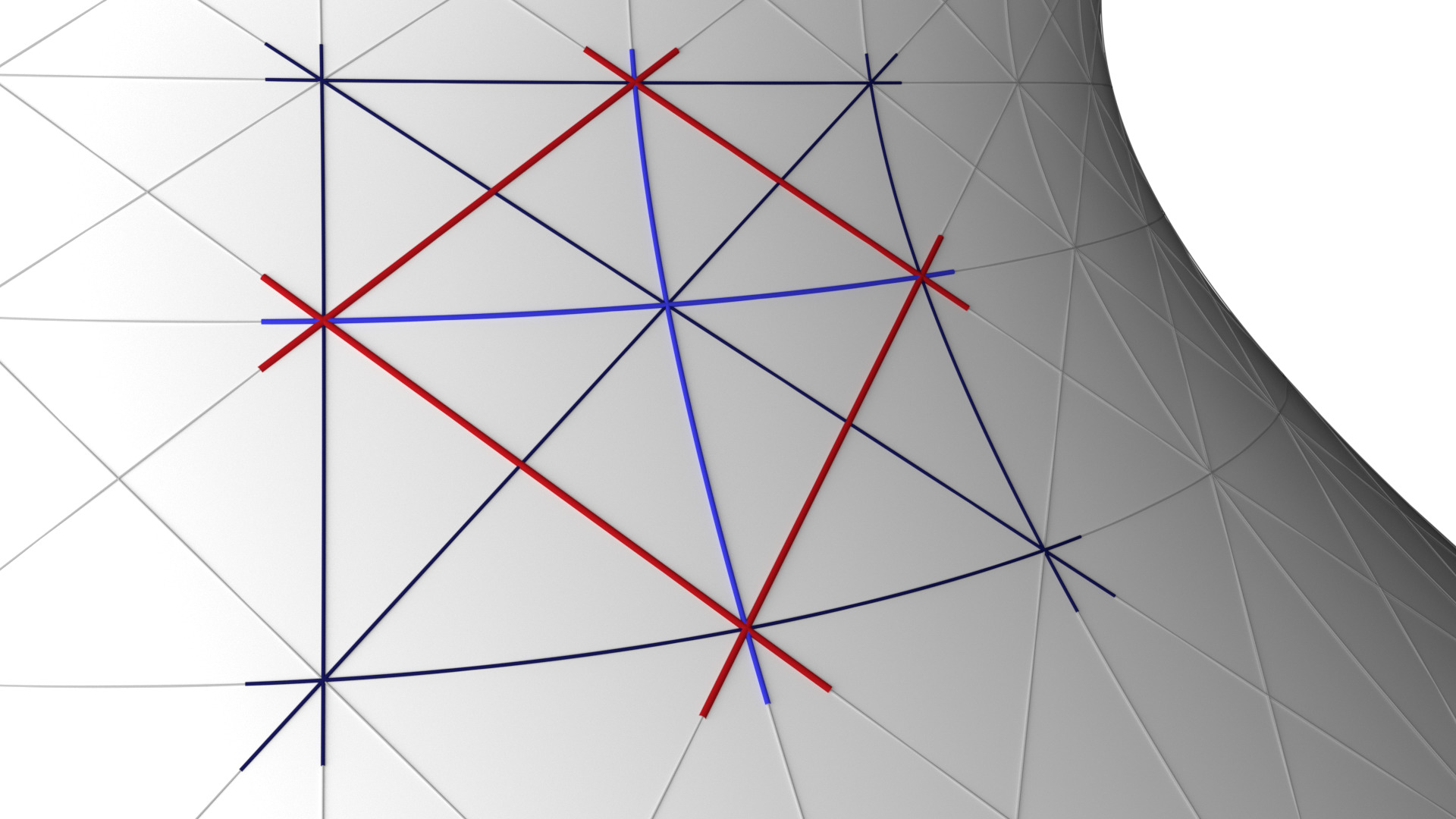}};
    
    \coordinate[label=180:\contour{white}{$P_1$}] (p1) at (2.30,3.53);
    \fill (p1) circle (2pt);
    \coordinate[label=0:\contour{white}{$P_{\bar{1}}$}] (p1b) at (6.53,3.85);
    \fill (p1b) circle (2pt);
    \coordinate[label=90:\contour{white}{$P$}] (p) at (4.74,3.65);
    \fill (p) circle (2pt);
    \coordinate[label=270:\contour{white}{$P_3$}] (p3) at (5.30,1.35);
    \fill (p3) circle (2pt);
    \coordinate[label=90:\contour{white}{$P_{\bar{3}}$}] (p3b) at (4.50,5.22);
    \fill (p3b) circle (2pt);
    \coordinate[label=225:\contour{white}{$P_{13}$}] (p13) at (2.32,0.98);
    \fill (p13) circle (2pt);
    \coordinate[label=-45:\contour{white}{$P_{\bar{1}3}$}] (p1b3) at (7.41,1.92);
    \fill (p1b3) circle (2pt);
    \coordinate[label=125:\contour{white}{$P_{1\bar{3}}$}] (p13b) at (2.29,5.24);
    \fill (p13b) circle (2pt);
    \coordinate[label=45:\contour{white}{$P_{\bar{1}\bar{3}}$}] (p1b3b) at (6.17,5.22);
    \fill (p1b3b) circle (2pt);
  \end{tikzpicture}
  \caption{A patch of mutually diagonal asymptotic and curvature lines on a one-sheeted hyperboloid.}
  \label{pic_patch}
\end{figure}
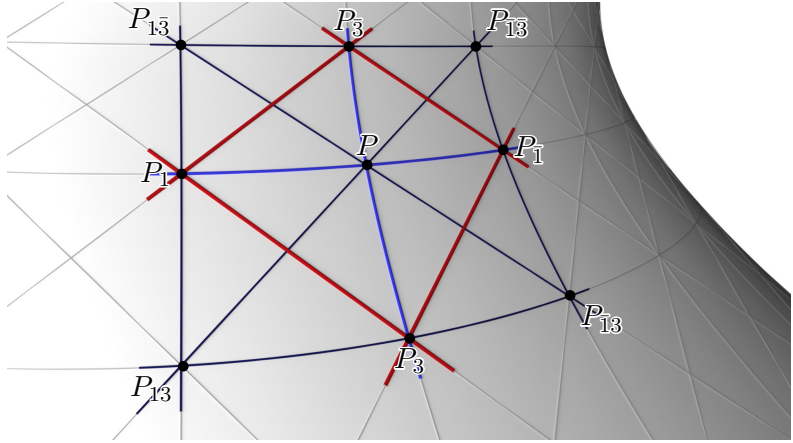

\begin{remark}\label{touch}
The above property of the sum of the lengths of opposite edges of a quadrilateral being the same for both pairs of edges implies that any such quadrilateral of asymptotic lines is associated with a one-parameter families of spheres which touch the edges. If the quadrilateral is planar then there exists a unique circle which is inscribed in the quadrilateral.
\end{remark}

\begin{proof} 
Diagonal nets of lines of curvature and asymptotic lines have been discussed in great detail in \cite{K76,K77}. In view of the application of the above theorem to systems of confocal quadrics, we here present an algebraic proof. Thus, we identify the hyperboloid $\mathcal{H}$ with a one-sheeted hyperboloid $\mathcal{H}_{u_2}$ of an associated system of confocal quadrics for a fixed value of the parameter $u_2$. Then, according to Theorem \ref{parametrisation}, the lines of curvature may be parametrised in such a manner that the second fundamental form of $\mathcal{H}_{u_2}$ is of the form
\bela{E15}
  \mbox{II}_{13} \sim ds_1^2 - ds_3^2.
\ela
Accordingly, the asymptotic lines on $\mathcal{H}_{u_2}$ are given by $s_1\pm s_3 = \mbox{const}$. Consequently, the pre-images of the lines of curvature and asymptotic lines on the $(s_1,s_3)$-plane trivially form mutually diagonal nets and, hence, curvature lines and asymptotic lines are likewise mutually diagonal. 

For the proof of the second part of the theorem, we introduce four additional lines of curvature which pass through the points  $P_1,P_3,P_{\bar{1}},P_{\bar{3}}$ and denote their points of intersection by  $P_{13,}P_{\bar{1}3},P_{\bar{1}\bar{3}},P_{1\bar{3}}$ (cf.\ Figure \ref{pic_patch}). Moreover, $P$ is the point of intersection of the original pair of lines of curvature.  Since the net of asymptotic lines is diagonal to the net of lines of curvature, the points $P_{13}, P, P_{\bar{1}\bar{3}}$ lie on an asymptotic line and so do the points $P_{\bar{1}3}, P, P_{1\bar{3}}$. Ivory's classical theorem \cite{IT} applied to the ``big'' quadrilateral of lines of curvature then states that $\overline{P_{13}P_{\bar{1}\bar{3}}} = \overline{P_{\bar{1}3}P_{1\bar{3}}}$. However, Ivory's theorem applied to the four ``small'' quadrilaterals of lines of curvature also implies that $\overline{P_{13}P_{\bar{1}\bar{3}}} = \overline{P_1P_3} + \overline{P_{\bar{1}}P_{\bar{3}}}$ and $\overline{P_{\bar{1}3}P_{1\bar{3}}} = \overline{P_1P_{\bar{3}}} + \overline{P_{\bar{1}}P_3}$ which completes the proof.
\end{proof}

The above theorem may now be exploited to determine the relationship between the lines of curvature and asymptotic lines on confocal one-sheeted hyperboloids. We begin by stating a theorem which is a direct consequence of the parametrisation \eqref{E3} of confocal quadrics.

\begin{theorem}
The members of the one-parameter family of one-sheeted hyperboloids of a system of confocal quadrics are related by affine (scaling) transformations. Specifically, the affine transformation $A_{u_2\mapsto\bar{u}_2}:\R^3\rightarrow\R^3$,
\bela{E16}
  A_{u_2\mapsto\bar{u}_2}(x,y,z) = \left(\sqrt{\frac{\bar{u}_2+a}{u_2+a}}\,x,\sqrt{\frac{\bar{u}_2+b}{u_2+b}}\,y,\sqrt{\frac{\bar{u}_2+c}{u_2+c}}\,z\right)
\ela
maps $\mathcal{H}_{u_2}$ to $\mathcal{H}_{\bar{u}_2}$.
\end{theorem}

The above affine transformation gives rise to the following geometric properties.

\begin{theorem}\label{deformation}
The one-parameter family of one-sheeted hyperboloids of any system of confocal quadrics is generated by the deformation of any constituent hyperboloid
$\mathcal{H}_{u_2}$ which is governed by the affine transformation $A_{u_2\mapsto\bar{u}_2}$, where $\bar{u}_2-u_2$ is interpreted as the deformation parameter. This deformation enjoys the following properties.
\begin{itemize}
\item Lines of curvature and asymptotic lines and their mutually diagonal relationship are preserved.
\item The deformation is isometric along asymptotic lines, that is, the distance between any two points $P_{u_2}$ and $P^*_{u_2}$ on any asymptotic line is preserved.
\item Any two points $P_{u_2}$ and $P^*_{u_2}$ on any asymptotic line and their deformed counterparts $P_{\bar{u}_2}$ and $P^*_{\bar{u}_2}$ constitute the vertices of a skew parallelogram, that is, in addition to $\overline{P_{u_2}P^*_{u_2}}=\overline{P_{\bar{u}_2}P^*_{\bar{u}_2}}$, the equality $\overline{P_{u_2}P^*_{\bar{u}_2}}=\overline{P^*_{u_2}P_{\bar{u}_2}}$ holds.
\item In the two limiting cases $u_2=-c$ and $u_2=-b$, the mutually diagonal nets of lines of curvature and asymptotic lines become planar with the straight lines touching the focal conics 
\bela{E17}
\begin{aligned}
  \frac{x^2}{a-c} + \frac{y^2}{b-c} & = 1,\quad & z & = 0\\
  \frac{x^2}{a-b} - \frac{z^2}{b-c} & = 1,\quad & y & = 0
 \end{aligned} 
\ela
respectively and the lines of curvature becoming conics which are confocal to the respective focal conic. The two planar pairs of nets of lines and confocal conics are isometric along their lines in the above sense. 
\end{itemize}
\end{theorem}

\begin{proof}
It is evident that the point transformation $A_{u_2\mapsto\bar{u}_2}$ preserves the coordinate lines $u_1=\mbox{const}$ and $u_3=\mbox{const}$ since, by construction,
$A_{u_2\mapsto\bar{u}_2}(\br(u_1,u_2,u_3)) = \br(u_1,\bar{u}_2,u_3)$. The trivial fact that the asymptotic lines are preserved since straight lines are preserved by affine transformations is also reflected in the fact that the special parametrisation of the lines of curvature in which the second fundamental form is conformally flat (cf.\ proof of Theorem \ref{diagonal}) is independent of $u_2$. Thus, the mutually diagonal nets of lines of curvature and asymptotic lines are preserved.

Since confocal coordinate systems are orthogonal, at any instant of the deformation, any point on a one-sheeted hyperboloid moves in the direction perpendicular to the corresponding tangent plane. Accordingly, asymptotic lines undergo rigid motions during the deformation so that the deformation is isometric along asymptotic lines. Specifically, if $P_{u_2}$ and $P^*_{u_2}$ are two points on an asymptotic line and if $P_{\bar{u}_2}$ and $P^*_{\bar{u}_2}$ denote the corresponding points on the deformed asymptotic line then $\overline{P_{u_2}P^*_{u_2}}=\overline{P_{\bar{u}_2}P^*_{\bar{u}_2}}$. These four points may be regarded as four vertices of a ``curved'' hexahedron, the ``edges'' of which are segments of lines of curvature on the corresponding confocal quadrics (including the two hyperboloids) as indicated in Figure \ref{pic_hexahedron}. According to Ivory's theorem \cite{IT}, the long diagonals of this hexahedron are congruent so that, in particular, $\overline{P_{u_2}P^*_{\bar{u}_2}}=\overline{P^*_{u_2}P_{\bar{u}_2}}$.

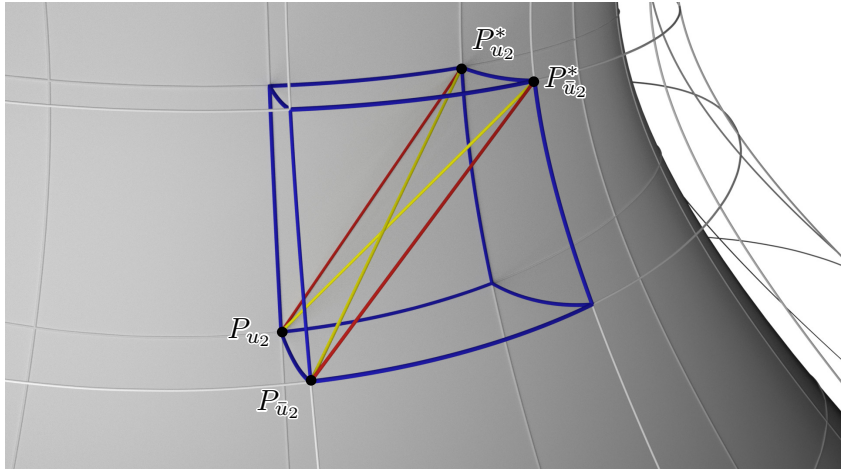
\begin{figure}
  \centering
  \begin{tikzpicture}[line cap=line join=round,>=stealth,x=1.0cm,y=1.0cm, scale=1]
    \node[anchor=south west,inner sep=0] (image) at (0,0) {\includegraphics[scale=0.16]{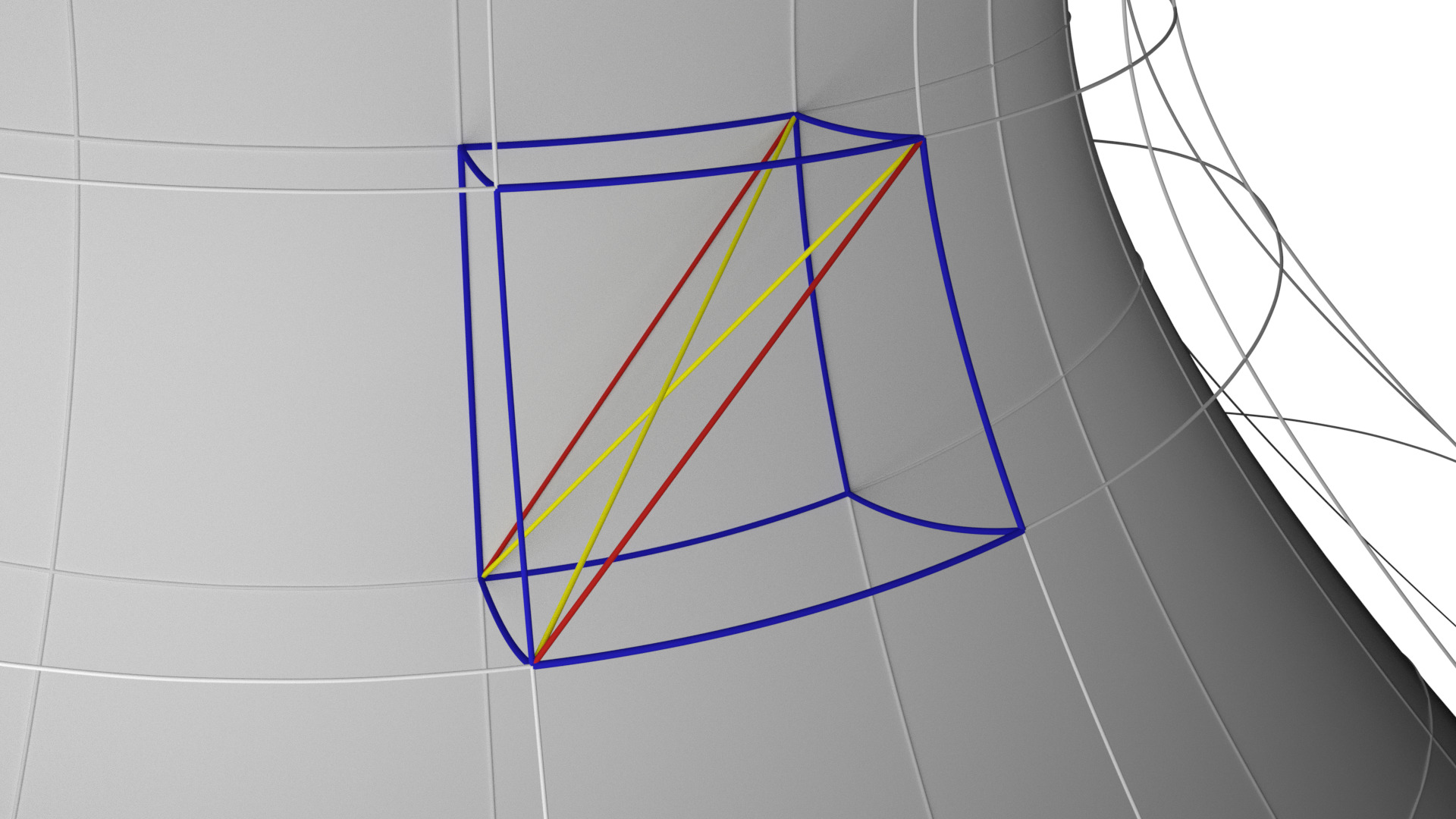}};
    
    \coordinate[label=180:\contour{white}{$P_{u_2}$}] (p2) at (3.65,1.82);
    \fill (p2) circle (2pt);
    \coordinate[label=225:\contour{white}{$P_{\bar{u}_2}$}] (p2b) at (4.03,1.18);
    \fill (p2b) circle (2pt);
    \coordinate[label=45:\contour{white}{$P_{u_2}^*$}] (p2s) at (6.01,5.31);
    \fill (p2s) circle (2pt);
    \coordinate[label=0:\contour{white}{$P_{\bar{u}_2}^*$}] (p2bs) at (6.96,5.14);
    \fill (p2bs) circle (2pt);
  \end{tikzpicture}
  
  \centering
  
  \caption{A curved hexahedron
    made of lines of curvature linking two confocal one-sheeted hyperboloids.
    The points $P_{u_2}, P_{u_2}^*, P_{\bar{u}_2}, P_{\bar{u}_2}^*$ constitute the vertices of a skew parallelogram.}
  \label{pic_hexahedron}
\end{figure}

The asymptotic lines on the hyperboloid $\mathcal{H}_{u_2}$ are ``tangent'' to the ellipse
\bela{E18}
  \frac{x^2}{u_2 + a} + \frac{y^2}{u_2 + b} = 1,\quad  z  = 0
\ela
in the sense that they lie in the tangent planes of the hyperboloid at the points of this ellipse. Hence, in the limit $u_2=-c$, these lines become tangent to the ellipse \eqref{E17}$_{1,2}$ which is one of the focal conics \cite{HCV52} associated with the underlying system of confocal quadrics. By virtue of \eqref{E3}, the coordinate system on the $(x,y)$-plane simplifies to
\bela{E19}
  x^2 = \frac{(u_1+a)(u_3+a)}{a-b},\quad y^2 = \frac{(u_1+b)(u_3+b)}{b-a}
\ela
which is the standard confocal coordinate system on the plane. Accordingly, the $u_1$- and $u_3$-lines constitute conics which are confocal to the ellipse \eqref{E17}$_{1,2}$. The same reasoning may now be applied to the limiting case $u_2=-b$ being associated with the focal hyperbola \eqref{E17}$_{3,4}$.
\end{proof}

The theory established in the preceding may now be applied to appropriate samplings of lines of curvature and asymptotic lines on one-sheeted hyperboloids to which we will refer as  {\em grids}. The existence of these grids is geometrically guaranteed by virtue of the mutual diagonality of the nets of lines of curvature and asymptotic lines.

\begin{definition}
A configuration of the combinatorics of a $\Z^2$-grid of two sequences $\left(\ell_n\right)_{n\in\Z}$ and $\left(m_{n'}\right)_{n'\in \Z}$ of the two one-parameter families of asymptotic lines on a one-sheeted hyperboloid is termed an {\em AC-grid} if pairs of opposite vertices of the elementary quadrilaterals formed by the lines $\ell_n,\ell_{n+1}$ and $m_{n'},m_{n'+1}$ are connected by lines of curvature.
\end{definition}

In this connection, Theorem \ref{diagonal} together with Remark \ref{touch} reveals the significance of the incircular (IC) nets introduced in \cite{AB}.

\begin{definition}
A planar configuration of straight lines of the combinatorics of a $\Z^2$-grid is termed an {\em  IC-net} if the elementary quadrilaterals formed by the lines circumscribe circles.
\end{definition}

\begin{remark}
IC-nets have the remarkable characteristic property that their straight lines are tangent to a conic and opposite pairs of vertices of elementary quadrilaterals are connected by conics which are confocal to the conic of tangency. Examples of IC-nets being ``bounded'' by an ellipse ({\em IC$_e$-net}) and a hyperbola ({\em IC$_h$-net}) respectively are displayed in Figure \ref{pic_icnets}.
\end{remark}

\begin{figure}
  \centering
  \includegraphics[scale=0.17]{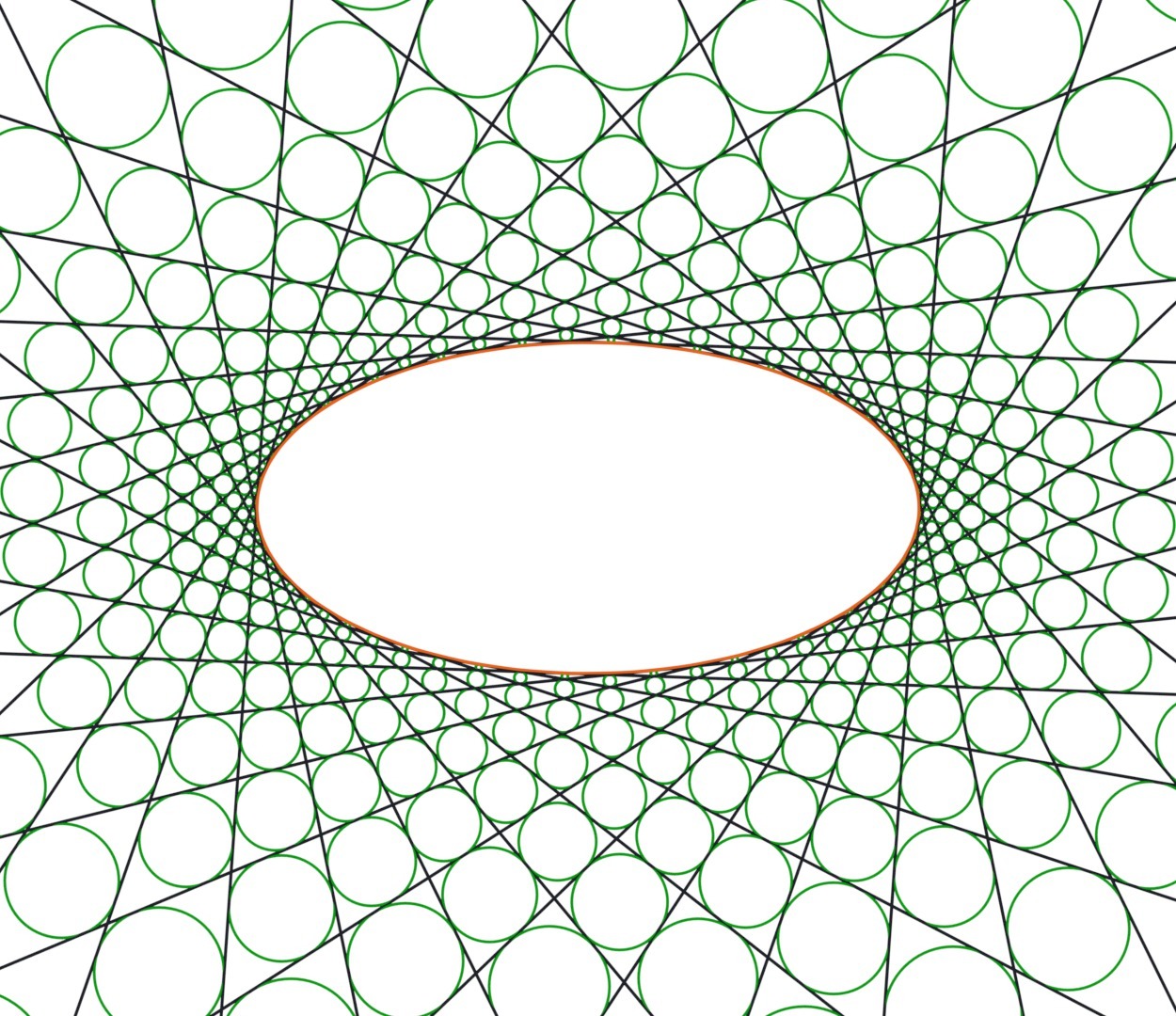}
  \includegraphics[scale=0.17]{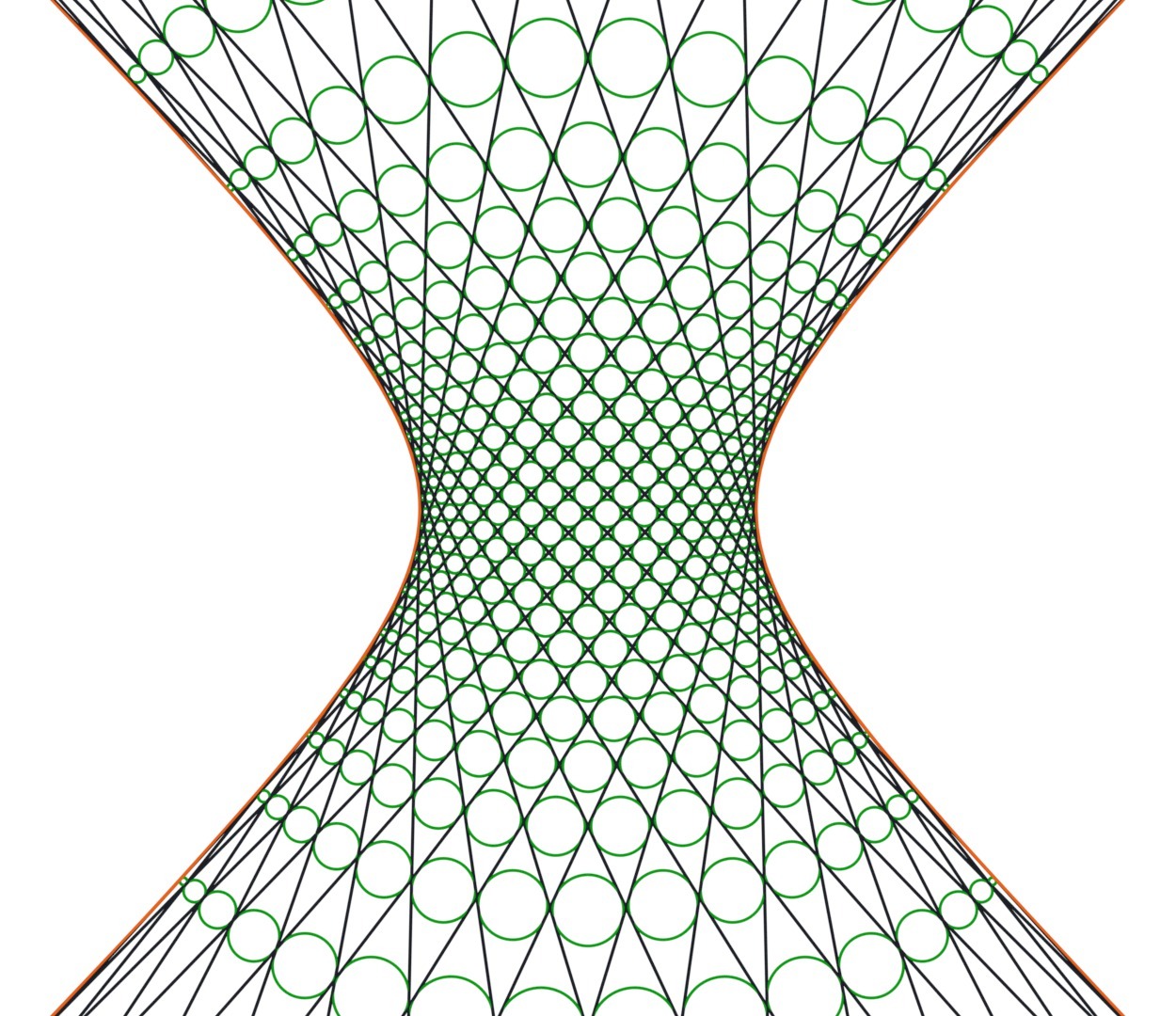}
\caption{An \ICe (left) and an \ICh (right).}
\label{pic_icnets}
\end{figure}

We conclude that AC-grids are non-planar generalisations of IC-nets as summarised below.

\begin{corollary}
The deformation linking the family of one-sheeted hyperboloids of a system of confocal quadrics as detailed in Theorem \ref{deformation} acts on associated AC-grids as follows. 
\begin{itemize}
\item AC-grids are preserved by the deformation.
\item During the deformation, the elementary quadrilaterals are folded but the length of the edges is preserved.
\item AC-grids become an \ICe and an \ICh in the planar limiting cases $u_2=-c$ and $u_2=-b$ respectively.
\item Any \ICe (or \ICh) may be transformed into an \ICh (or an \ICe) by means of a spatial deformation into AC-grids which preserves lines and is isometric along the lines.
\item AC-grids and their deformation are algebraically represented by the privileged parametrisation \eqref{E14} with the vertices of the grids being parametrised by $(s_1,s_3) = \delta(n_1+n_3,n_1-n_3) + (s_1^0,s_3^0)$, $n_1,n_3\in\Z$, where $\delta$ and $s_1^0,s_3^0$ are arbitrary constants.
\end{itemize}
\end{corollary}

\begin{proof}
For completeness, it is required to show that any IC-net may be interpreted as a planar deformation of an AC-grid. For instance, in the case of an \ICe, we may assume without loss of generality that the ellipse to which its lines are tangent is given by \eqref{E17}$_{1,2}$. This ellipse then constitutes the focal ellipse of a unique system of confocal quadrics. As stated in Theorem \ref{deformation}, in the limit $u_2=-c$, the asymptotic lines on the corresponding one-parameter family of one-sheeted hyperboloids become lines which touch the ellipse \eqref{E17}$_{1,2}$, while the lines of curvature become confocal conics. Accordingly, the lines of the IC-net and the associated diagonal confocal conics may be regarded as samplings of these planar mutually diagonal nets of tangent lines and confocal conics so that the corresponding samplings of the asymptotic nets and lines of curvature on the one-sheeted hyperboloids constitute the required AC-grids.
\end{proof}

\begin{remark}
The interpretation of a one-parameter family of one-sheeted confocal hyperboloids in terms of the deformation of a single hyperboloid which is isometric along its generators may be found in \cite{HCV52}. In fact, therein, it is pointed out that one may construct a deformable model of a one-sheeted hyperboloid by realising arbitrary samplings of generators as a collection of ``rods'' which are fastened at the points of intersection (see Figure \ref{asymptotic_wien}). 
However, it appears that the properties of AC-grids which constitute privileged samplings and their connection with IC-nets is not considered. In Figure \ref{pic_deformhyp}, various stages in the deformation of an AC-grid are displayed. Figure \ref{pic_deformhyp} (top-left) and Figure \ref{pic_deformhyp} (bottom-right) constitute an \ICe and an \ICh respectively. Here, the constant $\delta$ has been chosen in such a manner that the AC-grid ``closes'', that is, $\delta =4\mathsf{K}(k)/N$, $N\in\N$, where $\mathsf{K}(k)$ is the complete integral of the first kind and the real quarter-period of the Jacobi elliptic functions \cite{NIST}. 
\end{remark}

\begin{figure}
  \centering
  \includegraphics[scale=0.294]{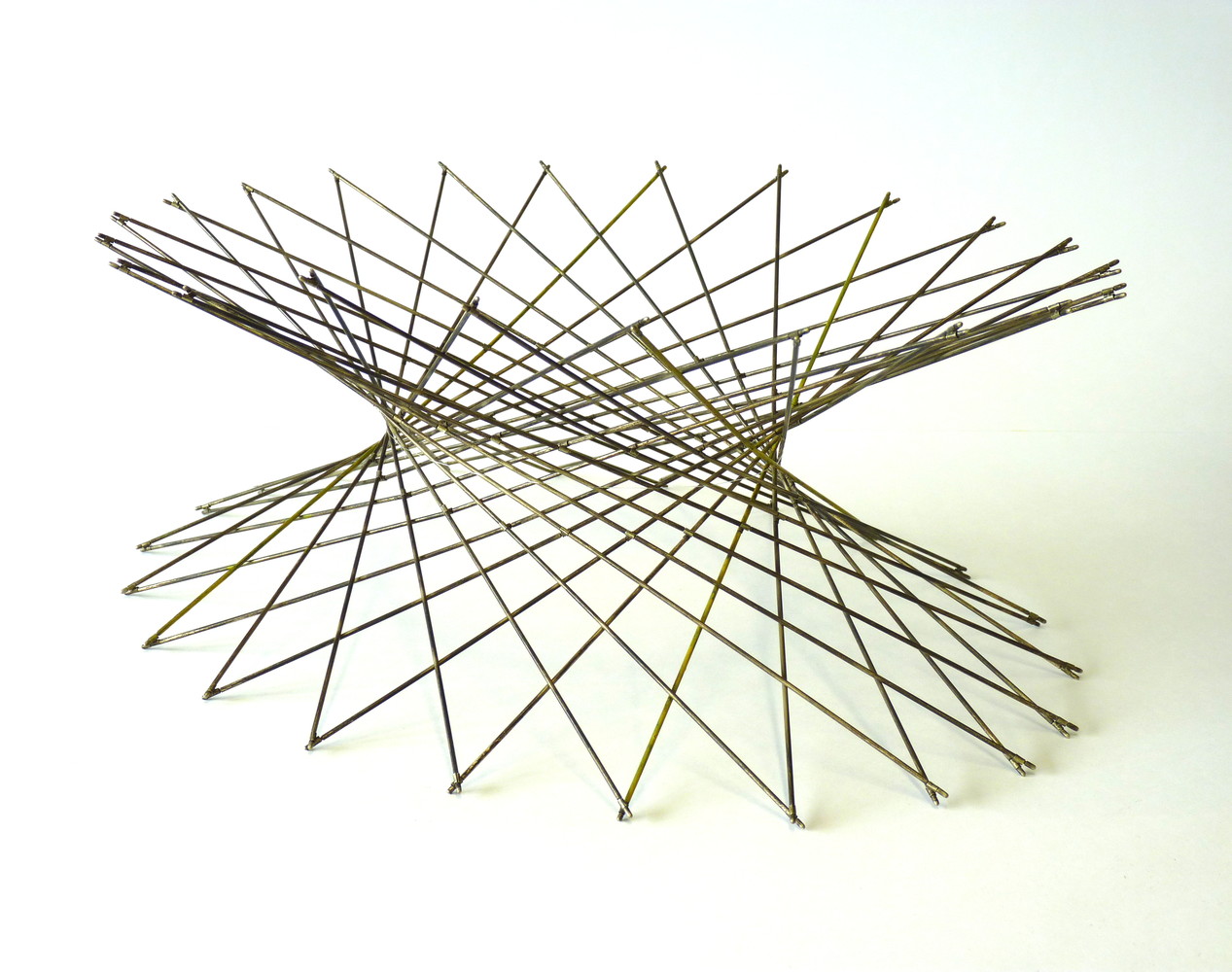}
  \includegraphics[scale=0.294]{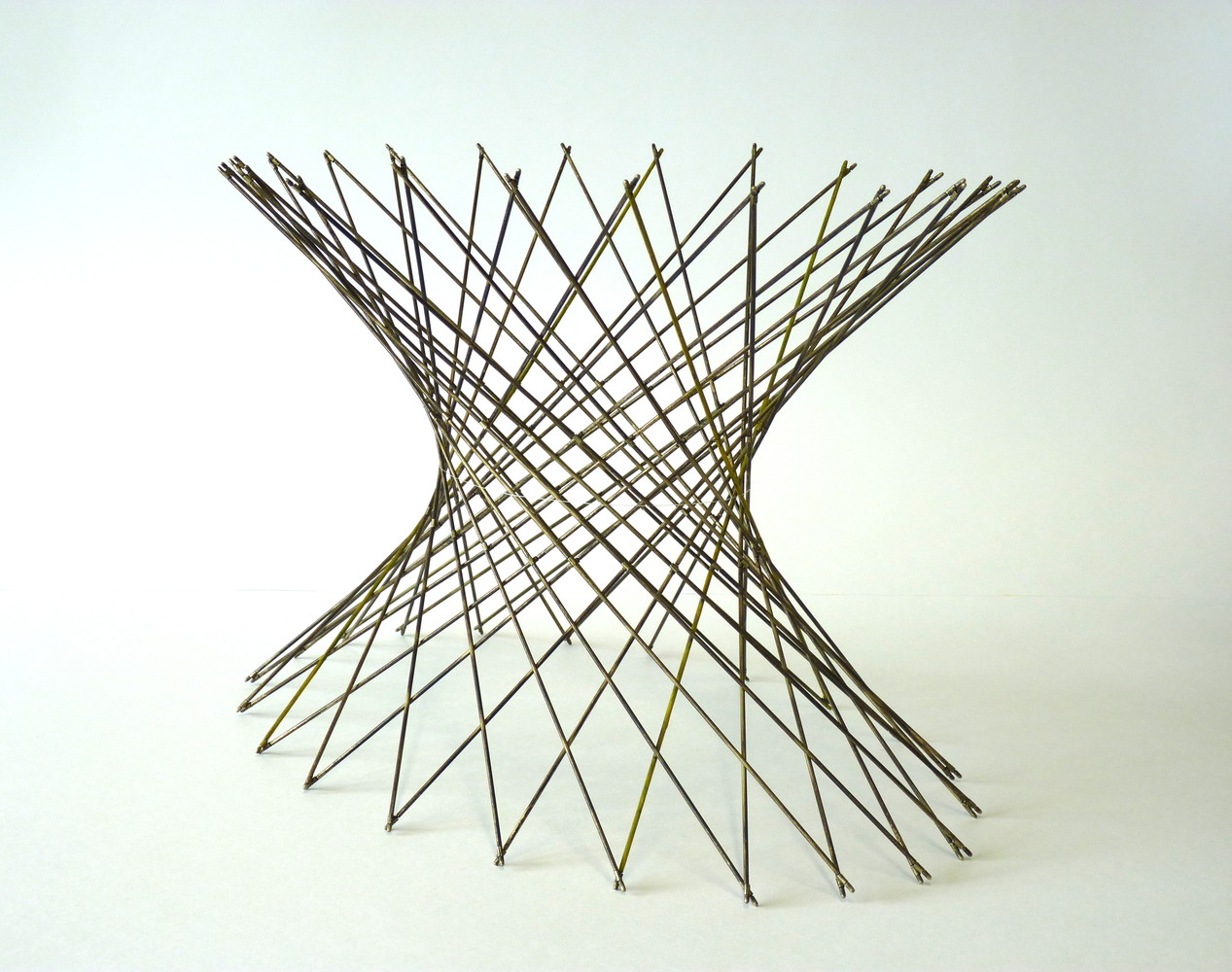}
  \includegraphics[scale=0.294]{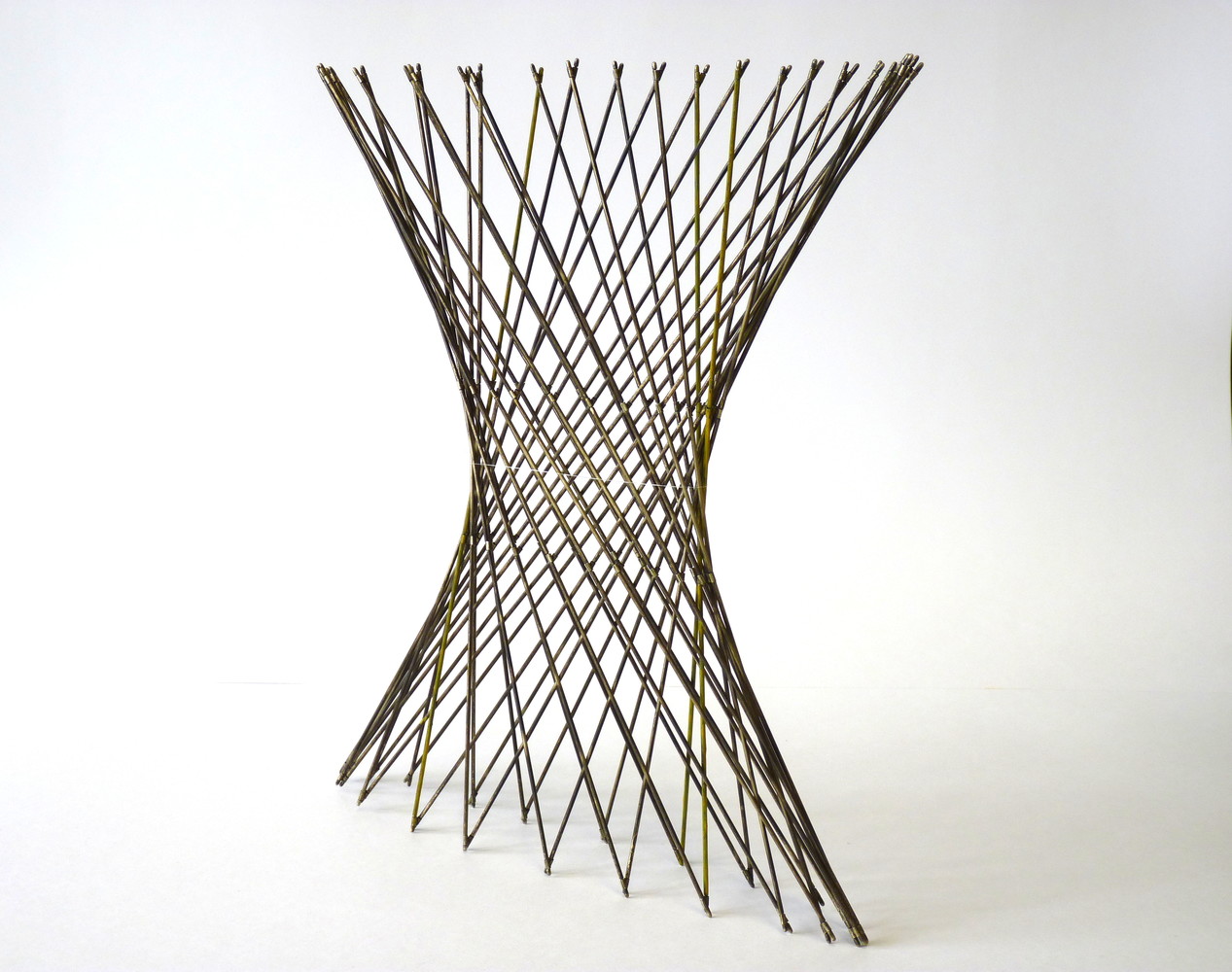}
  \caption{Photos of an ``isometrically'' deformable model of a one-sheeted hyperboloid from the collection of the Institute of Discrete Mathematics and Geometry, TU Wien.}
  \label{asymptotic_wien}
\end{figure}

\begin{figure}
  \centering
  \includegraphics[scale=0.135]{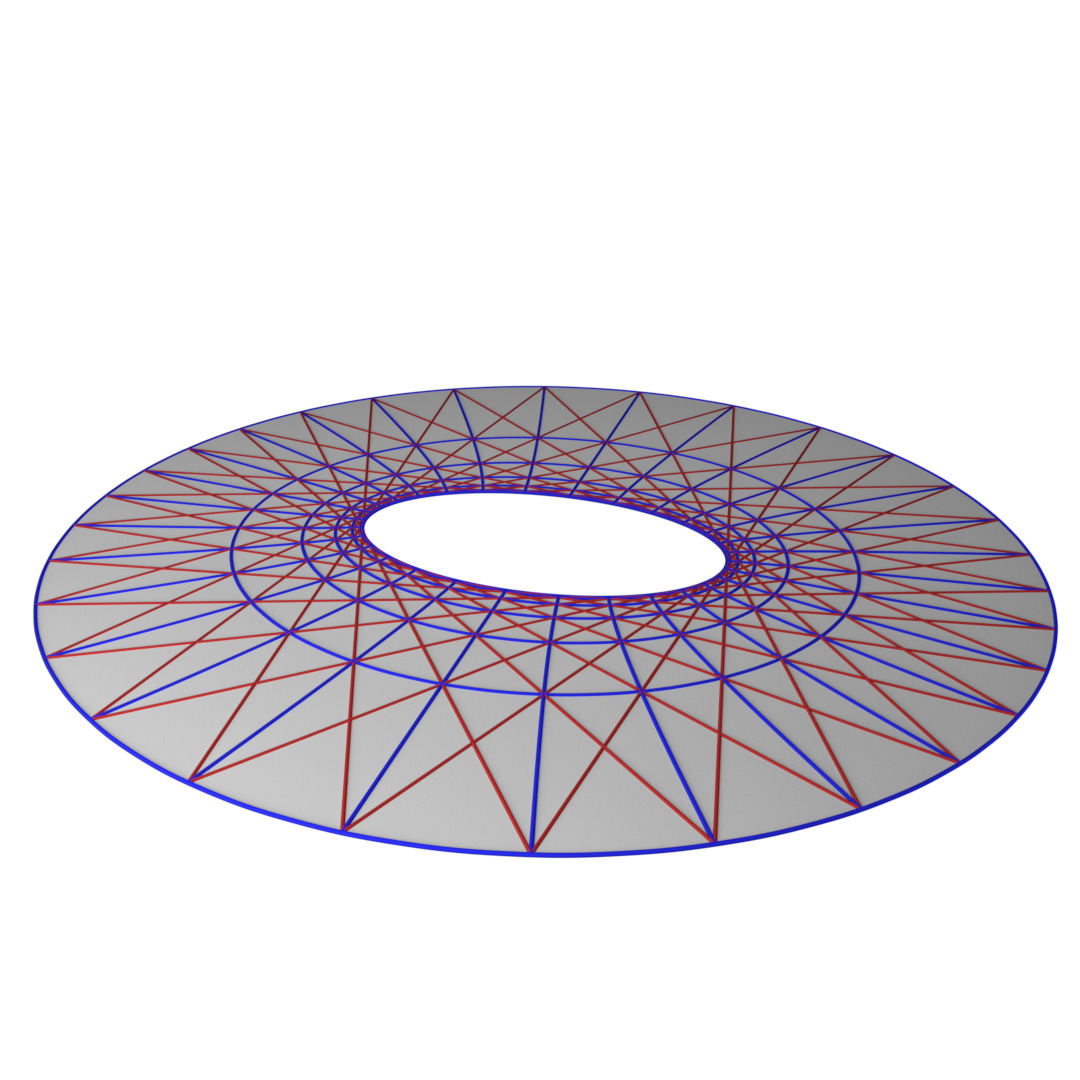}
  \includegraphics[scale=0.135]{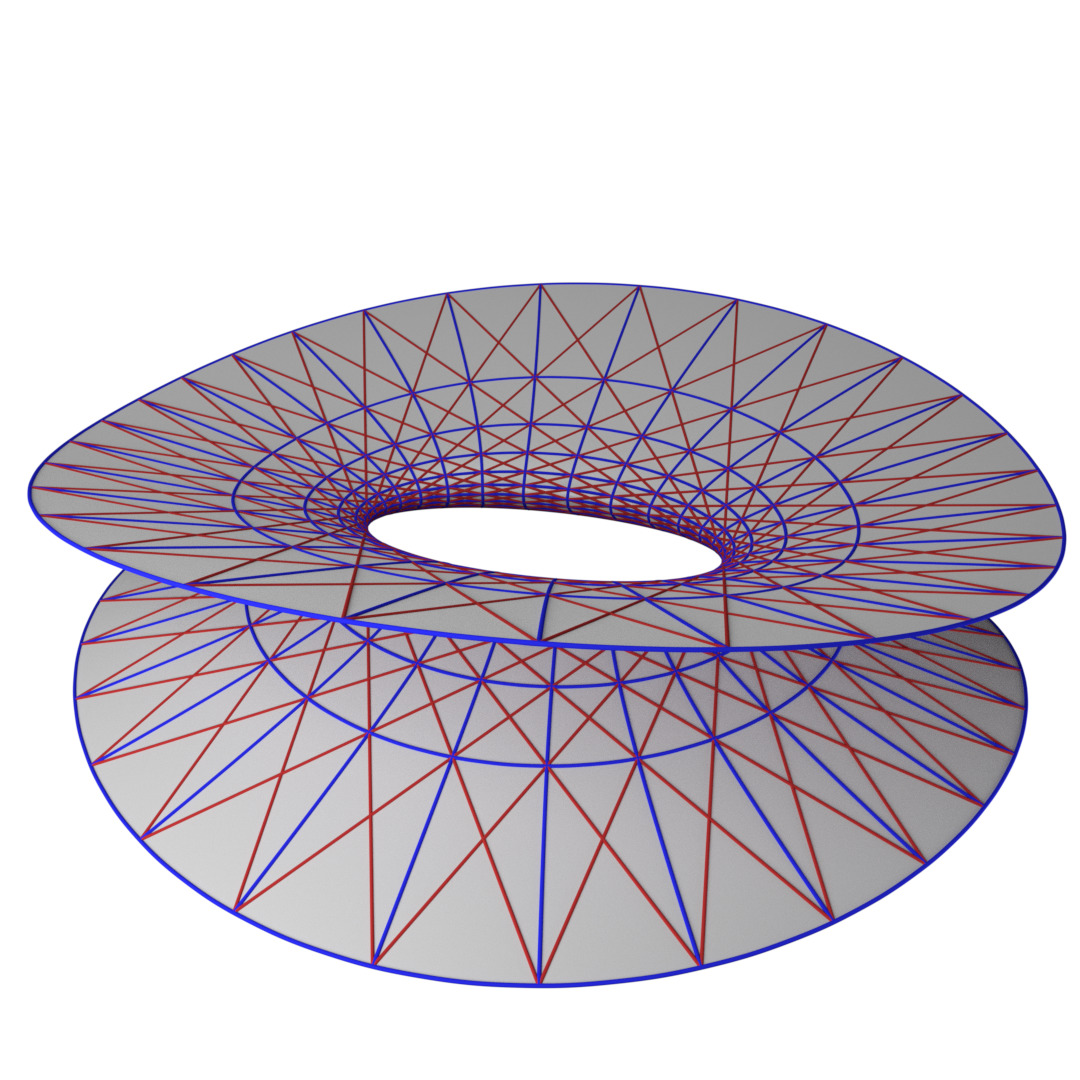}
  \includegraphics[scale=0.135]{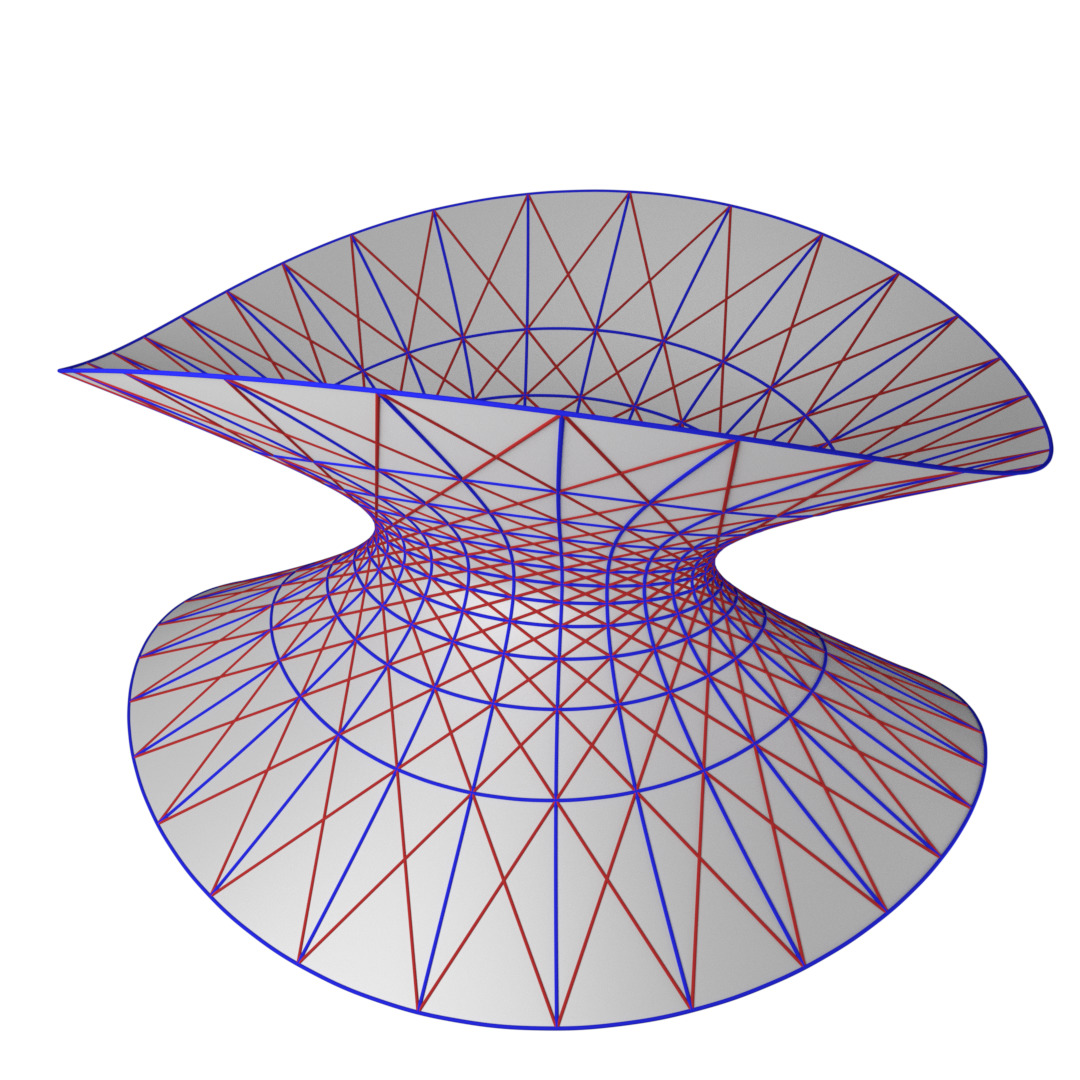}
  \includegraphics[scale=0.135]{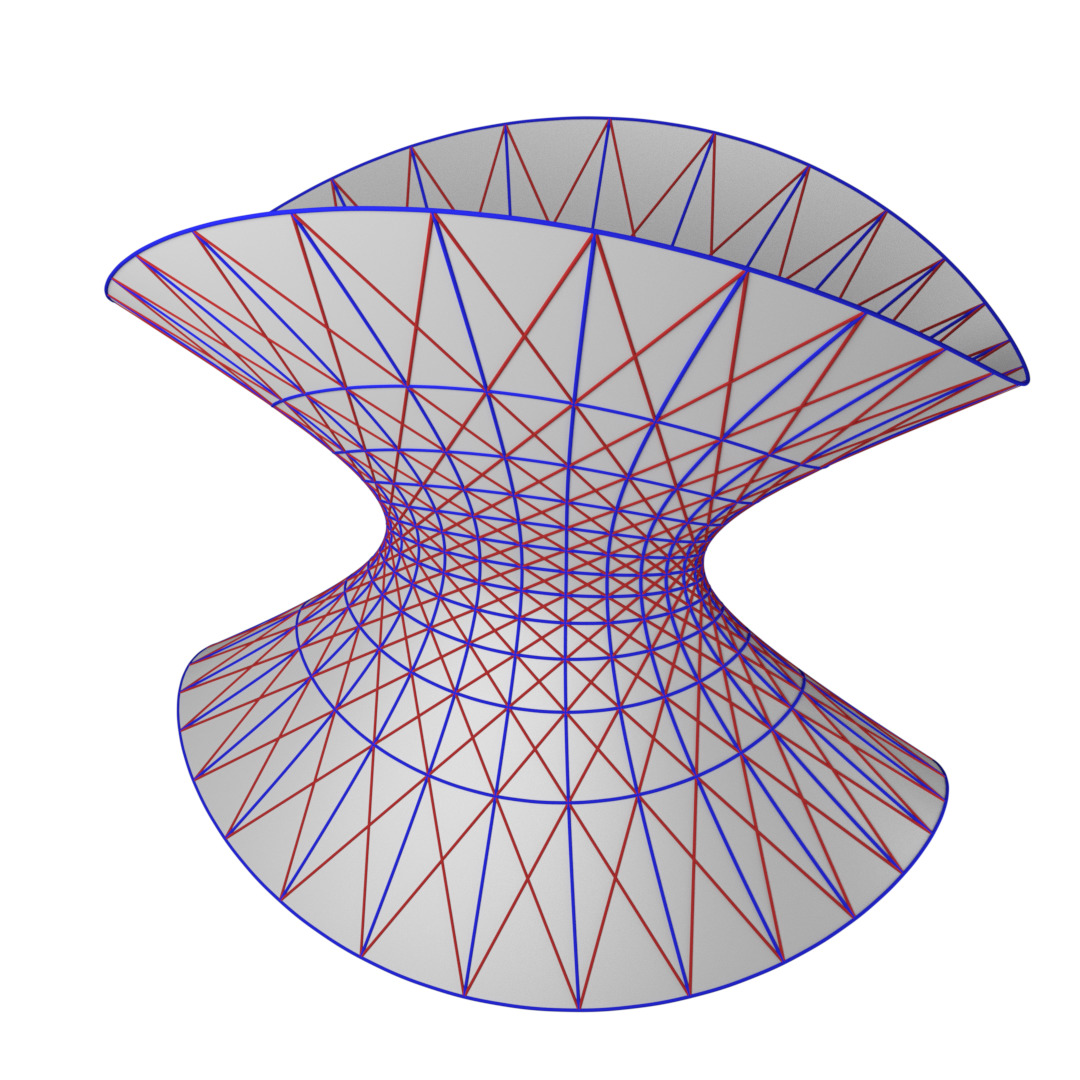}
  \includegraphics[scale=0.135]{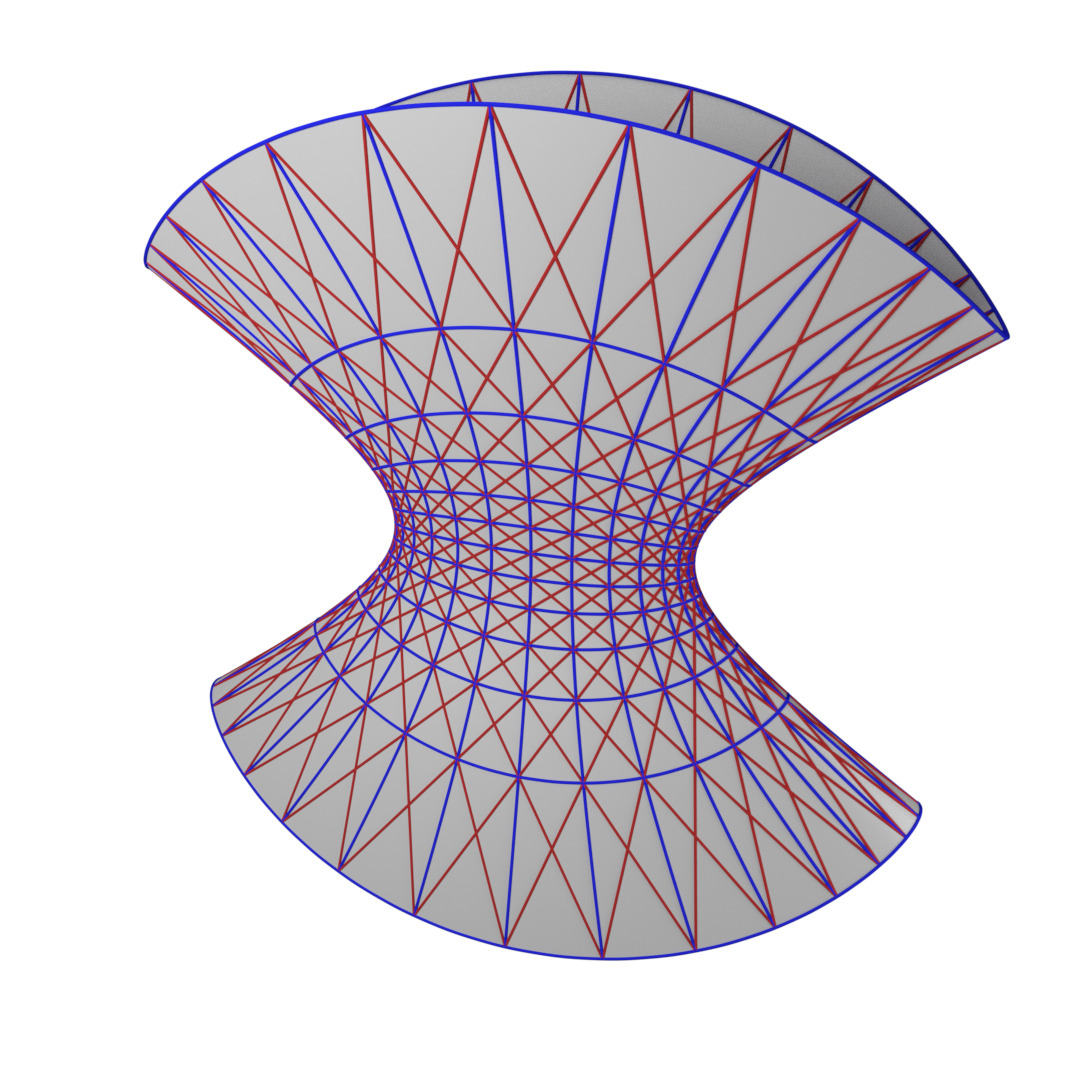}
  \includegraphics[scale=0.135]{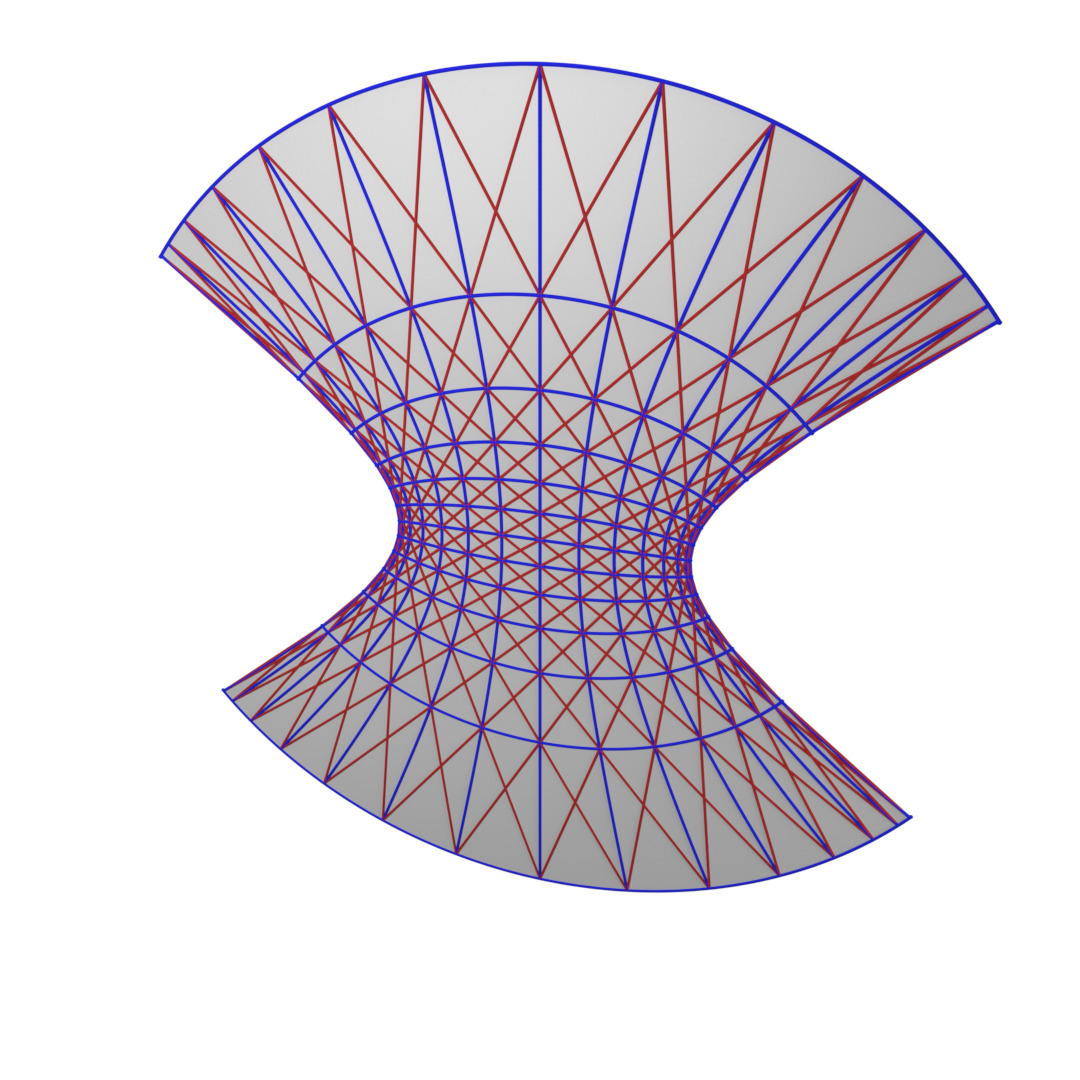}
\caption{The deformation of an AC-grid on a one-sheeted hyperboloid.}
\label{pic_deformhyp}
\end{figure}

\begin{remark}
It is observed that the special property of the quadrilaterals of AC-grids (cf.\ Remark \ref{touch}) implies that any AC-grid is associated with a one-parameter family of $\Z^2$-arrangements of spheres which are centred at the vertices and touch each other on the edges of the grid. During the deformation, touching spheres slide along each other but the relative position of aligned spheres does not change. In this manner, one obtains deformable models of one-sheeted hyperboloids composed of spheres (see Figure \ref{pic_spheres}). These, in turn, give rise to two associated models of one-sheeted hyperboloids composed of disks. On the one hand, each sphere may be replaced by the disk obtained by intersecting the sphere and the plane spanned by the pair of asymptotic lines passing through its centre. Any pair of neighbouring disks meet at the point on the asymptotic line where the corresponding pair of spheres touch. It is evident that the deformation of this model results in a rigid motion of the disks. An example of such a disk model is shown in Figure \ref{pic_spheres}. On the other hand, the quadrilaterals of any AC-grid may be represented by the disks bounded by the circles which pass through the sets of four points of contact of four touching spheres belonging to a quadrilateral (see Figure 9). Once again, the disks are linked at the points of contact of the spheres but this disk model is not deformable in the sense that the size of the disks changes as the quadrilaterals deform. 
\end{remark}

\begin{figure}
  \centering
  \includegraphics[scale=0.136]{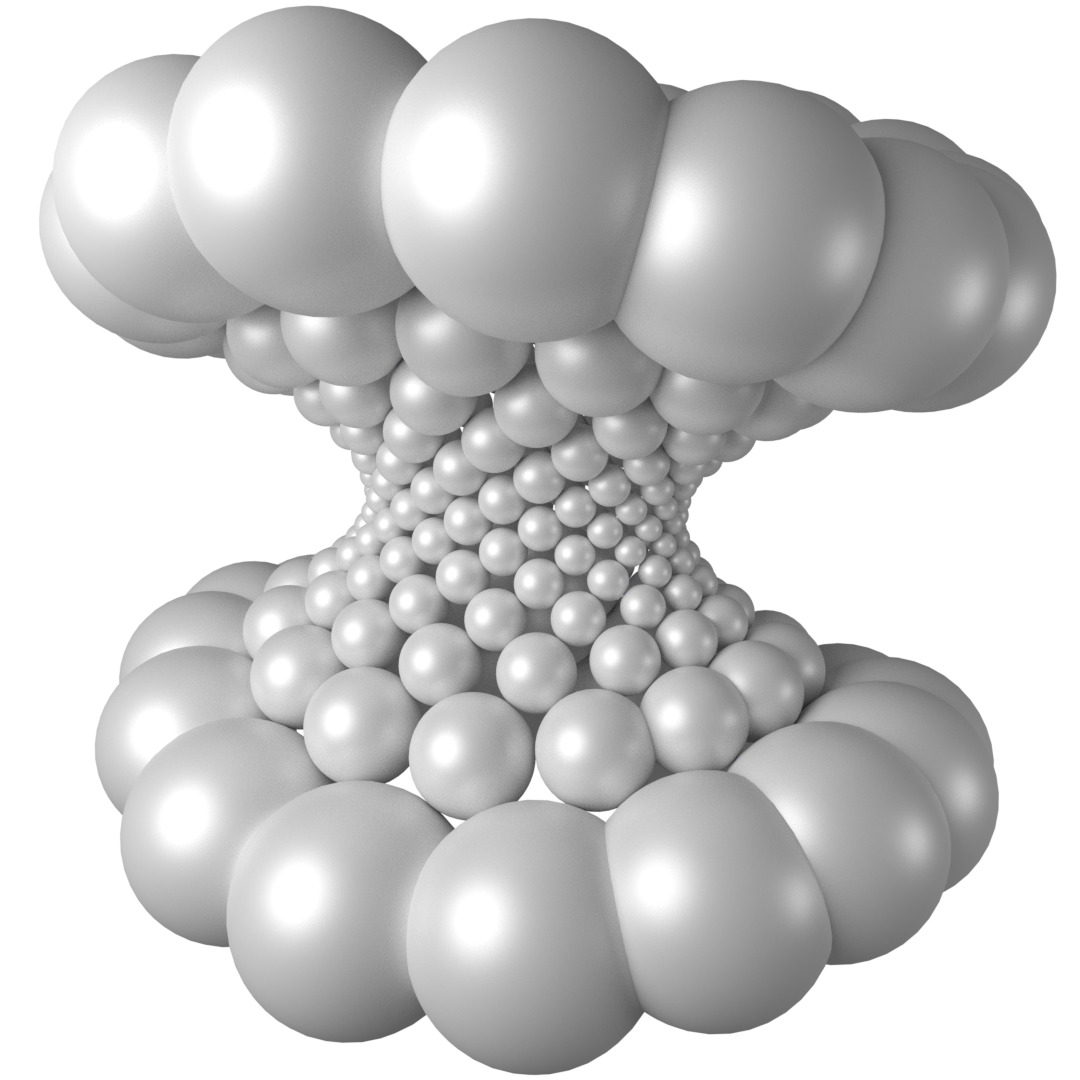}
  \includegraphics[scale=0.136]{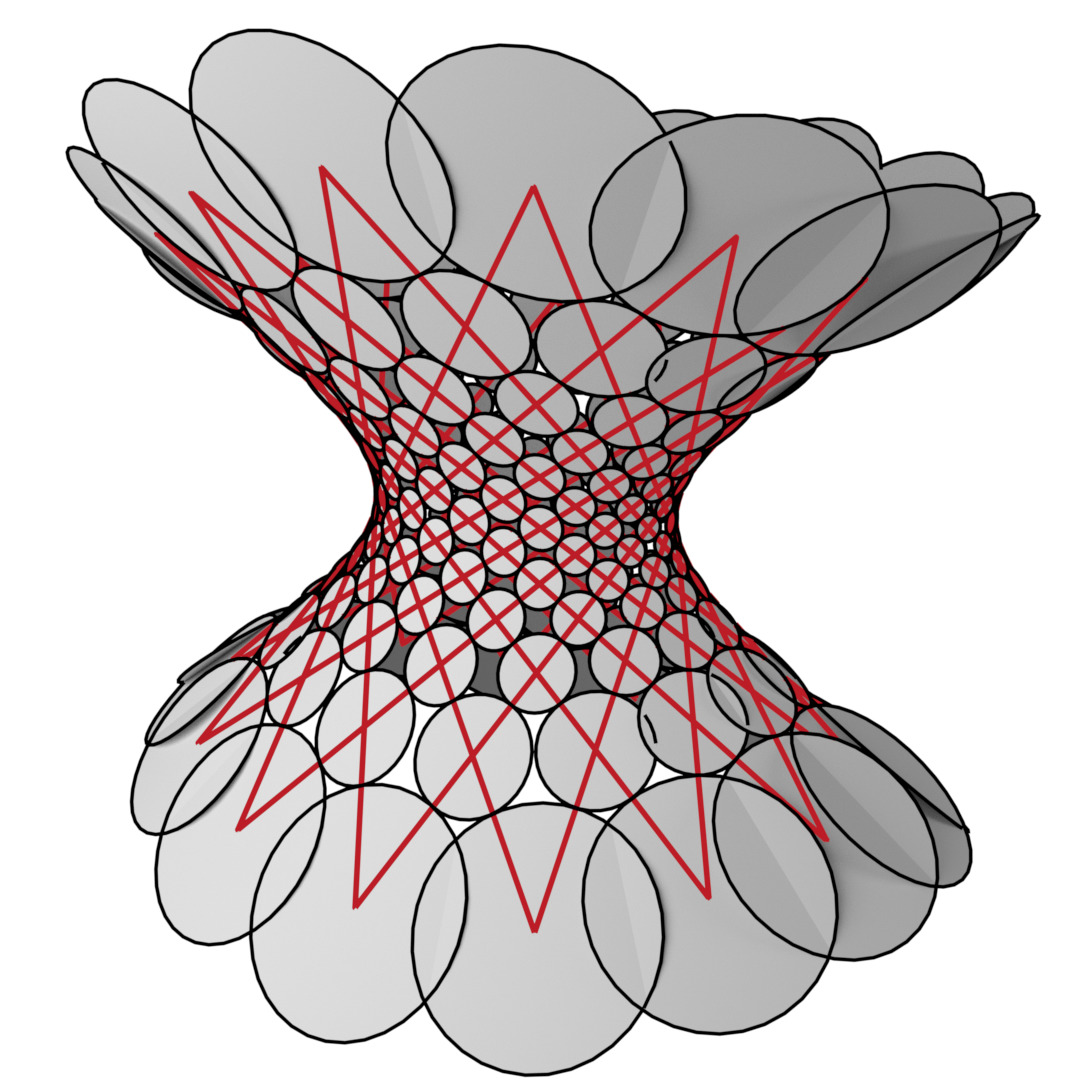}
  \includegraphics[scale=0.136]{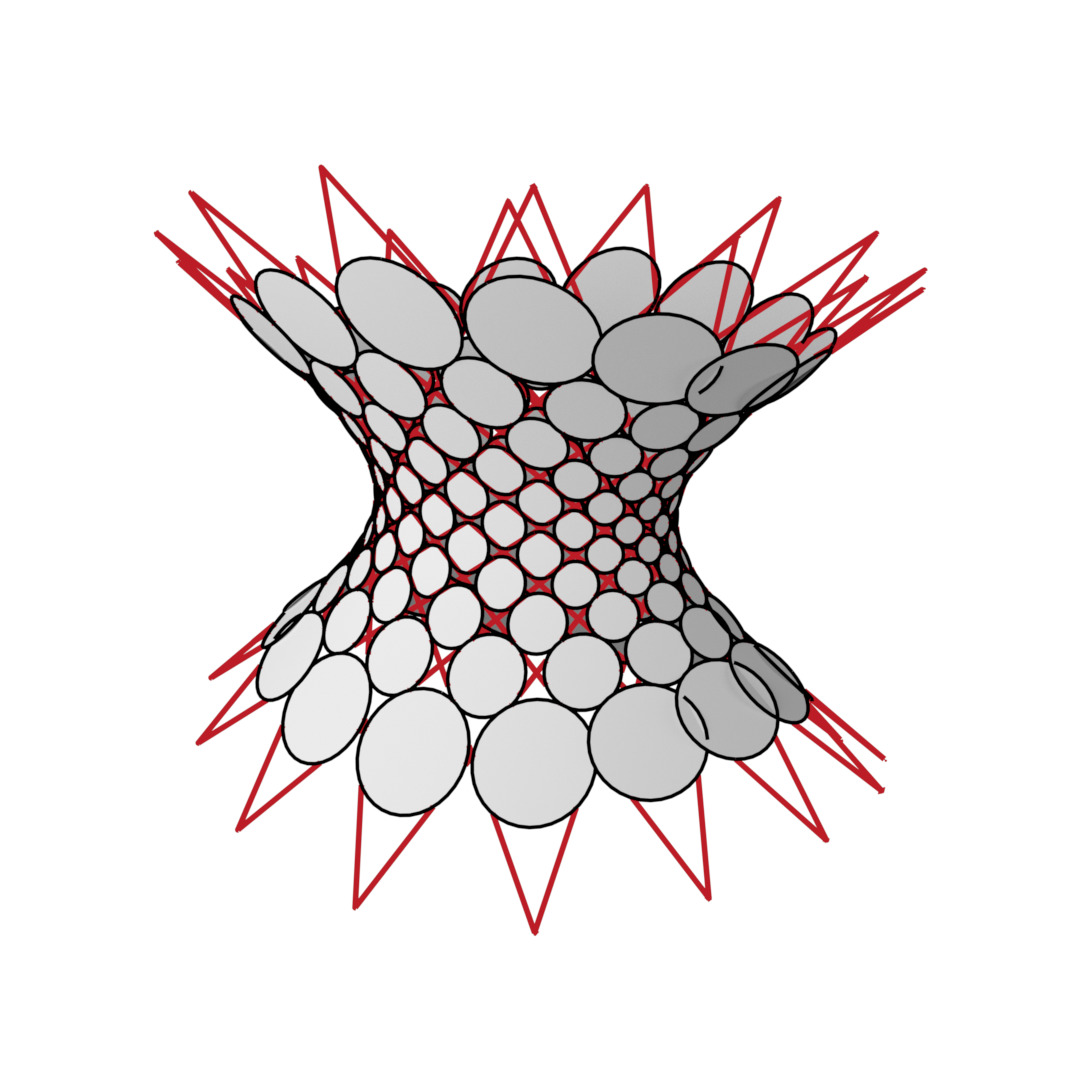}
  \caption{A sphere model and two disk models of the same one-sheeted hyperboloid.}
\label{pic_spheres}
\end{figure}


\section{Circular sections of (confocal) quadrics}

It is well known that any ellipsoid or one/two-sheeted hyperboloid possesses two one-parameter families of circular sections \cite{HCV52}. The circles of any such family lie in parallel planes. In the case of an ellipsoid or a two-sheeted hyperboloid, the planes are parallel to the tangent planes at the umbilic points of those quadrics (cf. Figure \ref{pic_sections}). In the following, we examine in detail the properties of the circular sections of both a single ellipsoid and its associated one-parameter family of confocal ellipsoids. The cases of one/two-sheeted hyperboloids may be treated in a similar manner. 

\begin{figure}
  \centering
  \raisebox{-0.5\height}{\includegraphics[scale=0.17]{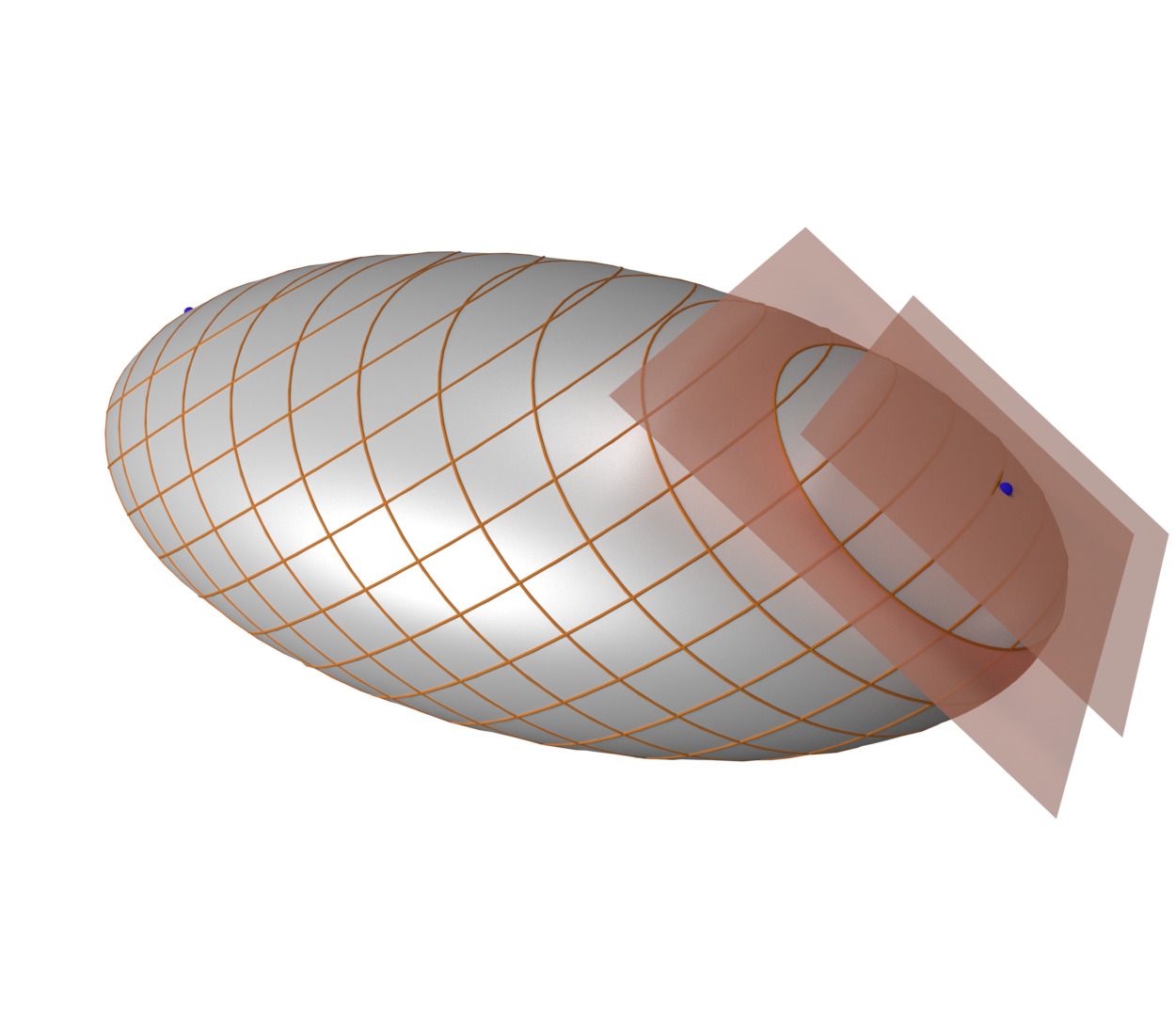}}
  \raisebox{-0.5\height}{\includegraphics[scale=0.17]{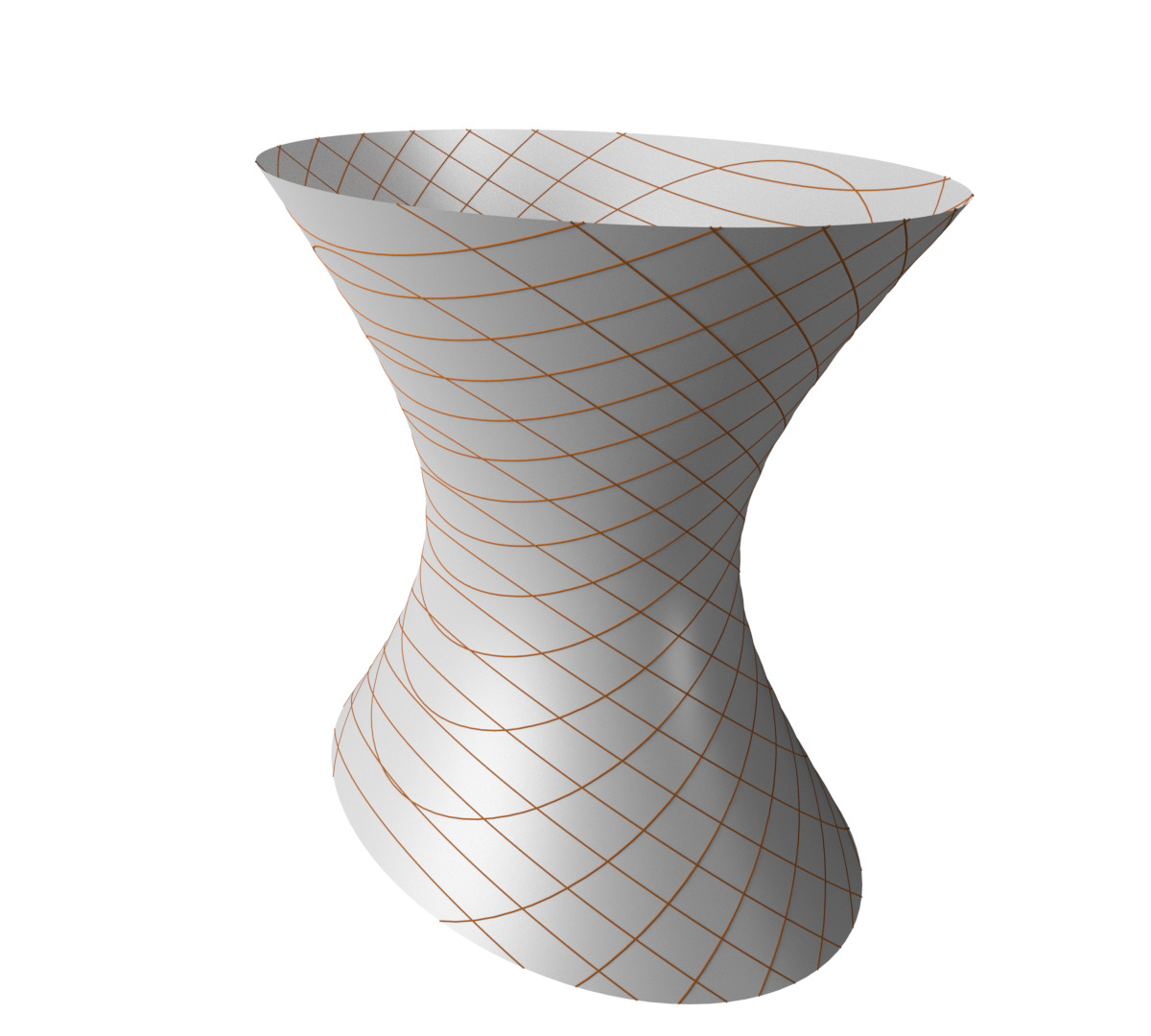}}
  \caption{Parallel circular sections of an ellipsoid and a one-sheeted hyperboloid.}
  \label{pic_sections}
\end{figure}

It turns out convenient to consider one-parameter families of ellipsoids which are obtained by uniformly scaling the ellipsoids of a system of confocal quadrics. Thus, we regard any (appropriately uniformly scaled) ellipsoid as being embedded in a corresponding one-parameter family of ellipsoids
\bela{E20}
    \frac{u_3+b}{u_3 + a}x^2 + y^2 + \frac{u_3+b}{u_3 + c}z^2 = 1
\ela
with the given ellipsoid being represented by a particular value of $u_3$ and suitable constants $a,b$ and $c$. By construction, the one-parameter family of ellipsoids is parametrised by
\bela{E21}
  x = \frac{f_1(s_1)f_2(s_2)\hat{f}_3(s_3)}{\sqrt{(a-b)(a-c)}},\quad
  y = \frac{g_1(s_1)g_2(s_2)}{\sqrt{(a-b)(b-c)}},\quad
  z = \frac{h_1(s_1)h_2(s_2)\hat{h}_3(s_3)}{\sqrt{(a-c)(b-c)}},
\ela
where 
\bela{E22}
  \hat{f}_3(s_3) = \frac{f_3(s_3)}{g_3(s_3)},\quad 
  \hat{h}_3(s_3) = \frac{h_3(s_3)}{g_3(s_3)}
\ela
subject to the functional relations \eqref{E10}. It is evident that $s_1$ and $s_2$ still parametrise the lines of curvature on the ellipsoids. In the case $u_3=-c$ corresponding to $h_3(s_3)=0$, the ellipsoid degenerates to an elliptic disk on the $(x,y)$-plane, while in the limit $u_3\rightarrow\infty$ corresponding to $\hat{f}_3(s_3)=\hat{h}_3(s_3)=1$, the ellipsoid becomes spherical. This corresponds to the spherical ellipsoid of infinite radius in any confocal family of ellipsoids. However, if we set aside the genesis of the parametrisation \eqref{E21} then we may ignore the constraint $u_3>-c$ and replace the functional relations \eqref{E10}$_{5,6}$ by their algebraic implication
\bela{E23}
  (b-c)\hat{f}^2_3(s_3) + (a-b)\hat{h}^2_3(s_3) = a-c.
\ela
In this manner, we extend the one-parameter family of ellipsoids beyond $u_3=\infty$, leading, in particular, to a second degeneration to an elliptic disk in the plane $x=0$ associated with $\hat{f}_3=0$. Accordingly, it is natural to introduce the parametrisation
\bela{E24}
  \hat{f}_3(s_3) = \sqrt{\frac{a-c}{b-c}}\cos s_3,\quad
  \hat{h}_3(s_3) = \sqrt{\frac{a-c}{a-b}}\sin s_3,\qquad 0\leq s_3\leq\frac{\pi}{2}
\ela
with $s_3=0$, $s_3=\pi/2$ and $s_3=\arctan\sqrt{\frac{a-b}{b-c}}$ representing the planar and spherical cases respectively.

In connection with the circular sections of the ellipsoids, the parametrisation
\bela{E25}
 \begin{aligned}
  f_1(s_1) &= \sqrt{a-c}\,\cos s_1,\quad & g_1(s_1) & = \pm\sqrt{a-b - (a-c)\cos^2 s_1},
  \quad &  h_1(s_1) & =  \sqrt{a-c}\,\sin s_1\\
  f_2(s_2) &= \sqrt{a-c}\,\cos s_2,\quad & g_2(s_2) & = \pm\sqrt{(a-c)\cos^2 s_2 - a+b},
  \quad & h_2(s_2) & =  \sqrt{a-c}\,\sin s_2
 \end{aligned}
\ela
of the lines of curvature which resolves the functional relations \eqref{E10}$_{1,2,3,4}$ turns out to be canonical. Here, the parameters $s_1$ and $s_2$ are constrained by
\bela{E26}
 s_0 \leq s_1 \leq \pi - s_0,\quad |s_2| \leq s_0,\qquad s_0= \arccos\sqrt{\frac{a-b}{a-c}}.
\ela
In particular, we obtain the projection
\bela{E27}
  \left(\bear{c}x\\ z\ear\right) = \frac{a-c}{\sqrt{(a-b)(b-c)}}
  \left(\bear{c}\cos s_1\cos s_2\cos s_3\\ \sin s_1\sin s_2\sin s_3\ear\right).
\ela
Accordingly, the net
\bela{E28}
  \alpha = s_1 + s_2 = \mbox{const},\quad \beta = s_1 - s_2 = \mbox{const}
 \ela
on any fixed ellipsoid, which is bisected by the lines of curvature, consists of planar curves which lie in two families of parallel planes since \eqref{E27} may be formulated as
\bela{E29}
  \left(\bear{c}x\\ z\ear\right) = \frac{a-c}{2\sqrt{(a-b)(b-c)}}
  \left(\bear{c}(\cos \beta+\cos\alpha)\cos s_3\\ (\cos \beta - \cos \alpha)\sin s_3\ear\right),
\ela
leading to, for instance,
\bela{E29a}
  x \sin s_3 + z \cos s_3 = \frac{(a-c)\cos\beta}{\sqrt{(a-b)(b-c)}}\sin s_3 \cos s_3.
\ela
In fact, as proven below, these planar curves are precisely the classical circular sections of the ellipsoid. In the following, we refer to the lines $\alpha=\mbox{const}$ and $\beta=\mbox{const}$ as the {\em circular lines} on the ellipsoid.

\begin{theorem}\label{ellipsoid}
On any ellipsoid, the lines of curvature and the circular lines form mutually diagonal nets. If the (uniformly scaled) ellipsoid is parametrised as in the preceding then the radii of the circles $\alpha=\mbox{const}$ and $\beta=\mbox{const}$ are given by $|\sin\alpha|$ and $|\sin\beta|$ respectively.
\end{theorem}

\begin{proof} 
On use of the parametrisation \eqref{E21}, \eqref{E24} and \eqref{E25}, it is straightforward to show that the curvature of the $\alpha$-lines is given by
\bela{E30}
  \frac{|\br_\alpha\times\br_{\alpha\alpha}|}{|\br_\alpha|^3}  = \frac{1}{|\sin\beta|}
\ela
so that the planar curve $\beta=\mbox{const}$ is indeed a circle of radius $|\sin\beta|$.
\end{proof}

As in the previous section, we may now regard the one-parameter family of ellipsoids
\bela{E31}
  \left(\frac{b-c}{a-c}\right)\frac{x^2}{\cos^2s_3} + y^2 + \left(\frac{a-b}{a-c}\right)\frac{z^2}{\sin^2s_3} = 1
\ela
as being generated by deforming any of its members, where $s_3$ constitutes the deformation parameter. In order to describe the limiting behaviour of the circular lines as $s_3\rightarrow0$ or $s_3\rightarrow\frac{\pi}{2}$, we need to introduce an appropriate notion of tangency. This is due to the fact that not all circular lines on the ellipsoids cross the ellipse in the plane $z=0$ or $x=0$ respectively. Thus, we say that two conics ``touch'' each other if two of the four possibly complex points of intersection coincide. Now, we are in a position to summarise the properties of the deformation of the ellipsoids in the following theorem.

\begin{theorem}\label{deformationcircles}
The deformation which generates the uniformly scaled and extended one-parameter family of confocal ellipsoids \eqref{E31} enjoys the following properties.
\begin{itemize}
\item Lines of curvature and circular lines and their mutually diagonal relationship are preserved.
\item The deformation is isometric along the circular lines, that is, the distance between any two points on and along any of the circular sections is preserved.
\item In the two limiting cases $s_3=0$ and $s_3=\pi/2$, the mutually diagonal nets of lines of curvature and circular lines become planar with the circles touching an ellipse and the lines of curvature becoming associated confocal conics. The two planar pairs of nets of circles and confocal conics are isometric along the circles in the above sense. 
\end{itemize}
\end{theorem}

\begin{proof}
Firstly, the relevant parts of the proof of Theorem \ref{deformation} apply {\it mutatis mutandis} in the current situation. Secondly, the independence of the radii of the circular sections on the deformation parameter $s_3$ as stated in Theorem \ref{ellipsoid} is a first indication that the deformation is isometric along the circular lines. In fact, one may directly verify that
\bela{E32z}
 \br_{\alpha}^2 = \br_{\beta}^2 = \frac{(a-c)^2[\cos^2s_1 - \cos^2s_2]^2}{4[a-b - (a-c)\cos^2s_1][(a-c)\cos^2 s_2 - a+b]}
\ela
which is likewise independent of $s_3$ and, hence, the asserted ``isometry'' has been proven.

Thirdly, in the limit $s_3=0$, the ellipsoids degenerate to an elliptic disk on the plane $z=0$ bounded by the ellipse 
\bela{E32a}
  \frac{b-c}{a-c}x^2 + y^2 = 1
\ela
and the circular lines $\beta=\mbox{const}$, for instance, become the circles
\bela{E32b}
 \left(x - \sqrt{\frac{a-b}{b-c}}\cos\beta\right)^2 + y^2 = \sin^2\beta.
\ela
The coordinates of the points of intersection of these two conics are given by
\bela{E32c}
  x = \frac{(a-c)\cos\beta}{\sqrt{(a-b)(b-c)}},\quad y^2 = 1 - \frac{a-c}{a-b}\cos^2\beta
\ela
so that there exist two pairs of coinciding points of intersection. Hence, the circles touch the ellipse but the $y$-coordinate of the points of contact becomes purely imaginary for sufficiently small $\beta$.

Finally, even though the case $s_3=\pi/2$ cannot be directly associated with a degenerate ellipsoid of the underlying family of confocal ellipsoids, it is evident that this case may be treated by formally considering the symmetry $(x,z,a,c)\rightarrow(z,x,c,a)$, leading to the case $s_3=0$. 
\end{proof}

Once again, the theory established in the preceding may now be applied to appropriate samplings of lines of curvature and circular lines on ellipsoids. 

\begin{definition}
A configuration of the combinatorics of a $\Z^2$-grid of two sequences $\left(\ell_n\right)_{n\in\Z}$ and $\left(m_{n'}\right)_{n'\in \Z}$ of the two one-parameter families of circular lines on an ellipsoid is termed a {\em CC-grid} if pairs of opposite vertices of the elementary quadrilaterals formed by the circles $\ell_n,\ell_{n+1}$ and $m_{n'},m_{n'+1}$ are connected by lines of curvature.
\end{definition}

The analogue of IC-nets in hyperbolic geometry has been introduced in \cite{AB}. Even though, in the current context, we adopt a Euclidean point of view, the following definition of these nets immediately lends itself to an interpretation in terms of the Poincar\'e half-plane model of hyperbolic geometry.

\begin{definition}
A planar configuration of circles of the combinatorics of a $\Z^2$-grid is termed an {\em HIC-net} if the circles admit a common axis of symmetry and the elementary quadrilaterals formed by the circles circumscribe (pairs of) circles (see Figure \ref{pic_flatccgrid}).
\end{definition}


We conclude that CC-grids are non-planar generalisations of (particular) HIC-nets. The reasoning is analogous to that provided in the case of AC-grids. 

\begin{corollary}
The deformation linking the family of ellipsoids as detailed in Theorem \ref{deformationcircles} acts on associated CC-grids as follows. 
\begin{itemize}
\item CC-grids are preserved by the deformation.
\item During the deformation, any two intersecting circles undergo a relative rotation about the line passing through the two points of intersection.
\item CC-grids become HIC-nets in the planar limiting cases $s_3=0$ and $s_3=\pi/2$ (cf.\ Figure \ref{pic_flatccgrid}). The two HIC-nets are isometric along the circles.
\item CC-grids and their deformation are algebraically represented by the privileged parametrisation \eqref{E21}, \eqref{E24}, \eqref{E25} with the vertices of the grids being parametrised by $(s_1,s_2) = \delta(n_1+n_2,n_1-n_2) + (s_1^0,s_2^0)$, $n_1,n_2\in\Z$. 
\end{itemize}
\end{corollary}

\begin{figure}
  \centering
  \includegraphics[scale=0.18]{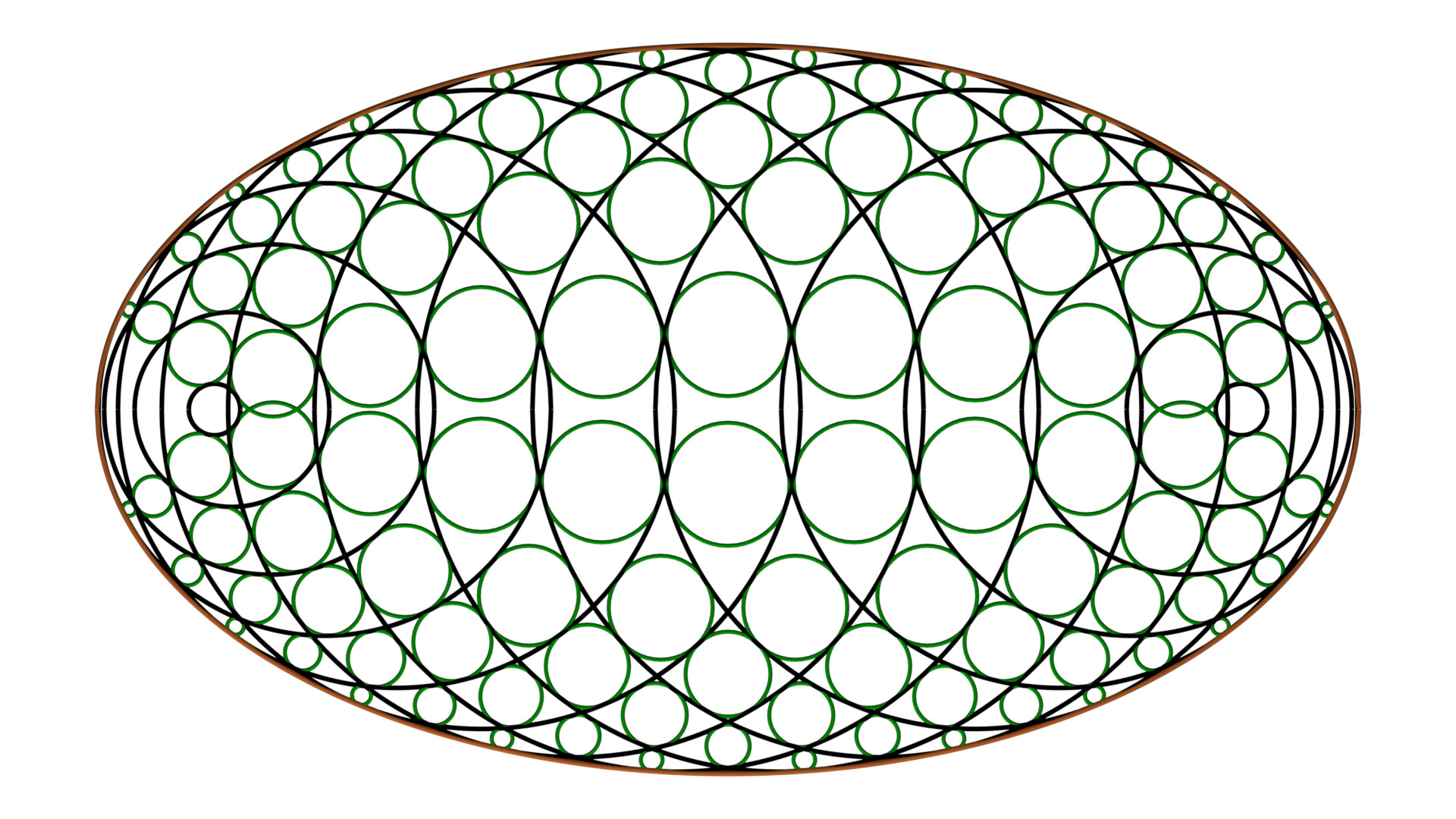}
  \caption{The constituent circles of a planar CC-grid (HIC-net) and the associated inscribed circles.}
  \label{pic_flatccgrid}
\end{figure}

\begin{proof}
It remains to demonstrate that the quadrilaterals of the planar limits of the CC-grids admit incircles. As an application of Lie circle geometry, one may show \cite{BobenkoSchief2017} that four oriented circles of radii $r_i$ and centres $(x_i,0)$ admit an oriented incircle if and only if 
\bela{E32d}
\left|\bear{cccc} 1 & x_1 & r_1 & x_1^2 - r_1^2\as
                         1 & x_2 & r_2 & x_2^2 - r_2^2\as
                         1 & x_3 & r_3 & x_3^2 - r_3^2\as
                         1 & x_4 & r_4 & x_4^2 - r_4^2
\ear\right| = 0,
\ela
where the orientation of the circles is encoded in the signs of $r_i$. Here, the assumption is made that the two incircles of any triple of circles exist in the sense that these are real. In the case $s_3=0$ and two pairs of circles $\alpha=\mbox{const}$ and $\beta=\mbox{const}$, this translates into
\bela{E32e}
\left|\bear{llll} 1 & \cos\alpha_+ & \epsilon_1\sin\alpha_+ & \cos 2\alpha_+\\
                        1 & \cos\alpha_- & \epsilon_2\sin\alpha_- & \cos 2\alpha_-\\
                        1 & \cos\beta_+ & \epsilon_3\sin\beta_+ & \cos 2\beta_+\\
                        1 & \cos\beta_- & \epsilon_4\sin\beta_- & \cos 2\beta_-
\ear\right| = 0,
\ela
where $\epsilon_i^2=1$, by virtue of \eqref{E32b} and its symmetric counterpart. Now, any quadrilateral of a planar CC-grid is composed of circles labelled by
\bela{E32f}
  \alpha_{\pm} = \alpha_0 \pm \delta,\quad \beta_{\pm} = \beta_0 \pm \delta
\ela
so that it may be verified that \eqref{E32e} is indeed satisfied for any choice of $\alpha_0,\beta_0$ and $\delta$. Here, the orientation of the circles has to be chosen such that $\epsilon_1=-\epsilon_2=-\epsilon_3=\epsilon_4$. 
\end{proof}

\begin{remark}
In \cite{HCV52}, it is asserted that one may construct a deformable model of an ellipsoid by realising arbitrary samplings of circular sections as a collection of ``rings'' (or ``slitted'' disks) which are fastened at the points of intersection
(see Figure \ref{circular_wien}). 
\begin{figure}
  \centering
  \includegraphics[scale=0.294]{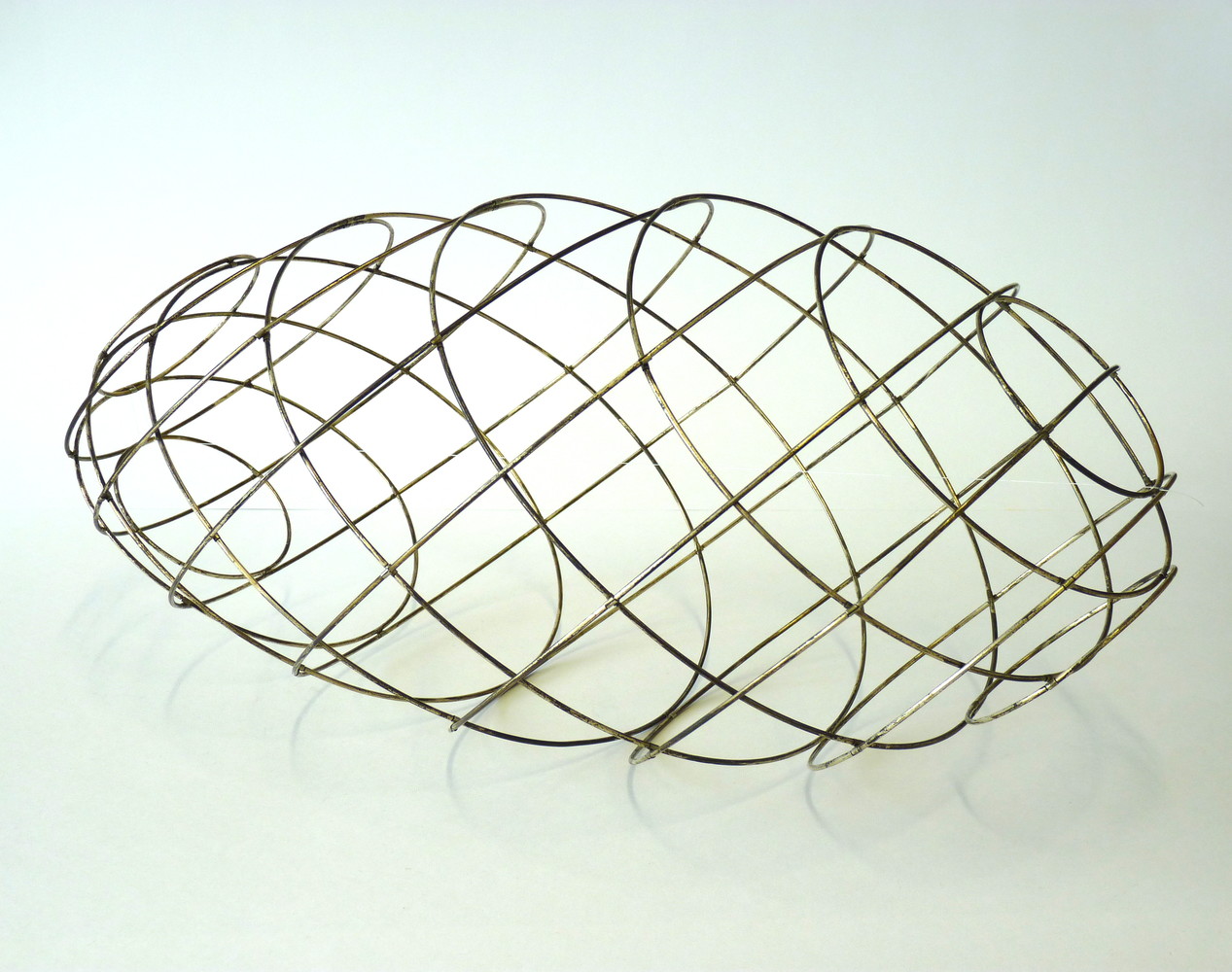}
  \includegraphics[scale=0.294]{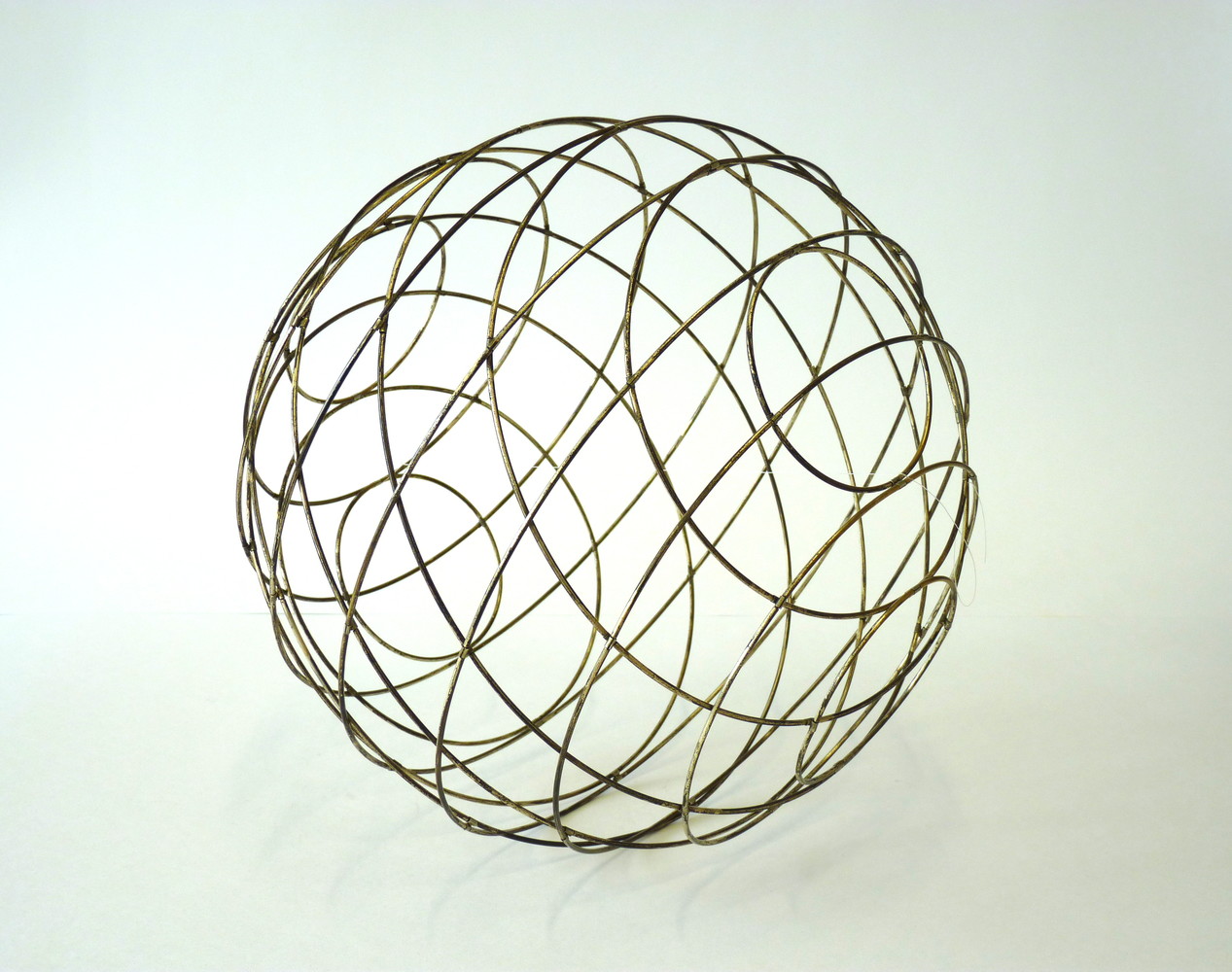}
  \includegraphics[scale=0.294]{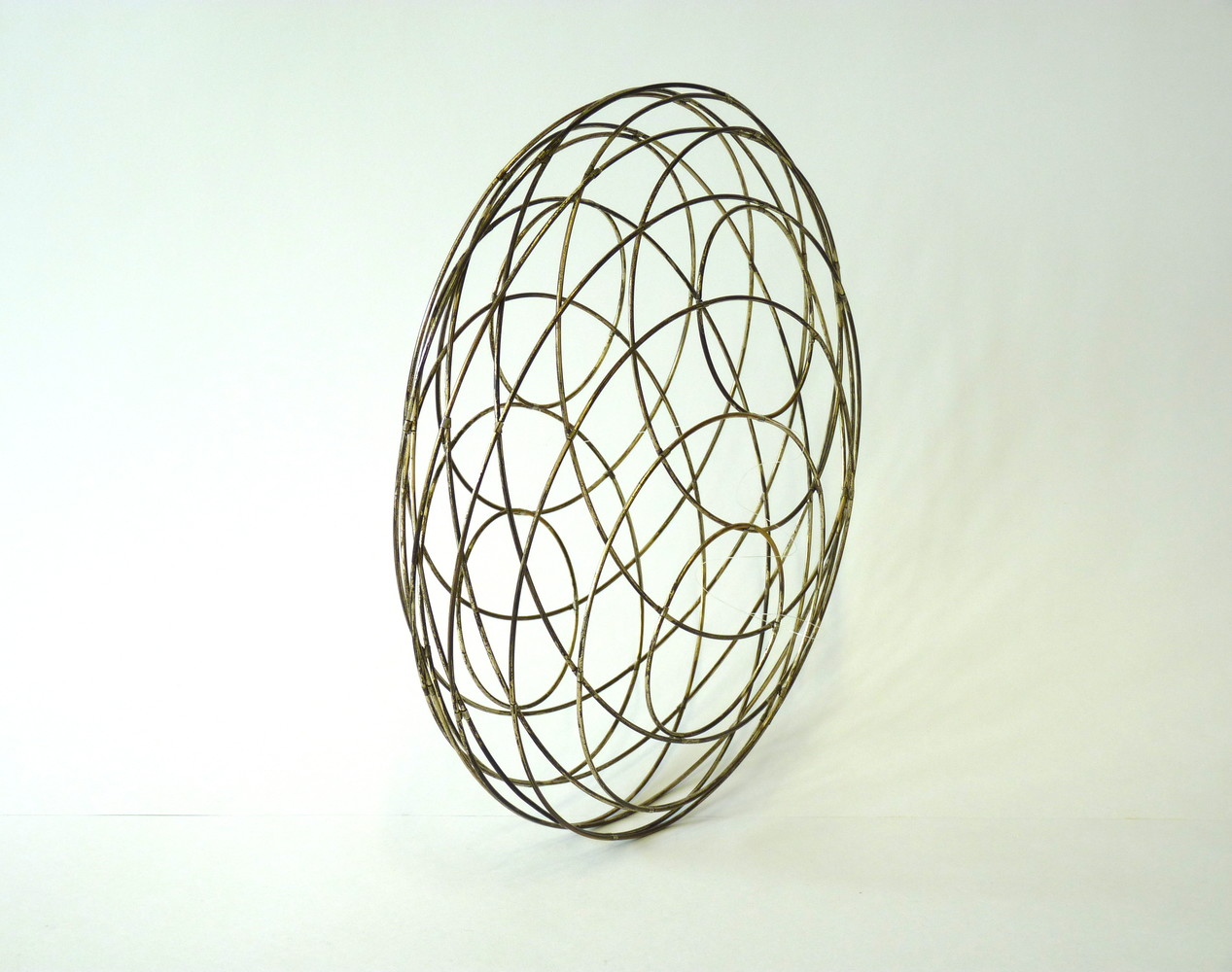}
  \caption{Photos of an ``isometrically'' deformable model of an ellipsoid from the collection of the Institute of Discrete Mathematics and Geometry, TU Wien.}
  \label{circular_wien}
\end{figure}
In particular, the special cases of planar and spherical configurations which occur during the deformation are mentioned. The discussion in the preceding provides an explicit representation and proof of this assertion. It also establishes the existence of CC-grids as privileged samplings. In Figure \ref{pic_deformell}, various stages in the deformation of a CC-grid are displayed. Figure \ref{pic_deformell} (top-left) and Figure \ref{pic_deformell} (bottom-right) constitute HIC-nets, while Figure \ref{pic_deformell} (bottom-left) represents a spherical CC-grid. It is observed that HIC-nets may also be obtained via planar degenerations of CC-grids on one- or two-sheeted hyperboloids with the conics of tangency being ellipses or hyperbolae.

\begin{figure}
  \centering
  \includegraphics[scale=0.13]{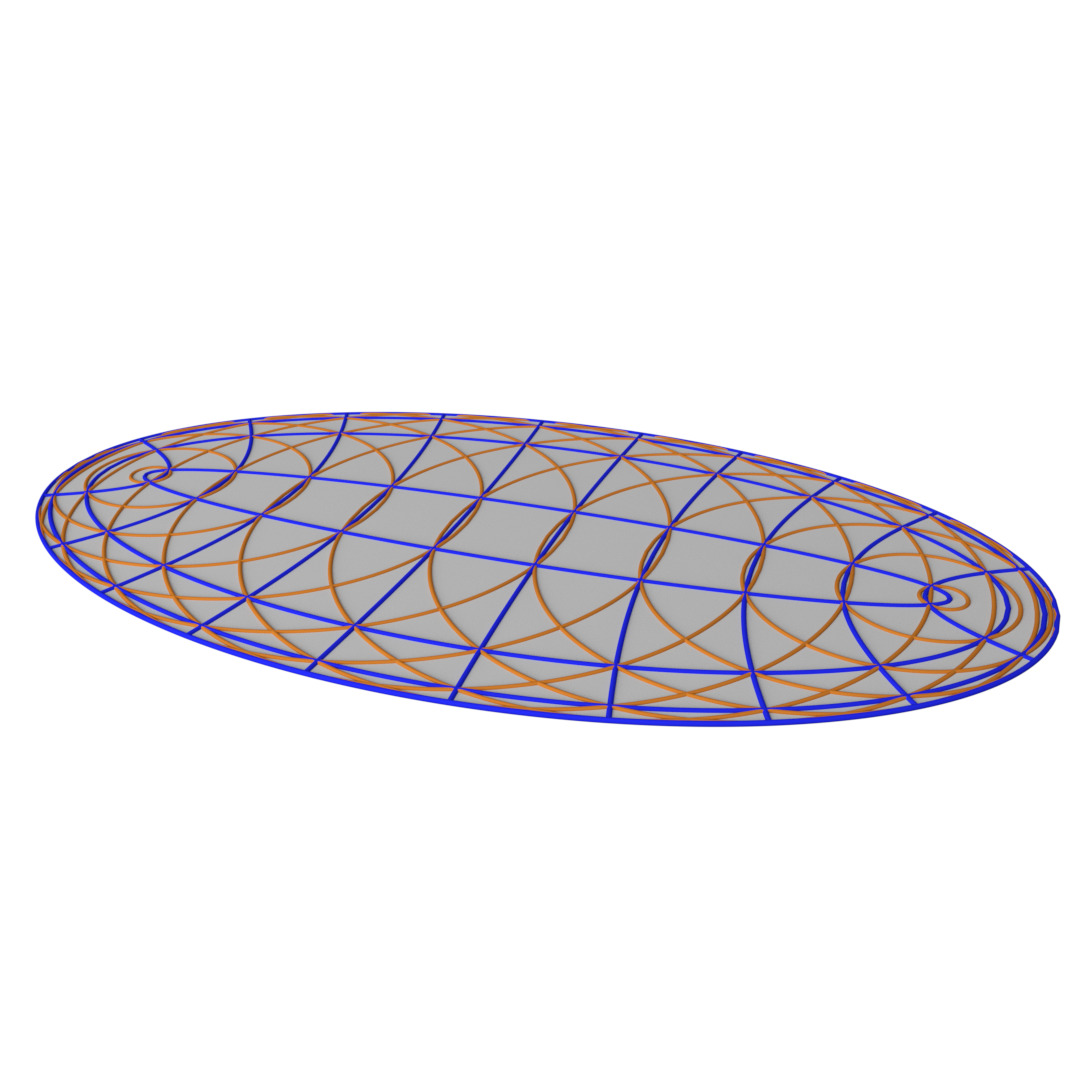}
  \includegraphics[scale=0.13]{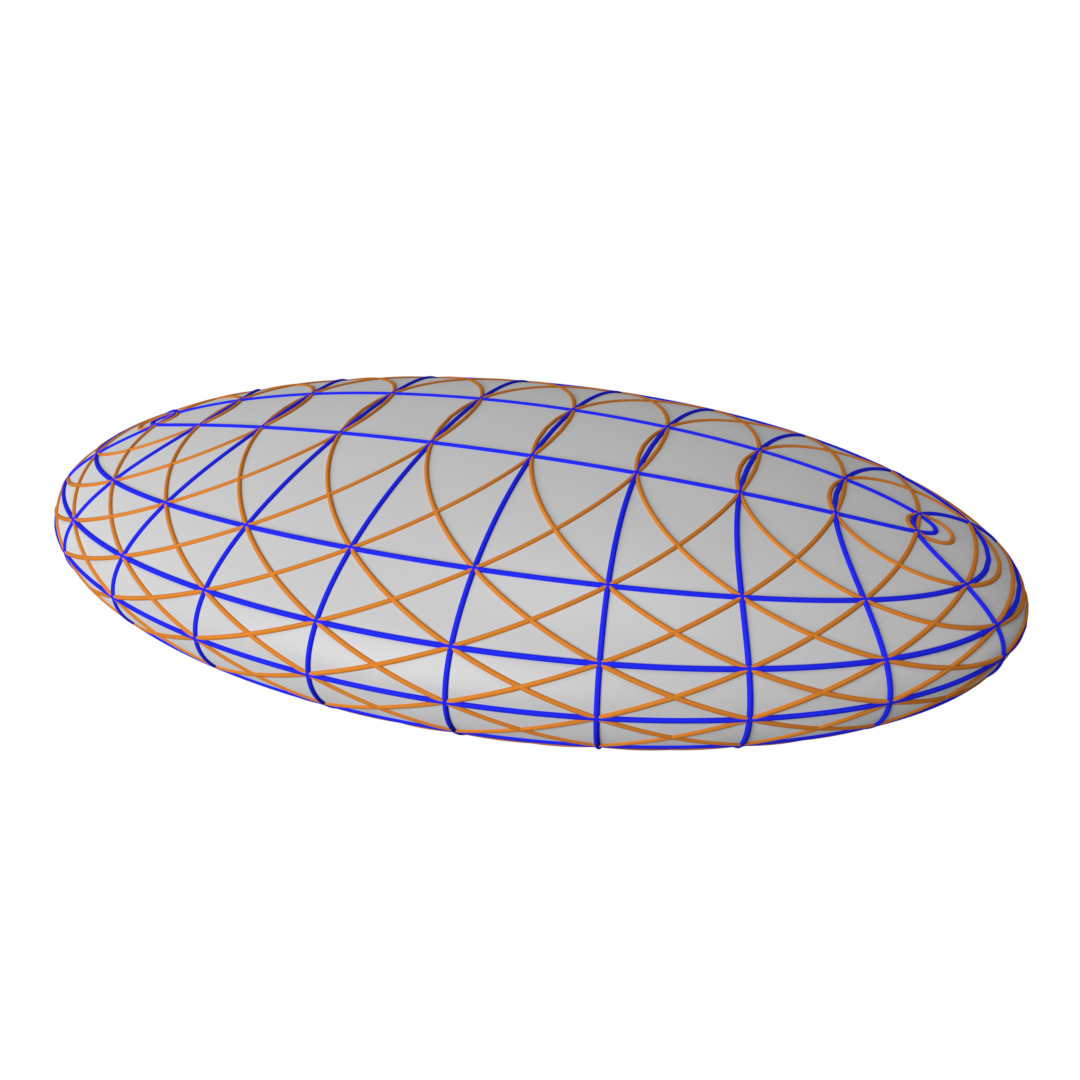}
  \includegraphics[scale=0.13]{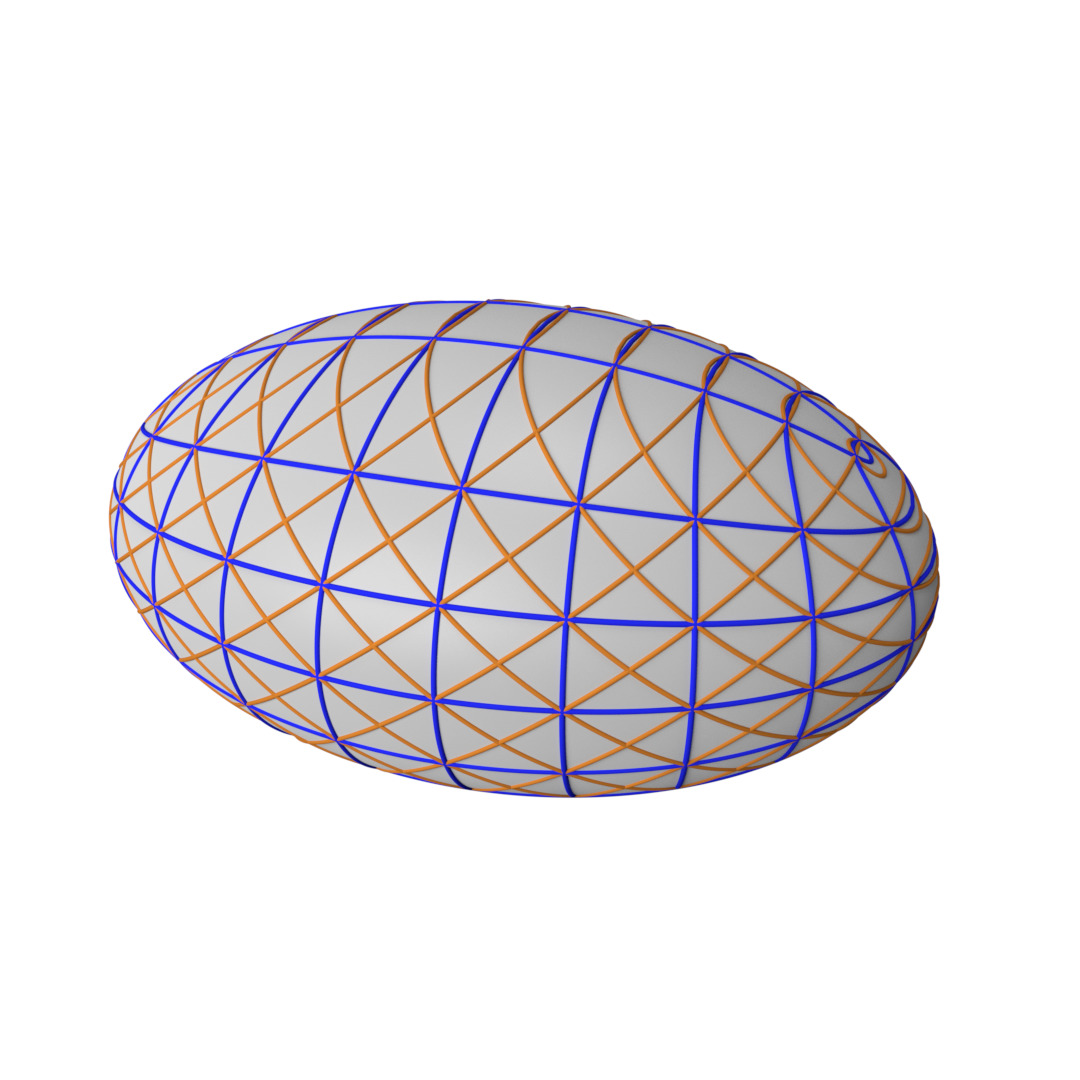}
  \includegraphics[scale=0.13]{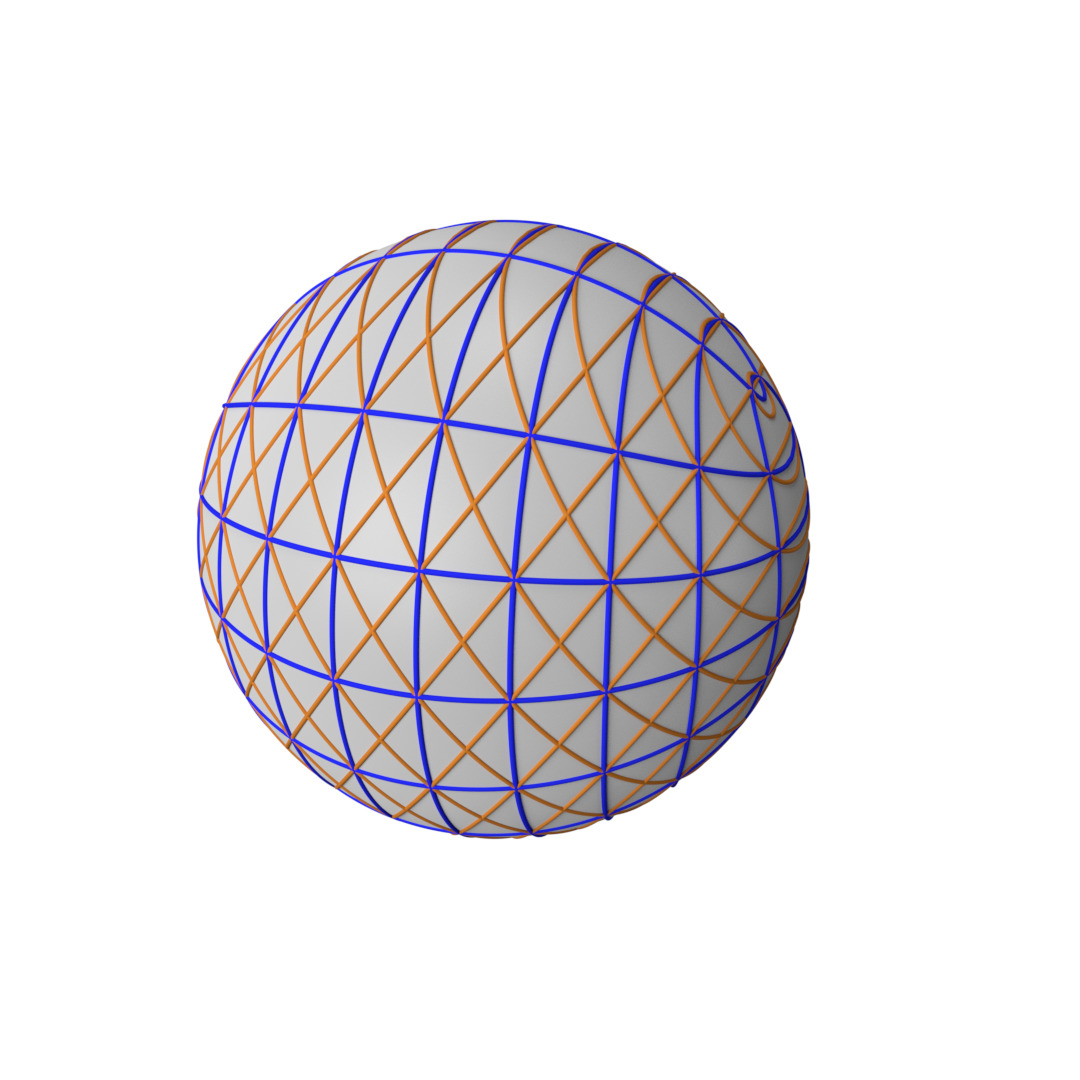}
  \includegraphics[scale=0.13]{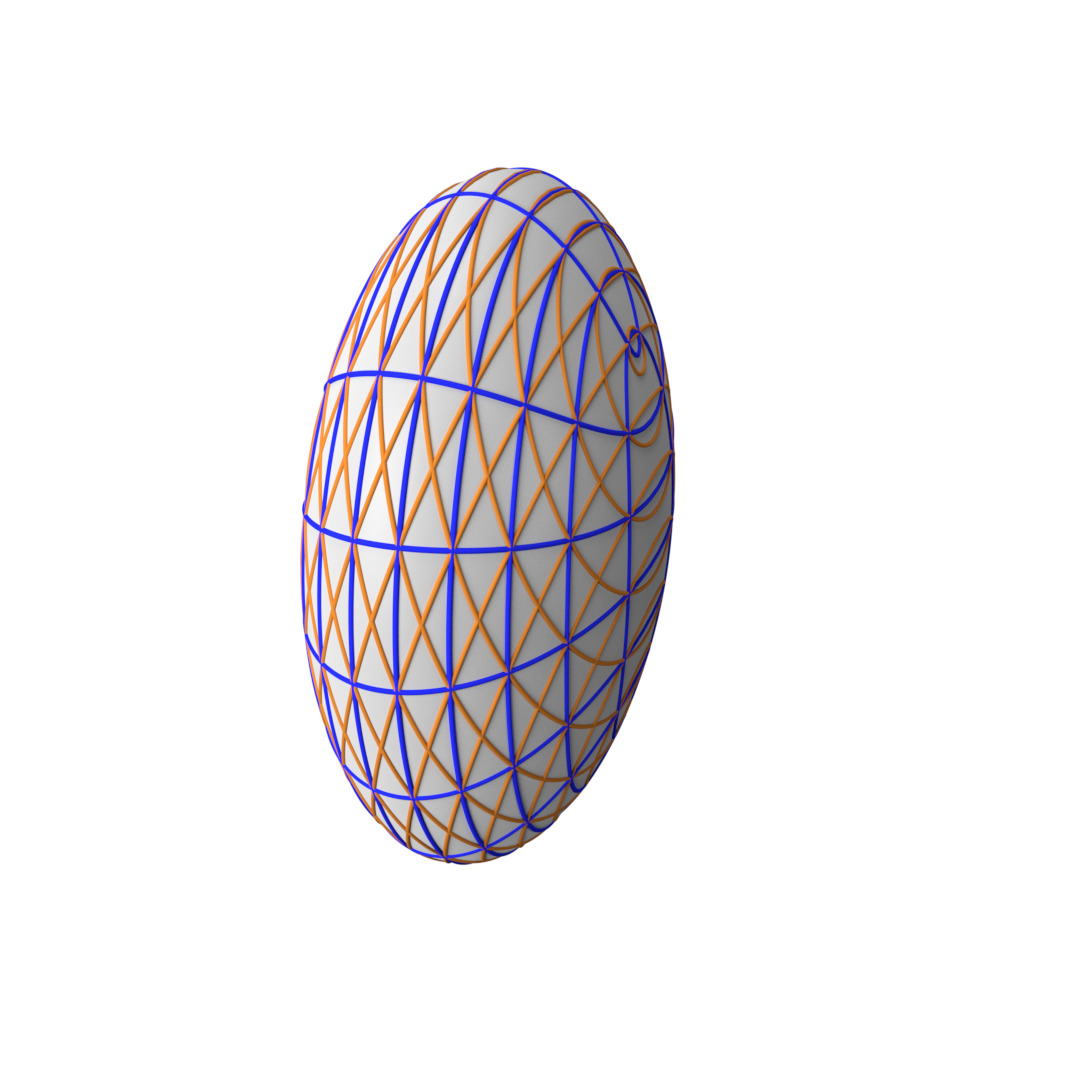}
  \includegraphics[scale=0.13]{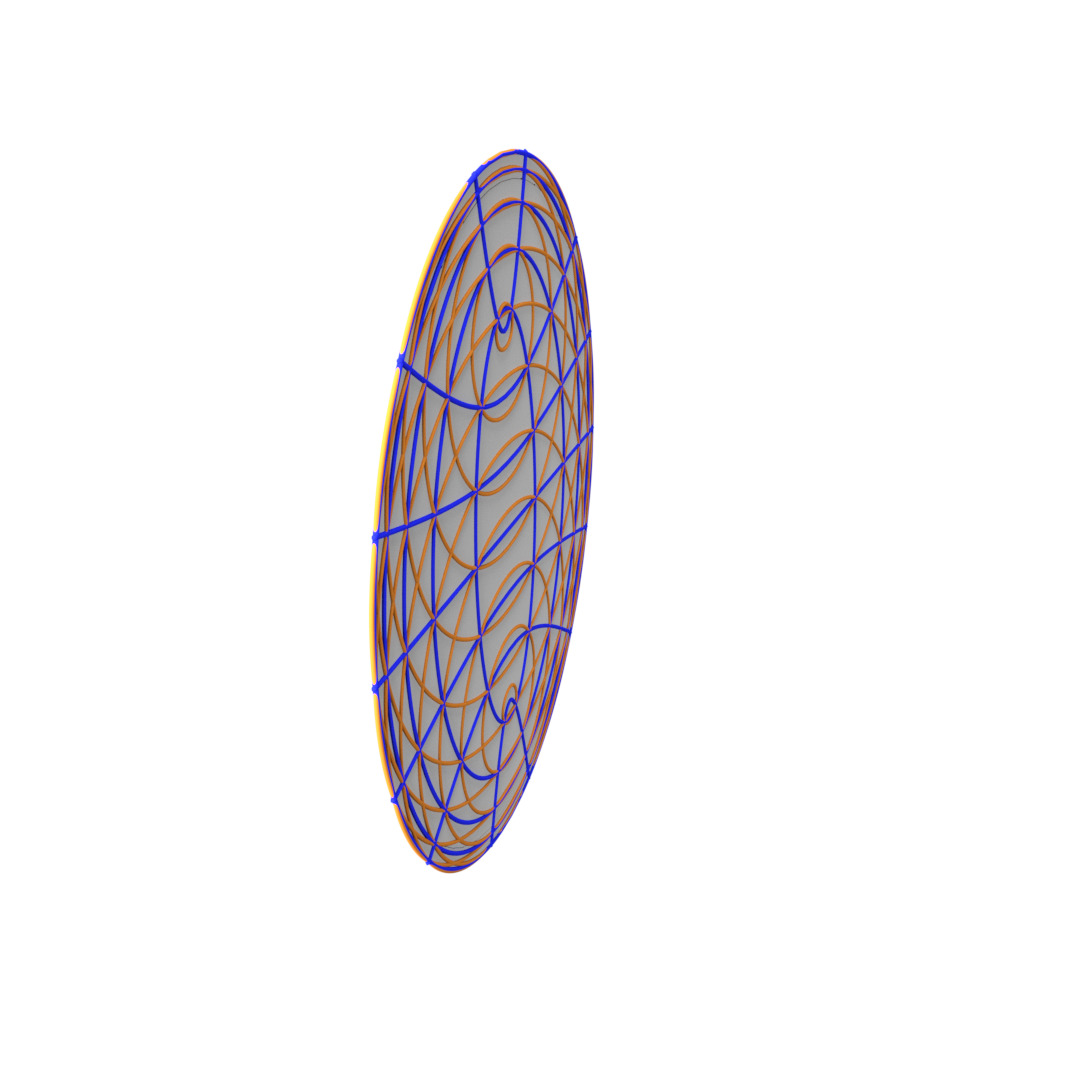}
  \caption{The deformation of a CC-grid on an ellipsoid. The CC-grid in the bottom-left image is spherical.}
  \label{pic_deformell}
\end{figure}
\end{remark}


\section{Characteristic lines on (confocal) ellipsoids and two-sheeted hyperboloids}

At each point of a surface, there exists a one-parameter family of pairs of conjugate directions of which one pair is privileged. The following defining properties of so-called {\em characteristic (conjugate) directions} are equivalent \cite{Eisenhart1960} and the {\em characteristic lines} are the integral curves of the characteristic direction fields.

\begin{itemize}
\item The angle between the conjugate directions is minimal or, equivalently, maximal.
\item Curvature lines bisect the conjugate directions.
\item The normal curvatures in the conjugate directions coincide.
\end{itemize}

In the same manner that asymptotic lines on a surface of negative Gaussian curvature are real and unique, characteristic lines are likewise real and unique on surfaces of positive Gaussian curvature.
Since, according to Theorem \ref{parametrisation}, the lines of curvature on the ellipsoids and two-sheeted hyperboloids of a confocal system of quadrics may be simultaneously parametrised in such a manner that the associated second fundamental forms are conformally flat, that is,
\bela{E32}
  \mbox{\rm II}_{12} \sim ds_1^2 + ds_2^2,\quad \mbox{\rm II}_{23} \sim ds_2^2 + ds_3^2,
\ela
it is evident that the lines $s_1\pm s_2=\mbox{const}$ and $s_2\pm s_3=\mbox{const}$ are conjugate on the surfaces $s_3=\mbox{const}$ and $s_1=\mbox{const}$ respectively. Moreover, due to the orthogonality of the lines of curvature, (the proof of) Theorem \ref{bisection} implies that the lines of curvature bisect these conjugate lines. Hence, the conjugate lines are characteristic and are diagonal to the lines of curvature as illustrated in Figure \ref{pic_charell}. As in the case of the asymptotic lines on the confocal one-sheeted hyperboloids, the characteristic lines on the ellipsoids and two-sheeted hyperboloids form infinitesimal rhombi. The class of surfaces for which the lines of curvature are isothermal-conjugate has been analysed in great detail in \cite{Eisenhart1903,Young1917}.

\begin{figure}
  \centering
  \includegraphics[scale=0.13]{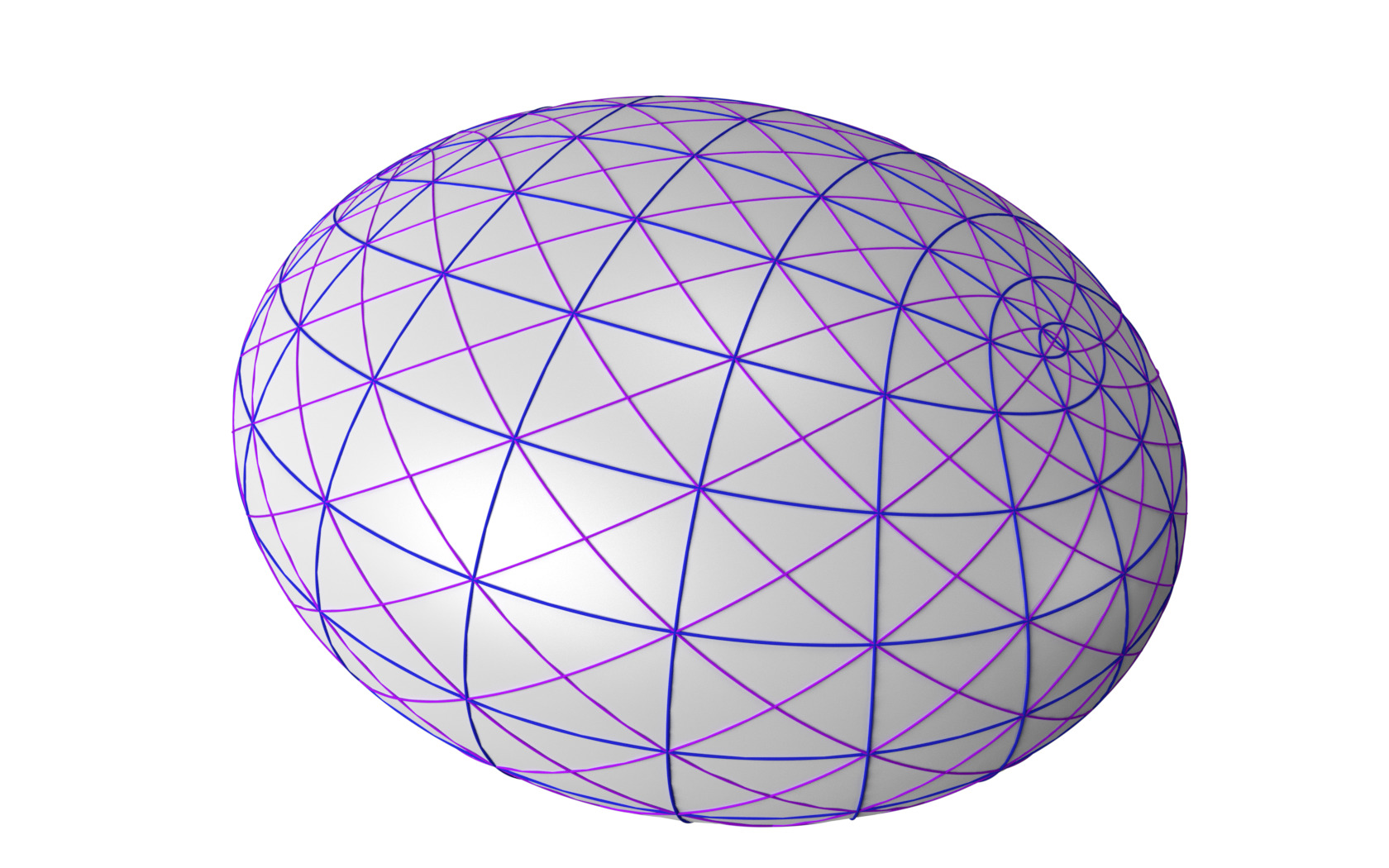}
  \includegraphics[scale=0.13]{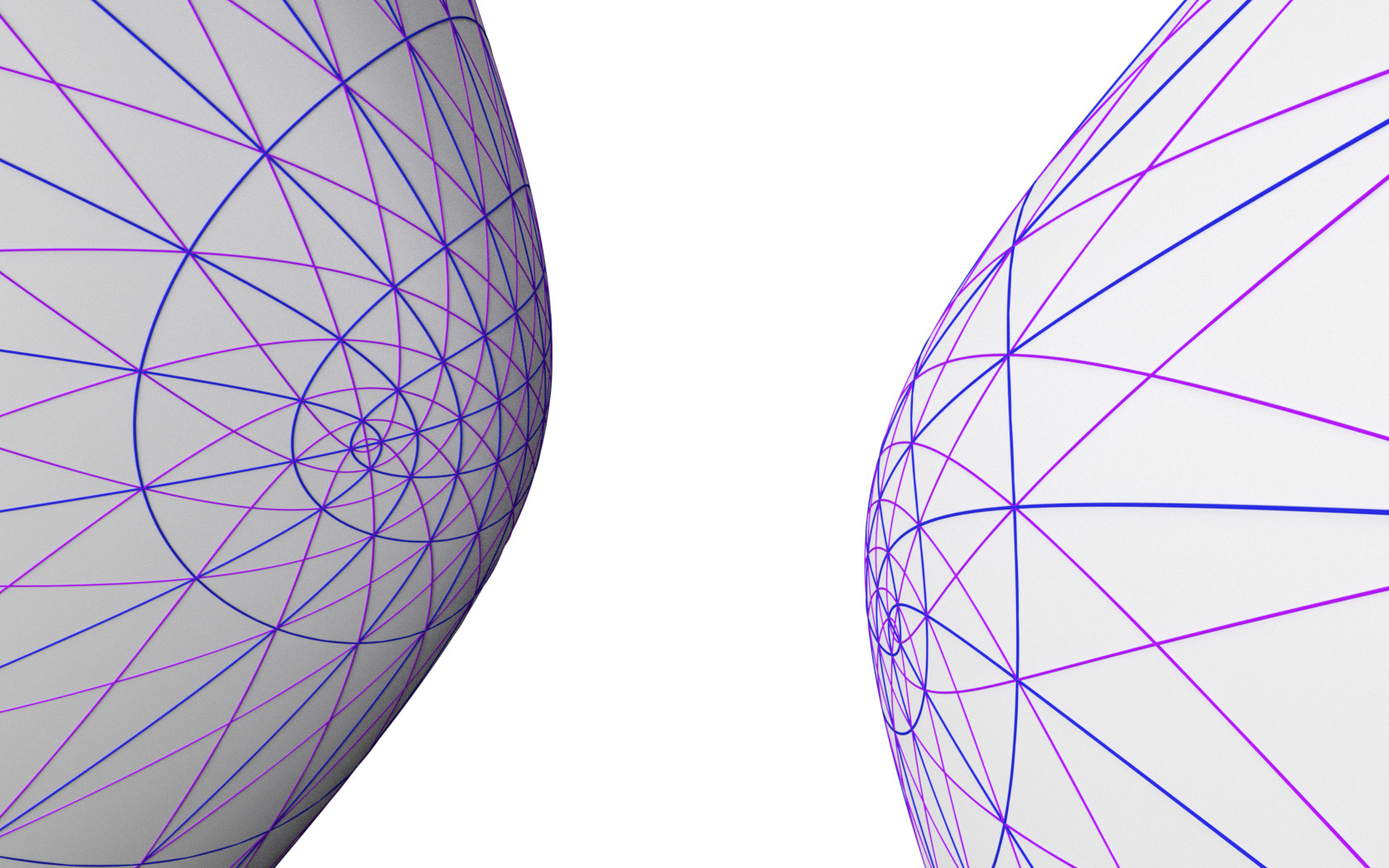}
  \caption{The mutually diagonal lines of curvature and characteristic lines on quadrics.}
  \label{pic_charell}
\end{figure}

The analogy between asymptotic lines and characteristic lines extends to all relevant properties discussed in Section 3 except for those related to the isometry along asymptotic lines. Indeed, one may directly verify that the characteristic lines on, for instance, the one-parameter family of ellipsoids are not isometrically mapped onto each other. Here, we focus on the relation between asymptotic lines, characteristic lines and lines of curvature on all quadrics of a confocal system. In this connection, the classical notion of octahedral webs (German: Achtflachgewebe \cite{B28b}) turns out to be of significance. Here, it is convenient to adopt the following characteristic property of octahedral webs as their definition.

\begin{definition}
Four distinct one-parameter families of surfaces $\{\Sigma_1\},\{\Sigma_2\},\{\Sigma_3\}$ and $\{\Sigma_4\}$ which simultaneously foliate (an open region of) a three-dimensional space form an {\em octahedral web of surfaces} if there exist another three one-parameter families of ``diagonal'' surfaces \mbox{$\{\Sigma_{12}=\Sigma_{34}\}$}, $\{\Sigma_{23}=\Sigma_{14}\}$ and $\{\Sigma_{13}=\Sigma_{24}\}$ such that any pair of surfaces $\Sigma_i,\Sigma_k$, $i<k$ meets in a surface $\Sigma_{ik}$.
\end{definition}

\begin{remark}
The above definition is based on the assumption that each triple of families of surfaces defines a coordinate system in (an open region of) the three-dimensional space. As a consequence, it implies that the surfaces of an octahedral web may be sampled in such a manner that they ``partition'' (an open region of) the three-dimensional space into ``curved'' octahedra and tetrahedra. In combinatorial terms, the collection of vertices of any partition of this type may be regarded as the vertices of an $A_3$-lattice (octahedral-tetrahedral honeycomb lattice).  
\end{remark}

In light of the notion of octahedral webs, the existence of the privileged parametrisation of the lines of curvature on systems of confocal quadrics in terms of the parameters $s_1$, $s_2$ and $s_3$ immediately gives rise to the following remarkable theorem.

\begin{theorem}\label{octahedralwebtheorem}
Let $(u_1,u_2,u_3)$ be arbitrary parameters of a system of confocal quadrics and $(s_1,s_2,s_3)$ be the associated privileged parameters as defined by Theorem \ref{parametrisation}. Then, the following properties hold.
\begin{itemize}
\item
The six congruences of asymptotic/characteristic lines on the confocal quadrics  ``commute'' with the three congruences of lines of curvature in the following sense. If one draws two asymptotic/characteristic lines on the quadrics $u_i=u_i^0$ and $u_k=u_k^0$ through a point on the line of curvature $(u_i=u_i^0,u_k=u_k^0)$ then the points of intersection of these lines with two lines of curvature $(u_i=u_i^0,u_l=u_l^0)$ and $(u_k=u_k^0,u_l=u_l^0)$ are linked by an asymptotic/characteristic line on the quadric $u_l=u_l^0$.
\item
The above triangular paths of closed asymptotic and characteristic lines belong to four triples of congruences of asymptotic/characteristic lines. The lines of any such triple of congruences  form a one-parameter family of surfaces. These distinct families of (ruled) surfaces form an octahedral web and are represented by
\bela{E33}
 \begin{split}
  s_1 + s_2 + s_3 &= \mbox{const},\quad s_1 + s_2 - s_3 = \mbox{const}\\
  s_1 - s_2 + s_3 &= \mbox{const},\quad s_1 - s_2 - s_3 = \mbox{const}
 \end{split}
\ela
and pairwise intersect in the asymptotic/characteristic lines of the corresponding families of quadrics. 
\end{itemize}
\end{theorem} 

The above theorem may be exploited to construct three-dimensional extensions of AC-grids.

\begin{corollary}
Any AC-grid on a one-sheeted hyperboloid (or its degeneration to an IC-net) may be extended to a three-dimensional grid consisting of asymptotic and characteristic lines which represent the lines of intersection of pairs of surfaces of an octahedral web. The two-dimensional subgrid on any quadric of the associated confocal system constitutes an AC-grid or an analogous configuration of characteristic lines. An octahedron, the curved faces of which are bounded by asymptotic/characteristic lines, is displayed in Figure \ref{pic_3dgrid}.

\begin{figure}
  \centering
  \includegraphics[scale=0.16]{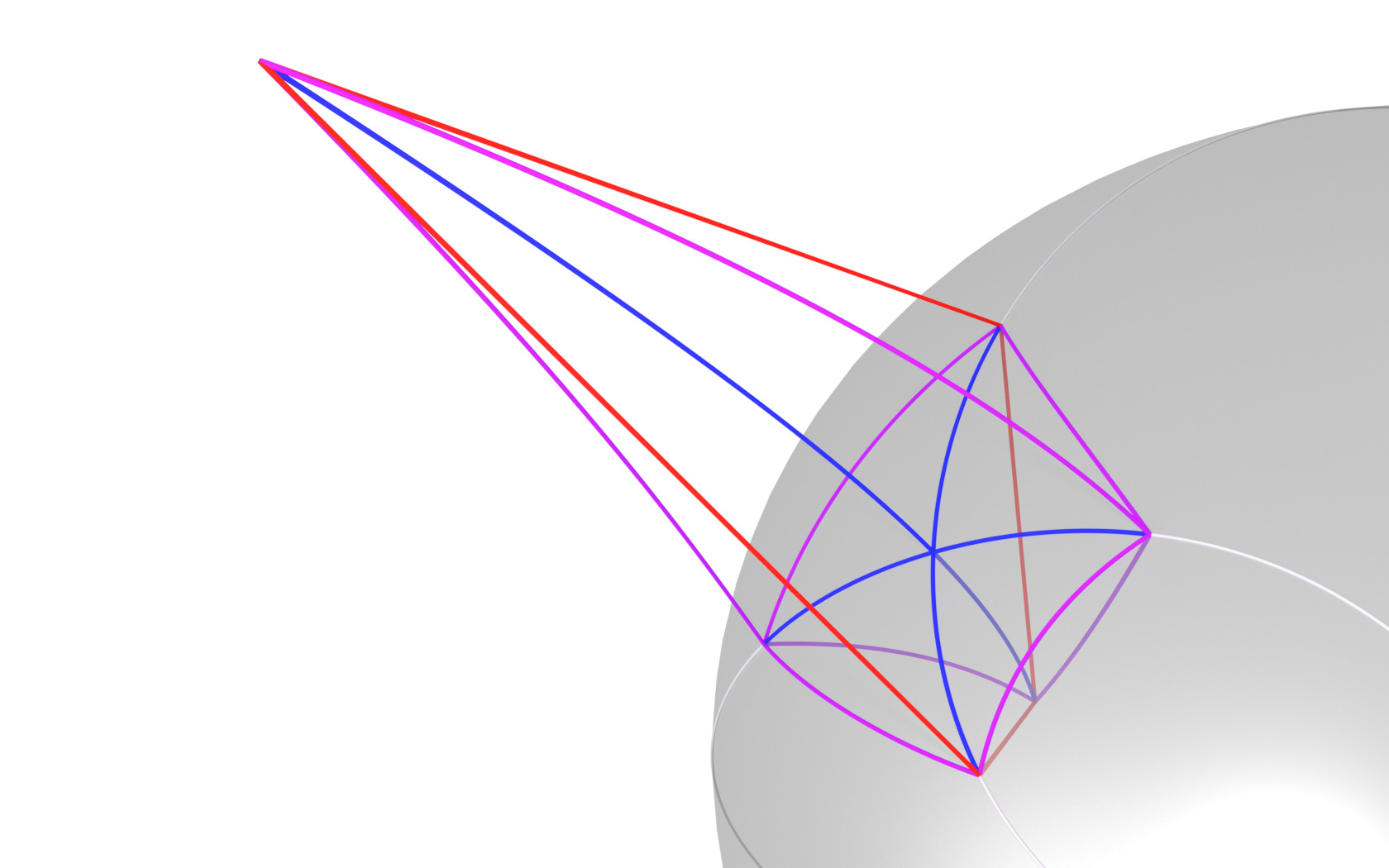}
  \caption{An octahedron of a three-dimensional grid of asymptotic lines and characteristic lines on confocal quadrics. The curves connecting opposite vertices of the octahedron are lines of curvature. The lines of curvature and characteristic/asymptotic lines form pairs of mutually diagonal nets on the
quadrics. Here, this relation is depicted on an ellipsoid.}
  \label{pic_3dgrid}
\end{figure}
\end{corollary}


\section{Conical octahedral grids}

The formal analogy between asymptotic and characteristic lines on surfaces of negative and positive Gaussian curvature respectively may now, in the first instance, be exploited to construct spatial extensions of (checkerboard) IC-nets composed of planes which form three-dimensional octahedral grids. Thus, if we apply the Jacobi imaginary transformations \cite{NIST}
\bela{E34}
  \jac{sn}(u_2,k_2) = -i\jac{sc}(iu_2,k_1),\quad \jac{cn}(u_2,k_2) = \jac{nc}(iu_2,k_1),\quad
  \jac{dn}(u_2,k_2) = \jac{dc}(iu_2,k_1),
\ela
where $k_1^2 + k_2^2 = 1$, then the substitutions $iu_2\rightarrow u_2$ and $iz\rightarrow z$ in the parametrisation \eqref{E14} of systems of confocal quadrics in a three-dimensional Euclidean space leads to the parametrisation
\bela{E35}
  \left(\bear{c}x\\ y\\ z\ear\right) = \left(\bear{c}
     \jac{sn}(s_1,k)\jac{dc}(s_2,k)\jac{ns}(s_3,k)\\ 
     \jac{cn}(s_1,k)\jac{nc}(s_2,k)\jac{ds}(s_3,k)\\
     \jac{dn}(s_1,k)\jac{sc}(s_2,k)\jac{cs}(s_3,k)
 \ear\right)
\ela
of systems of confocal quadrics
\bela{E36}
  \frac{x^2}{\mu+a} + \frac{y^2}{\mu+b} - \frac{z^2}{\mu+c} = 1
\ela
in a three-dimensional Minkowski space $\R^{2,1}$ equipped with the diagonal metric $\operatorname{diag}(1,1,-1)$. Here, for convenience, we have scaled the quadrics in such a manner that $a-c=1$ so that $k^2=a-b< 1$ and, hence, $-a< -b <-c$. The three families $s_i=\mbox{const}$ of quadrics are composed of one-sheeted hyperboloids with corresponding second fundamental forms
\bela{E37}
  \mbox{II}_{ij}\sim ds_i^2 - ds_j^2,\quad i<j
\ela
and the curves $s_i\pm s_j=\mbox{const}$ being the (straight) asymptotic lines on the hyperboloids \mbox{$s_l=\mbox{const}$}. The connection between the parameter $\mu$ and the parameters $s_i$ is given by
\bela{E37a}
  \mu = -a + (a-b)\jac{sn}^2(s_1,k),\quad \mu = -b + (b-c)\jac{nc}^2(s_2,k),\quad \mu = -a + (a-c)\jac{ns}^2(s_3,k)
\ela
respectively. The latter reveals that the family of quadrics $s_1=\mbox{const}$ corresponds to the interval $-a\leq\mu\leq-b$, while the two families of quadrics $s_2=\mbox{const}$ and $s_3=\mbox{const}$ are associated with the same semi-infinite region $\mu \geq -c$. It is noted that the parametrisation \eqref{E35} captures the region of Minkowski space where, at each point, three confocal quadrics (one-sheeted hyperboloids) meet. The intersection of three confocal quadrics is depicted in Figure \ref{confocal_minkowski}.
\begin{figure}
  \centering
  \includegraphics[scale=0.173]{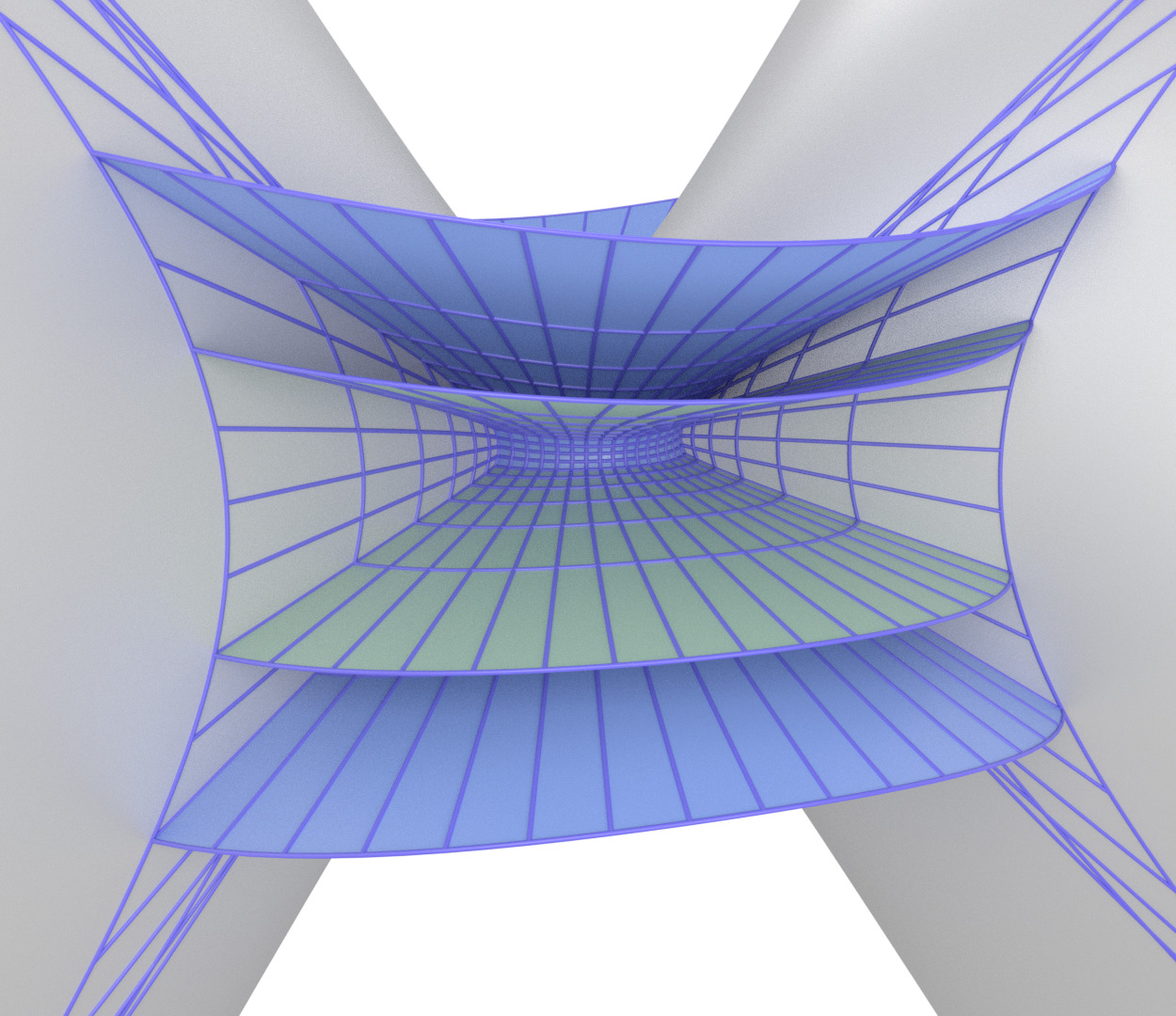}
  \includegraphics[scale=0.173]{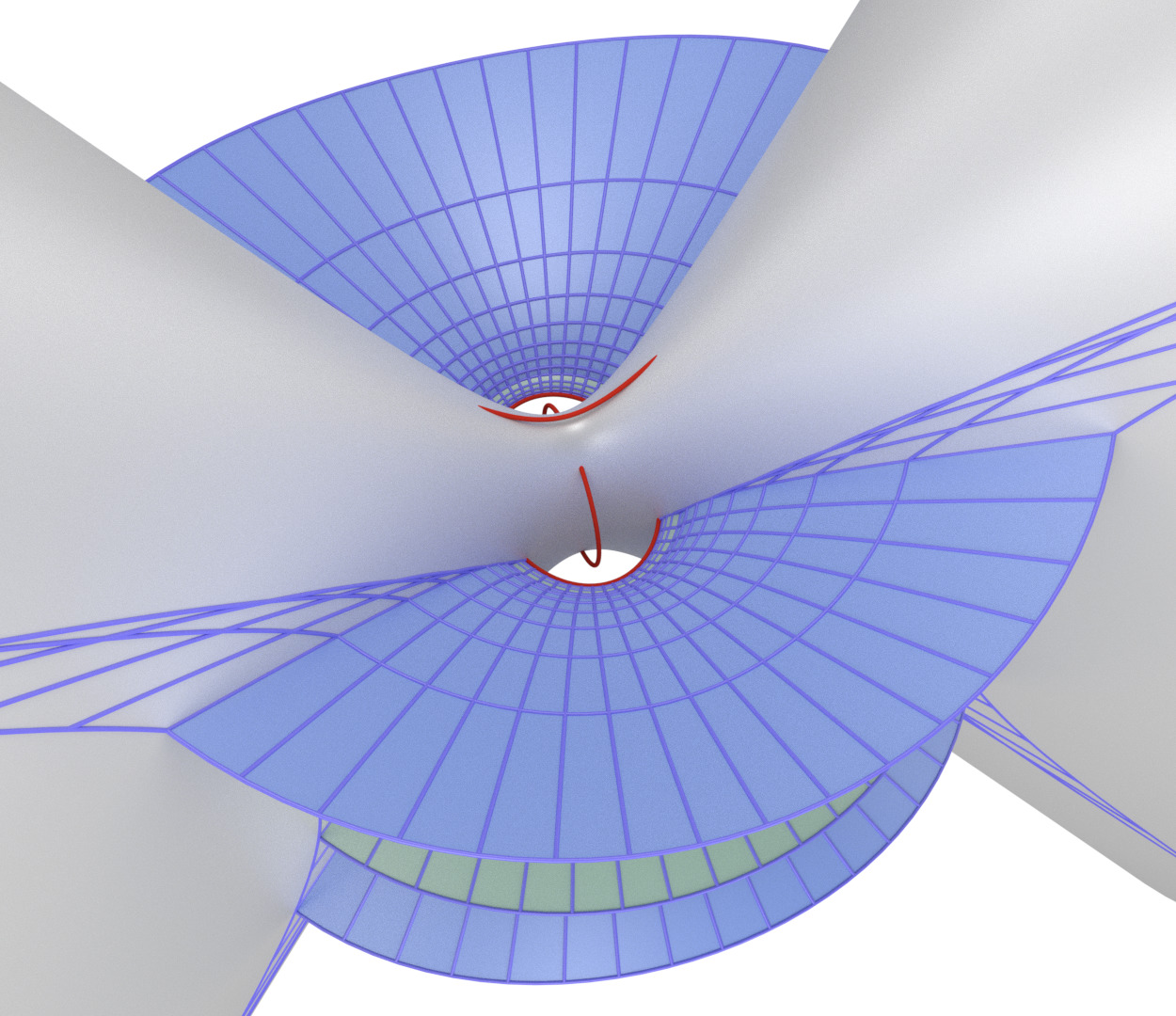}
  \caption{Three confocal one-sheeted hyperboloids in Minkowski space $\R^{2,1}$ and the associated triple of focal conics.}
  \label{confocal_minkowski}
\end{figure} 
In analogy with Theorem \ref{octahedralwebtheorem}, we therefore come to the following conclusion.

\begin{theorem}\label{planarweb}
The confocal coordinate system \eqref{E35} in Minkowski space has the following properties.
\begin{itemize}
\item The four families of surfaces
\bela{E38}
\begin{split}
  s_1 + s_2 + s_3 &= \lambda_1,\quad s_1 + s_2 - s_3 = \lambda_2\\
  s_1 - s_2 + s_3 &= \lambda_3,\quad s_1 - s_2 - s_3 = \lambda_4
 \end{split}
\ela
form an octahedral web and are composed of the planes
\bela{E39}
\begin{split}
  x\jac{cn}(\lambda_1,k) - y\jac{sn}(\lambda_1,k) + z 
  + \jac{dn}(\lambda_1,k) & = 0\\
  x\jac{cn}(\lambda_2,k) - y\jac{sn}(\lambda_2,k) + z 
  - \jac{dn}(\lambda_2,k) & = 0\\
- x\jac{cn}(\lambda_3,k) + y\jac{sn}(\lambda_3,k) + z 
  - \jac{dn}(\lambda_3,k) & = 0\\
- x\jac{cn}(\lambda_4,k) + y\jac{sn}(\lambda_4,k) + z 
  + \jac{dn}(\lambda_4,k) & = 0.
\end{split}
\ela

\item The planes of the octahedral web are tangent to the confocal system of quadrics \eqref{E36} and meet the plane $z=0$ at an angle of 45 degrees.

\item Any four planes related by $\lambda_1 + \lambda_4 = \lambda_2 + \lambda_3$ meet at a point which constitutes the vertex of a circular cone touching the four planes.

\item In the limit $s_2\rightarrow0$, the two families of asymptotic lines $s_1\pm s_3=\mbox{const}$ on the plane $z=0$, parametrised by
\bela{E40}
  \left(\bear{c}x\\ y\ear\right) = \left(\bear{c}
     \jac{sn}(s_1,k)\jac{ns}(s_3,k)\\ 
     \jac{cn}(s_1,k)\jac{ds}(s_3,k)
 \ear\right),
\ela
are tangent to the ellipse
\bela{E41}
  x^2 + \frac{y^2}{1-k^2} = 1.
\ela
\end{itemize}
\end{theorem}

\begin{proof}
Since the curves $s_i\pm s_j = \mbox{const}$ on the surfaces $s_l=\mbox{const}$ constitute straight lines, each surface $\lambda_m=\mbox{const}$ contains three families of straight lines and, hence, must be a plane. These planes are given by \eqref{E39} which may be verified by insertion of the parametrisation \eqref{E35} into \eqref{E39} and evaluation modulo $\lambda_m=\mbox{const}$ on use of the addition theorems for Jacobi elliptic functions~\cite{NIST}. Since the unit normals to the planes are of the form $(\pm\jac{cn}(\lambda_m,k),\mp\jac{sn}(\lambda_m,k),1)/\sqrt{2}$, their scalar product with $(0,0,1)$ is $1/\sqrt{2}$ so that the angle between the planes and the plane $z=0$ is indeed $\pi/2$. The latter guarantees the existence of a circular cone touching any four planes of the octahedral web which meet at a point. In fact, it is evident that the axes of the cones are perpendicular to the plane $z=0$ and their opening angle is 90 degrees. Moreover, for any given $\lambda_m$, the equations \eqref{E38} have a common solution if and only if $\lambda_1+\lambda_4 = \lambda_2+\lambda_3$, in which case the common solution parametrises the point at which the four corresponding planes meet.

The tangent plane of a quadric
\bela{E42}
  \frac{x^2}{A} + \frac{y^2}{B} - \frac{z^2}{C} = 1
\ela
at some point $(x_0,y_0,z_0)$ is given by
\bela{E43}
  \frac{x_0x}{A} + \frac{y_0y}{B} - \frac{z_0z}{C} = 1.
\ela
If the latter is matched with, for instance, the planes \eqref{E39}$_1$ then it is required that
\bela{E44}
  x_0 = -A\jac{cd}(\lambda_1,k),\quad y_0 =B\jac{sd}(\lambda_1,k),\quad z_0 = C\jac{nd}(\lambda_1,k). 
\ela
The condition for $(x_0,y_0,z_0)$ to be on the quadric \eqref{E42} for all $\lambda_1$ now leads to the two constraints
\bela{E45}
  A - B = k^2,\quad A - C = 1
\ela
which encapsulate the confocal system of quadrics \eqref{E36}. The same line of arguments also applies to the remaining three families of planes \eqref{E39}$_{2,3,4}$ and, hence, the planes of the octahedral web touch all members of the confocal system of quadrics.

Finally, the planar degenerate one-sheeted hyperboloids $s_2=0$ of the Euclidean and Min\-kowskian systems of confocal quadrics coincide so that the ``asymptotic lines'' touch the focal conic \eqref{E41} on the plane $z=0$ as stated in Theorem \ref{deformation}. This concludes the proof of the theorem.
\end{proof}

\begin{remark}
Octahedral webs composed of planes have been investigated in great detail by Sauer \cite{Sauer1925}. In particular, Sauer has demonstrated in a projective setting that all octahedral webs of this type may be parametrised in terms of $\sigma$-functions \cite{NIST}. This is consistent with the occurrence of the Jacobi elliptic functions exploited in the preceding. 
\end{remark}

\subsection{Octahedral grids of planes related to confocal quadrics}

Appropriate sampling of the planes of the above octahedral webs ({\em octahedral grids}) now leads to the above-mentioned connection with {\ICe}s.

\begin{corollary} \label{gridofplanes}
The octahedral grid of planes composed of the subset of the planes \eqref{E38} given by
\bela{E50}
 \lambda_i = \lambda_i^0 + \delta n_i,\quad n_i\in\Z,
\ela
where the ``grid spacing parameter'' $\delta$ is arbitrary and the constants $\lambda_i^0$ are constrained by
\bela{E51}
  \lambda_1^0 + \lambda_4^0 = \lambda_2^0 + \lambda_3^0,
\ela
locally partitions the three-dimensional ambient space into octahedra and tetrahedra. A partition of (part of) the three-dimensional ambient space into a finite number of (non-overlapping) octahedra and tetrahedra (cf.\ Figure \ref{octahedralelliptic}) is obtained by making the choice 
\bela{E52}
 \lambda_i^0=0,\quad \delta=4\frac{\mathsf{K}(k)}{N},\quad N\in\N.
\ela
In this case, the planes of the octahedral grid pairwise meet in lines on the plane $z=0$ which form an \ICe.
\end{corollary}

\begin{figure}
  \centering
  \includegraphics[scale=0.173]{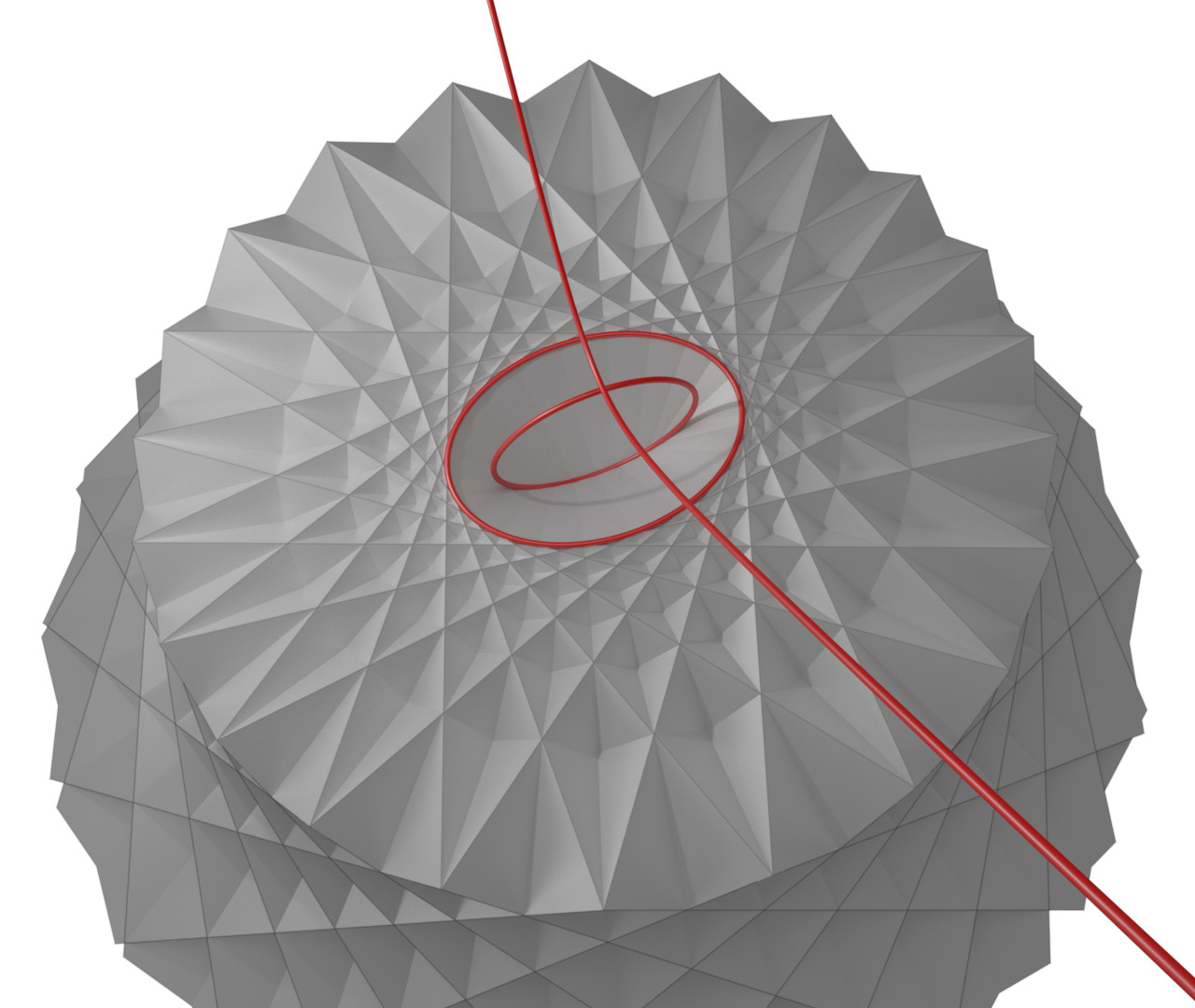}
  \includegraphics[scale=0.173]{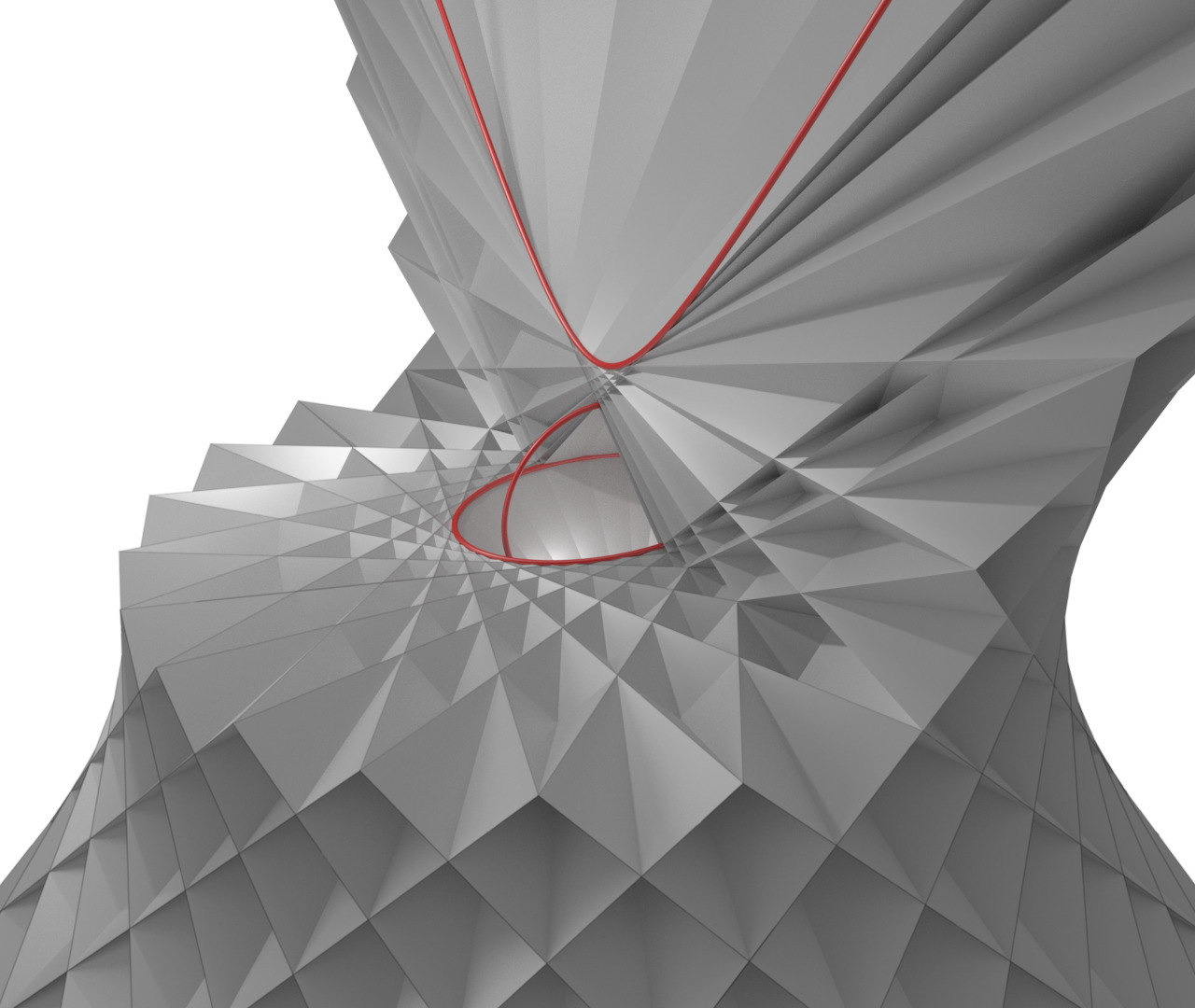}
\caption{Clippings of a three-dimensional octahedral extension of an \ICe and the triple of focal conics of the corresponding system of confocal quadrics. It is observed that each plane touches all focal conics.}
\label{octahedralelliptic}
\end{figure}

\begin{remark}
The octahedral grids constructed via confocal quadrics in Minkowski space turn out to constitute explicit examples of octahedral grids of planes discovered by Sauer via sampling octahedral webs \cite{Sauer1925} in a projective setting.  Moreover, the samplings of the asymptotic lines on the planes of the octahedral grids form hexagonal grids of lines, that is, samplings of classical hexagonal webs composed of lines \cite{B28}. The latter have been examined in great detail in \cite{GrafSauer1924}. 
\end{remark}

It turns out that any IC-net may be extended to an octahedral grid of planes which meet its lines pairwise at an angle of 45 degrees with the plane in which the IC-net is embedded. In fact, B\"ohm \cite{B} exploited this fact in his constructive proof of the existence of IC-nets. The key property of IC-nets on which its spatial extension is based is the fact that the existence of the circles inscribed in the quadrilaterals formed by the lines $\ell_n,\ell_{n+1}$ and $m_{n'},m_{n'+1}$ guarantees that the lines $\ell_n,\ell_{n+k}$ and $m_{n'},m_{n'+k}$ circumscribe circles for any $k\geq 1$. Hence, any quadruple of lines 
$\ell_n,\ell_{n+k}$, $m_{n'},m_{n'+k}$ may be regarded as the intersection of four planes with the plane $z=0$ which are tangent to a circular cone of opening angle $\pi/2$, its cross-section $z=0$ being the incircle of the corresponding quadrilateral. The collection of planes associated with all circles admitted by an IC-net then constitutes an octahedral grid of planes with the vertices of the circular cones being the vertices of the octahedra. 

In the case of {\ICe}s, Corollary \ref{gridofplanes} provides an explicit parametrisation of spatially extended {\ICe}s. Extensions of {\ICh}s are obtained by allowing the modulus $k$ in the parametrisation \eqref{E35} to be greater than 1. Hence, in this connection, it is natural to employ the identities \cite{NIST}
\bela{E53}
  \jac{sn}(s,1/k) = k\jac{sn}(s/k,k),\quad \jac{cn}(s,1/k) = \jac{dn}(s/k,k),\quad \jac{dn}(s,1/k) = \jac{cn}(s/k,k)
\ela
so that the transition $k\rightarrow1/k$ in the parametrisation \eqref{E35} leads to an alternative system of confocal quadrics \eqref{E36} in Minkowski space parametrised by
\bela{E54}
  \left(\bear{c}x\\ y\\ z\ear\right) = \left(\bear{l}
     \jac{sn}(s_1,k)\jac{cd}(s_2,k)\jac{ns}(s_3,k)\\ 
     \jac{dn}(s_1,k)\jac{nd}(s_2,k)\jac{cs}(s_3,k)/k\\
     \jac{cn}(s_1,k)\jac{sd}(s_2,k)\jac{ds}(s_3,k)
 \ear\right).
\ela
As in the elliptic case, the quadrics have been scaled such that $a-c=1$ but, now,  $a-b=1/k^2>1$ and, hence, $-a<-c<-b$ so that the asymptotic lines on the degenerate one-sheeted hyperboloid $s_2=0$ are tangent to the hyperbola
\bela{E55}
  x^2 - \frac{k^2y^2}{1-k^2} = 1
\ela
on the plane $z=0$. The associated collection of planes is readily shown to be given by 
\bela{E56}
\begin{split}
  x\jac{dn}(\lambda_1,k) - yk\jac{sn}(\lambda_1,k) + z 
  + \jac{cn}(\lambda_1,k) & = 0\\
  x\jac{dn}(\lambda_2,k) - yk\jac{sn}(\lambda_2,k) + z 
  - \jac{cn}(\lambda_2,k) & = 0\\
- x\jac{dn}(\lambda_3,k) + yk\jac{sn}(\lambda_3,k) + z 
  - \jac{cn}(\lambda_3,k) & = 0\\
- x\jac{dn}(\lambda_4,k) + yk\jac{sn}(\lambda_4,k) + z 
  + \jac{cn}(\lambda_4,k) & = 0.
\end{split}
\ela
An example of spatially extended \ICh is displayed in Figure \ref{octahedralhyperbolic}.
\begin{figure}
  \centering
  \includegraphics[scale=0.17]{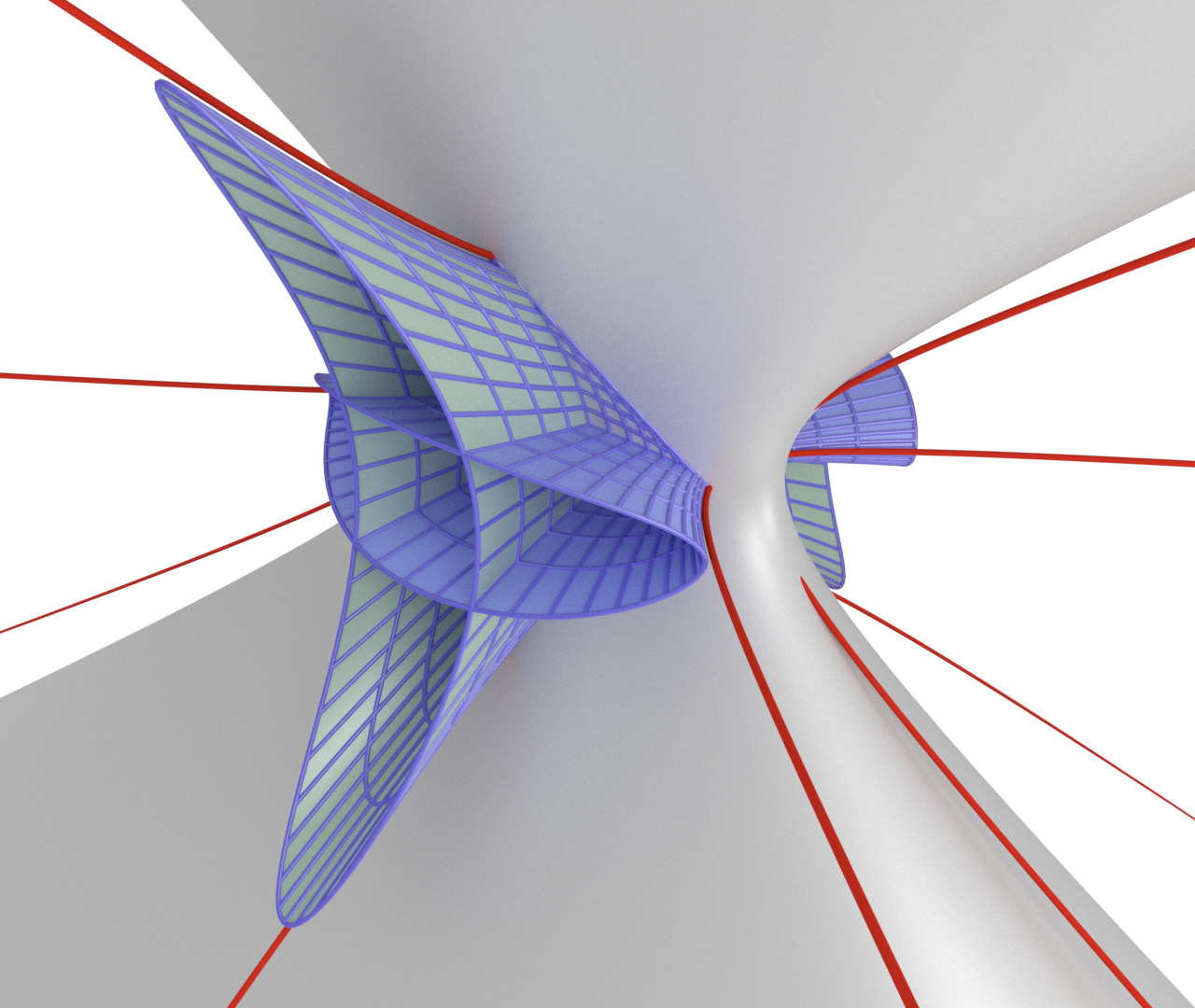}
  \includegraphics[scale=0.17]{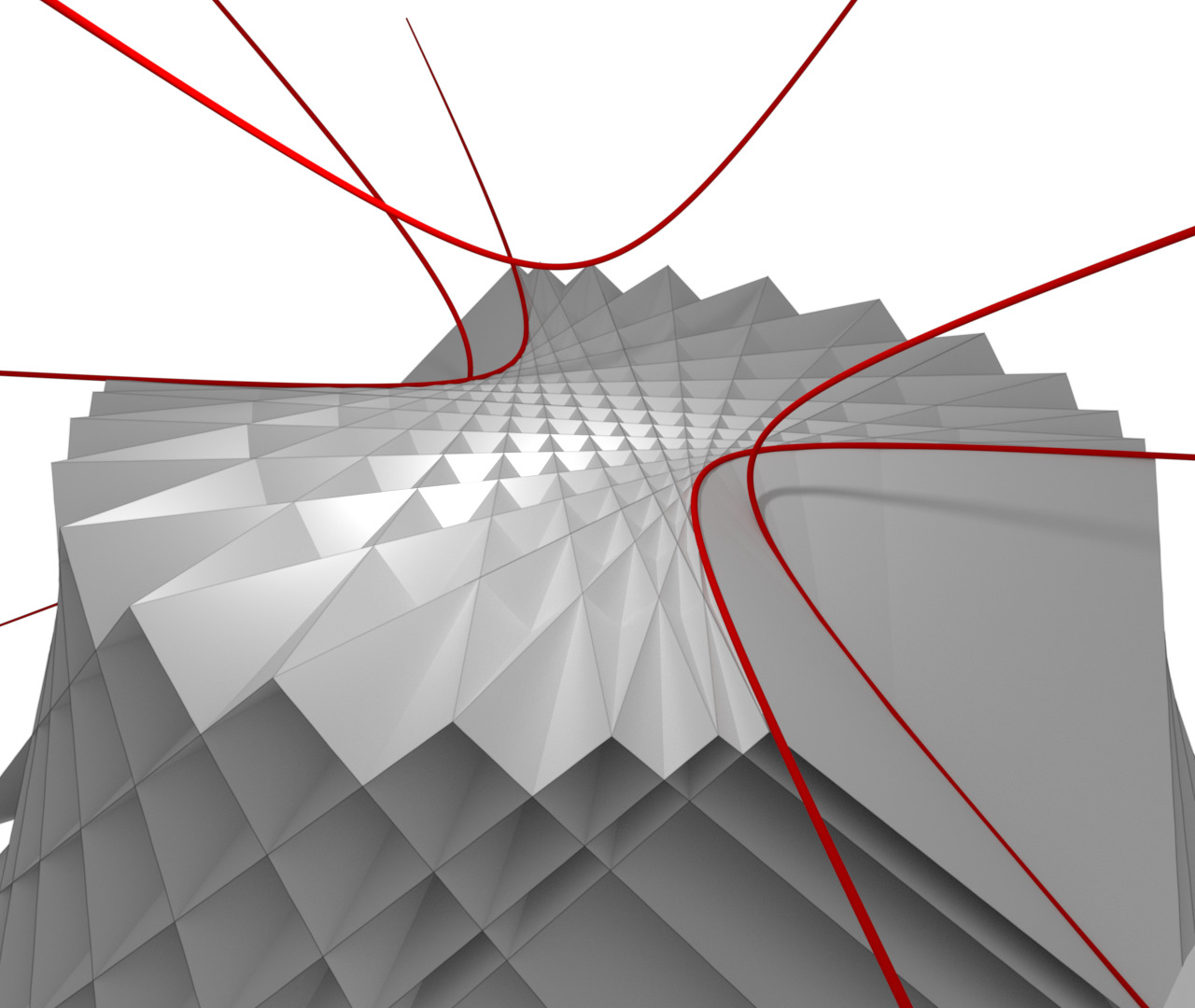}
  \caption{Left: Three confocal one-sheeted hyperboloids in Minkowski space $\R^{2,1}$ and the associated triple of focal conics. Right: A clipping of a three-dimensional octahedral extension of an \ICh.}
  \label{octahedralhyperbolic}
\end{figure} 

The connection with {\em checkerboard IC-nets}, that is, planar configurations of oriented lines and circles of the combinatorics of a checkerboard \cite{BST19}, is now obtained as follows. We first recall that any IC-net may be regarded as a degenerate checkerboard IC-net, that is, any line of an IC-net may be interpreted as two coinciding lines of opposite orientation. Since in the cyclographic model of planar Laguerre geometry, oriented lines on the plane $z=0$ are represented by planes which intersect the plane $z=0$ at an angle of 45 degrees (see, e.g., \cite{BS}), the three-dimensional extension of an IC-net is nothing but its representation in the cyclographic model. In fact, the parametrisation of the lines \eqref{E39}$|_{z=0}$ and \eqref{E56}$|_{z=0}$ is consistent with the Laguerre geometric parametrisation of {\ICe}s and {\ICh}s respectively presented in \cite{BST19}.

The degeneration of a checkerboard IC-net to an IC-net may be illustrated by vertically moving the plane of intersection $z=0$ which generates an IC-net from an octahedral grid of planes of the type \eqref{E50}, \eqref{E51}, \eqref{E52}. Indeed, if we consider the planes $z=\epsilon$ and examine the transition from $\epsilon=0$ to $\epsilon\neq0$ then we see that each line of the IC-net on the plane $z=0$ splits into two lines (see Figure \ref{pic_checkerclipping}) and the points of intersection of the lines which are the vertices of circular cones become circles.
\begin{figure}
  \centering
  \includegraphics[scale=0.17]{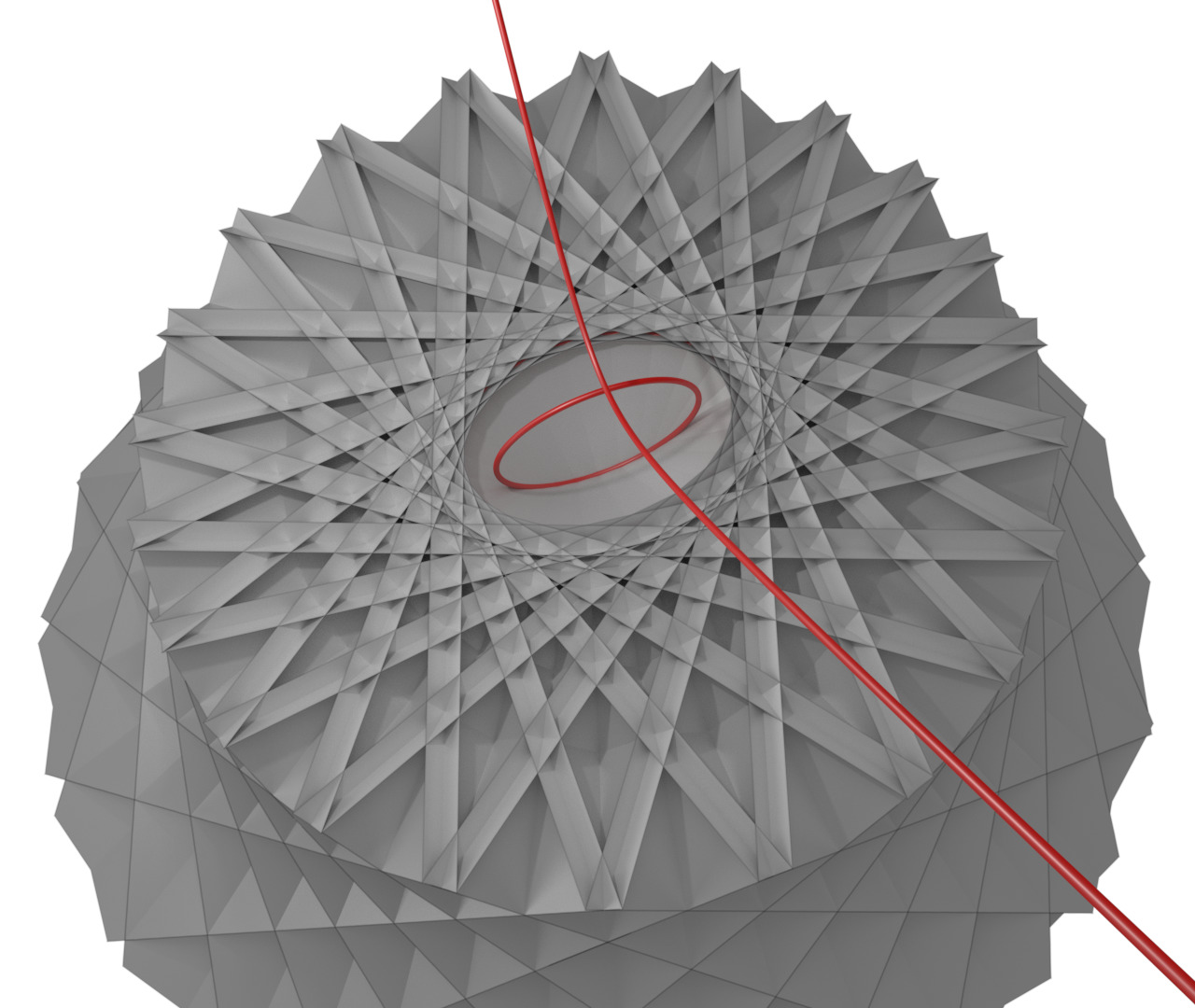}
  \caption{A three-dimensional octahedral extension of an \ICe cut off at $z=\epsilon\neq0$ so that the associated checkerboard IC-net is visible.}
  \label{pic_checkerclipping}
\end{figure} 
In this manner, one generates a one-parameter family of checkerboard IC-nets since only every other quadrilateral formed by the lines is associated with an incircle. Moreover, the lines and circles may be consistently oriented so that one obtains, by definition \cite{AB,BST19}, a one-parameter family of checkerboard IC-nets with the checkerboard IC-net degenerating to an IC-net for $\epsilon=0$. In general, that is, without imposition of the constraints \eqref{E52}, it is evident that the following holds. 

\begin{corollary}
The (suitably oriented) lines of intersection of an octahedral grid of planes of the type \eqref{E50}, \eqref{E51} with any plane $z=\mbox{const}$ form a checkerboard IC-net.
\end{corollary}

\subsection{Conical octahedral grids of planes}

B\"ohm \cite{B} has shown that the requirement that the octahedra of an octahedral grid of planes or the hexahedra of a grid of planes of $\Z^3$ combinatorics circumscribe {\em distinct} spheres  is too limiting to obtain nontrivial three-dimensional grids of IC-net-type configurations. However, in this section, we demonstrate that it is the weaker condition of touching cones of {\em a priori} arbitrary symmetry axis and opening angle that provides enough degrees of freedom to generalise the spatial extensions of IC-nets discussed in the preceding. To this end, we now equip the planes of an octaedral grid with an orientation so that the following definition is natural.

\begin{definition}
The point of intersection of four concurrent oriented planes is termed {\em conical} if it constitutes the vertex of a circular cone which is in oriented contact with the four planes. An {\em oriented octahedron} is an octahedron which is generated by the intersection of 8 oriented planes.
An oriented octahedron is called {\em conical} if all vertices of the octahedron are conical. A {\em conical octahedral grid} is an octahedral grid of oriented planes for which all octahedra are conical.
\end{definition}

\begin{remark}
The above notion of conical octahedral grids is reminiscent of the notion of {\em conical nets} introduced in \cite{LPWYW06} which has been used extensively in discrete differential geometry \cite{PW08,BS}. In particular, conical nets constitute a canonical integrability-preserving discretisation of lines of curvature on surfaces admitting a remarkable offset property which, in turn, finds application in architectural design \cite{LPWYW06}. A conical net (of $\Z^2$ combinatorics) is defined as a polyhedral surface composed of planar quadrilaterals such that the four quadrilaterals attached to any vertex touch a circular cone.
\end{remark}

An oriented plane in a three-dimensional Euclidean space may be labelled by a quadruple $(v_1,v_2,v_3,d)$, where $(v_1,v_2,v_3)$ is a unit normal, such that its defining equation is given by
\bela{E57}
  v_1x + v_2 y + v_3 z = d.
\ela
Accordingly, there exists a one-to-one correspondence between oriented planes and the {\em Blaschke cylinder}  (see, e.g., \cite{BS})
\bela{E58}
  \mathcal{B} = \{(v_1,v_2,v_3,d)\in\R^4 : v_1^2 + v_2^2 + v_3^2 = 1\}.
\ela
It is recalled that any planar section of the Blaschke cylinder $\mathcal{B}$ corresponds to the set of planes which are in oriented contact with a circular cone, while the intersection of $\mathcal{B}$ with any three-dimensional affine subspace encodes the oriented tangent planes of a sphere. In particular, four oriented planes are in oriented contact with a circular cone if and only if the corresponding points in $\mathcal{B}$ are coplanar. Thus,  a conical octahedron corresponds to a cuboid, that is, a combinatorial cube with planar faces, in $\mathcal{B}$. This correspondence gives rise to the following theorem.

\begin{theorem}\label{completion}
If the three vertices of one face of an oriented octahedron are conical then the remaining three vertices (and the octahedron) are conical modulo a change of orientation of the plane corresponding to the ``opposite'' face.
\end{theorem}

\begin{proof}
We consider the oriented plane corresponding to the privileged face and the 6 oriented planes which are combinatorially linked (via an edge or a vertex of the face) to this plane. These 7 planes correspond to 7 vertices of an incomplete cuboid in $\mathcal{B}$ since the three vertices of the face are conical. The eighth vertex of the cuboid is uniquely determined via the planarity of its faces. Now, a general theorem states \cite{BS} that if 7 vertices of a cuboid lie on a quadric then so does the eighth vertex. Accordingly, the octahedron associated with the cuboid constructed in this manner is conical. Moreover, since 7 faces of an octahedron uniquely determine the eighth, the two octahedra coincide and the proof is complete.
\end{proof}

\begin{remark} \label{janputhisfootdown}
The proof of the above theorem gives rise to the following constructive statement. Given 7 planes of an incomplete oriented octahedron such that the three vertices in which quadruples of planes meet are conical, there exists a unique eighth oriented plane which completes the configuration to a conical octahedron. 
\end{remark}

In order to present a constructive proof of the existence of conical octahedral grids which are not necessarily derived from (checkerboard) IC-nets, it is convenient to  adopt the following definition.

\begin{definition}
Two oriented octahedra are {\em in oriented contact along an edge} if the two sets of six oriented face planes of the two octahedra incident with the vertices of a common edge coincide (cf.\ Figure \ref{orientedcontact}). Several oriented octahedra are {\em in oriented contact} if all pairs of octahedra which have an edge in common are in oriented contact along the edge.
\end{definition}

\begin{figure}
  \centering
  \includegraphics[scale=0.12]{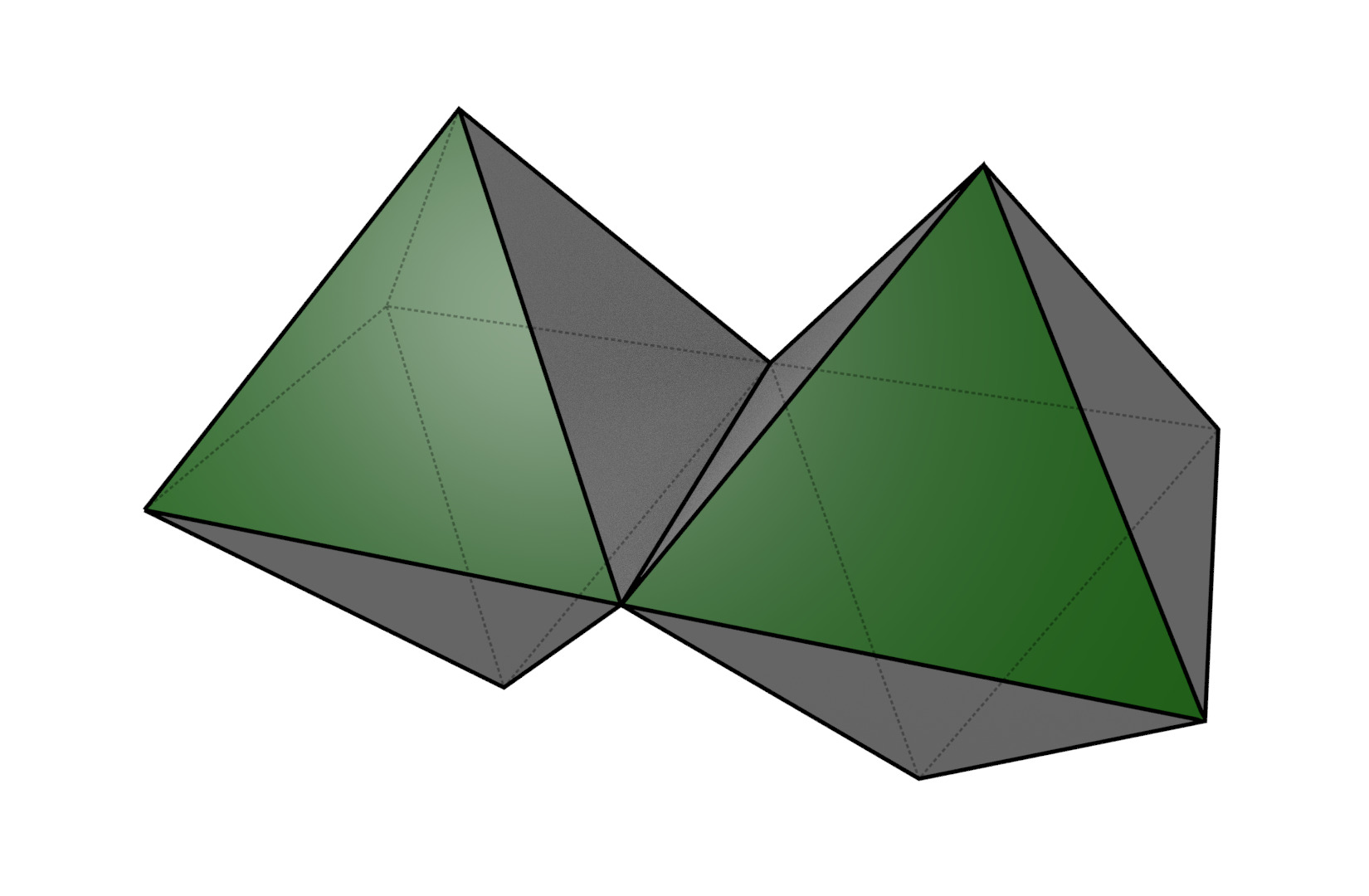}
  \caption{A pair of octahedra which are in oriented contact along an edge,
    that is, the six pairs of corresponding faces incident with a vertex of the common edge (e.g.\ the two highlighted faces) have coinciding oriented face planes.}
\label{orientedcontact}
\end{figure} 

Since a conical octahedron corresponds to a cuboid in the Blaschke cylinder and any cuboid is contained in a three-dimensional subspace, we first observe that the 8 planes of any conical octahedron are in oriented contact with a (possibly ``degenerate'') sphere. In fact, this sphere is the same for all octahedra of a conical octahedral grid.

\begin{theorem}
The planes of a conical octahedral grid are in oriented contact with a common (possibly degenerate) sphere.
\end{theorem}

\begin{proof}
We consider any two conical octahedra which are in oriented contact along a common edge. Each of the two sets of oriented planes of the octahedra is in oriented contact with a sphere. These two spheres must coincide since the two octahedra have 6 oriented planes in common.
\end{proof}

\begin{remark}
It is observed that, in the case of octahedral grids of planes derived from the systems of confocal quadrics in Minkowski space \eqref{E35}, \eqref{E36} and encapsulated in \eqref{E50}, \eqref{E51}, the above-mentioned common sphere is indeed degenerate.  
\end{remark}

The proof of the following theorem now provides a constructive method of generating conical octahedral grids.

\begin{theorem}\label{construction}
A conical octahedral grid is uniquely determined by two conical octahedra which are in oriented contact along an edge.
\end{theorem}

\begin{remark}
Pairs of conical octahedra which are in oriented contact are easily obtained. One first constructs a conical octahedron by iteratively prescribing and adding as required 7 oriented planes and associated circular cones, and subsequently applying Remark \ref{janputhisfootdown}. If one now selects an edge of this octahedron then six oriented planes of the second octahedron are determined. Accordingly, four vertices of the second octahedron are known. The two vertices of the common edge are conical and the other two vertices define two circular cones in oriented contact with the respective triple of concurrent oriented planes. If we now choose a seventh plane which is in oriented contact with one of these circular cones then according to Remark \ref{janputhisfootdown} there exists a unique eighth oriented plane such that the second octahedron is conical.
\end{remark}

\begin{proof}
The proof is a corollary of the analogous theorem for octahedral grids of planes. Thus, Sauer \cite{Sauer1925} has demonstrated that an octahedral grid of planes is uniquely determined by two octahedra which share an edge and have 6 planes in common. It is evident that if two such octahedra are prescribed then neighbouring octahedra are iteratively and uniquely determined by the existing planes. However, it is easy to see that this procedure is only consistent if a closing condition is satisfied, that is, it generates pairs of planes which are required to coincide and Sauer has employed a theorem due to Chasles to show that this is indeed the case. If we now start with two conical octahedra which are in oriented contact along an edge then the representation of conical octahedra as cuboids in the Blaschke cylinder implies that the property of an octahedron being conical propagates as octahedra are being iteratively added. It is emphasised that, in this case, Sauer's coincidence of pairs of non-oriented planes may be shown to translate into the coincidence of pairs of oriented planes as required.
\end{proof}

\begin{remark}
Since oriented planes and circular cones are preserved by Laguerre transformations (see, e.g., \cite{BS}), a conical octahedral grid constitutes  a Laguerre geometric object, that is, any conical octahedral grid is mapped to a multi-parameter class of conical octahedral grids by means of the group of Laguerre transformations. This applies, in particular, to spatially extended (checkerboard) IC-nets with the orientation of the planes being determined by the orientation of the relevant circular cones. Since, in this case, the axes of all cones are parallel, it is consistent to assume that all cones have the ``same'' orientation. 
\end{remark}

\begin{remark}
We conclude by observing that there exists a one-to-one correspondence between conical octahedral grids and {\em spherical checkerboard IC-nets} \cite{AB}. Indeed, if one translates the planes and associated cones of a conical octahedral grid in such a manner that the vertices of the cones coincide with the centre of the common sphere then the cones and planes 
intersect the sphere in circles and great circles respectively. Moreover, the orientation of the cones and planes naturally induces an orientation of the circles and great circles so that these circles are in oriented contact. Hence, by definition, one obtains a spherical checkerboard IC-net. Conversely, any spherical checkerboard IC-net may be related to a conical octahedral grid.
\end{remark}

\begin{acknowledgements}
This research was supported by the DFG Collaborative Research Center TRR 109 ``Discretization in Geometry and Dynamics''. W.K.S.\ was also supported by the Australian Research Council (DP1401000851). A.V.A.\ was also supported by the European Research Council (ERC) under the European Union's Horizon 2020 research and innovation programme (grant agreement No 78818 Alpha).
\end{acknowledgements}



\end{document}